\documentclass{amsart}

\usepackage{color}

\RequirePackage{amsmath}
\RequirePackage{amssymb}
\RequirePackage{amsxtra}
\RequirePackage{amsfonts}
\RequirePackage{latexsym}
\RequirePackage[mathscr]{euscript}
\RequirePackage{amscd}
\RequirePackage{amsthm}
\RequirePackage{xypic}
\usepackage{mathrsfs}
\xyoption{all}

\usepackage{dsfont}
\usepackage{amsmath,amssymb,amsthm,verbatim,amscd}

\providecommand{\bysame}{\leavevmode\hbox to3em{\hrulefill}\thinspace}

\def\bA{\mathbf A}

\def\C{\mathbb C}   
\def\bD{\mathbf D}

\def\Q{\mathbb{Q}}

\def\sR{\mathscr{R}}

\def\e{\mathbf{e}}

\def\Z{\mathbb{Z}}

\def\Qbar{\overline{\Q}}

\def\O{\mathcal O}

\def\Bbar{\overline{B}}

\def\unif{\varpi}

\def\rig{\mathrm{rig}}
\def\id{\mathrm{id}}

\def\rig{\mathrm{rig}}

\def\sm{\mathrm{sm}}

\def\cont{\mathrm{cont}}

\def\ord{\mathrm{ord}}

\def\new{\mathrm{new}}

\def\Sym{\mathrm{Sym}}
\def\Ext{\mathrm{Ext}}

\def\Lie{\mathrm{Lie} \, }
\def\Sen{\mathrm{Sen}}

\def\Iw{\mathrm{Iw}}

\def\Gal{\mathop{\mathrm{Gal}}\nolimits}
\def\Hom{\mathop{\mathrm{Hom}}\nolimits}

\def\Fit{\mathop{\mathrm{Fitt}}\nolimits}

\def\Tr{\mathop{\mathrm{Tr}}\nolimits}

\def\Ind{\mathop{\mathrm{Ind}}\nolimits}

\def\Con{\mathop{\mathcal C}\nolimits}

\def\rhobar{\overline{\rho}}

\def\hw{\mathrm{hw}}
\def\crys{\mathrm{crys}}
\def\dif{\mathrm{dif}}

\def\dR{\mathrm{dR}}

\def\ilim#1{\displaystyle \lim_{\longrightarrow \atop #1}}
\def\plim#1{\displaystyle \lim_{\longleftarrow \atop #1}}

\def\iso{\buildrel \sim \over \longrightarrow}

\def\GLQp{\GL_2(\Q_p)}

\def\Nbar{\overline{\mathrm{N}}}

\def\triv{\mathds{1}}

\def\ev{\mathrm{ev}}

\def\res{\mathop{\mathrm{res}}\nolimits}
\def\cores{\mathop{\mathrm{cor}}\nolimits}

\newfont{\cyrrlarge}{wncyr10}

\renewcommand{\O}{\mathcal O}

\newcommand{\Qinf}{\Q_\infty}
\newcommand{\Kinf}{K_\infty}
\newcommand{\Linf}{L_\infty}
\newcommand{\Zp}{\Z_p}
\newcommand{\Qp}{\Q_p}
\newcommand{\Qpbar}{\overline{\Q}_p}
\newcommand{\Qmpn}{\Q(\mu_{p^n})}

\renewcommand{\H}{\mathbb H}
\newcommand{\G}{\mathcal G}
\newcommand{\Ln}{\Lambda_n}

\newcommand{\Kn}{K_n}

\newcommand{\Knp}{K_{n,p}}
\newcommand{\Kmp}{K_{m,p}}

\newcommand{\hatKinfp}{\widehat{K}_{\infty,p}}
\newcommand{\Kinfp}{K_{\infty,p}}
\newcommand{\p}{\mathfrak p}
\newcommand{\Kinfw}{K_{\infty,w}}
\newcommand{\Knv}{K_{n,v}}
\newcommand{\Knvn}{K_{n,v_n}}

\renewcommand{\tt}{\tilde{\theta}}

\newcommand{\maps}{\rightarrow}
\newcommand{\lra}{\longrightarrow}

\newcommand{\MTnj}{\theta^{\an}_{n,j}(f)}
\newcommand{\MTnv}[1]{\theta^{\an}_{n,#1}(f)}
\newcommand{\MTnvv}[2]{\theta^{\an}_{n,#1}(#2)}

\newcommand{\MTvj}[1]{\theta^{\an}_{#1,j}(f)}
\newcommand{\tnj}{\theta_{n,j}}

\newcommand{\algn}{\theta^{\alg}_n(f)}
\newcommand{\algnj}{\theta^{\alg}_{n,j}(f)}

\newcommand{\algvj}[1]{\theta^{\alg}_{#1,j}(f)}

\newcommand{\Adualj}{A^*(-j)}
\newcommand{\Tdualj}{T^*(-j)}
\newcommand{\Tj}{T(1+j)}
\newcommand{\Xj}{X_j}
\newcommand{\Zj}{Z_j}
\newcommand{\Hj}{H_j}

\newcommand{\psmallmat}[4]{{\left( \begin{smallmatrix} #1 & #2 \\ #3 & #4 \end{smallmatrix} \right)}}

\DeclareMathOperator{\cond}{cond}

\DeclareMathOperator{\im}{image}
\DeclareMathOperator{\corank}{corank}
\DeclareMathOperator{\alg}{alg}
\DeclareMathOperator{\an}{an}

\DeclareMathOperator{\chr}{char}
\DeclareMathOperator{\tor}{tor}
\DeclareMathOperator{\Aut}{Aut}

\DeclareMathOperator{\SL}{SL}
\DeclareMathOperator{\Kato}{K}
\DeclareMathOperator{\GL}{GL}
\DeclareMathOperator{\Ann}{Ann}

\DeclareMathOperator{\Ltor}{\Lambda{\text-tor}}
\DeclareMathOperator{\cork}{\corank}
\DeclareMathOperator{\LP}{LP}

\theoremstyle{plain}

\newtheorem{theorem}[subsubsection]{Theorem}
\newtheorem*{theoremA}{Theorem A}
\newtheorem*{corB}{Corollary B}
\newtheorem{lemma}[subsubsection]{Lemma}

\newtheorem{cor}[subsubsection]{Corollary}
\newtheorem{conj}[subsubsection]{Conjecture}

\newtheorem{prop}[subsubsection]{Proposition}

\newtheorem{thm}[subsubsection]{Theorem}

\newcommand{\ve}{\varepsilon}

\theoremstyle{definition}

\newtheorem{remark}[subsubsection]{Remark}

\newcommand{\wt}{\widetilde}
\newcommand{\D}[1]{#1^\vee}

\begin{document}

\title
{Explicit reciprocity laws and
Iwasawa theory for modular forms}
\author{Matthew Emerton \and Robert Pollack \and Tom Weston}

\maketitle

\begin{abstract}
We prove that the Mazur-Tate elements of an eigenform $f$ sit inside the Fitting ideals of the corresponding dual Selmer groups along the cyclotomic $\Zp$-extension (up to scaling by a single constant).  Our method begins with a construction of local cohomology classes built via the $p$-adic local Langlands correspondence.  From these classes, we build {\it algebraic} analogues of the Mazur-Tate elements which we directly verify sit in the appropriate Fitting ideals.  Using Kato’s Euler system and explicit reciprocity laws, we prove that these algebraic elements divide the corresponding Mazur-Tate elements, implying our theorem.  
\end{abstract}

\section{Introduction}

Associated to a newform $f \in S_k(\Gamma_1(N),\psi,\Qpbar)$ is its collection of  Mazur-Tate elements $\MTnj \in \O[\G_n]$ with $n \geq 0$, $\O/\Z_p$ finite, and $\G_n \cong \Gal(\Q(\mu_{p^n})/\Q)$.  Here the superscript `an' is meant to remind us that these elements are analytically defined and are connected to $L$-values via\footnote{See Remarks~\ref{rmk:nonstandard} and~\ref{rmk:duality}
if you were expecting $\chi$ to be replaced by $\chi^{-1}$.}
$$
\chi(\MTnj) = \chi(-1) \cdot j! \cdot p^{nj} \cdot \tau(\chi^{-1}) \cdot \frac{L(f,\chi,j+1)}{(-2\pi i)^j \Omega_f^\pm}
$$
where $\chi$ is a primitive Dirichlet character of conductor $p^n$, $0 \leq j \leq k-2$, and $\Omega_f^\pm$ are the cohomological periods attached to $f$ as in \cite[Definition 2.1]{PW}.  

When $f$ corresponds to an elliptic curve~$E$, Mazur and Tate in \cite{MT} conjectured that $\MTnj$ 
belongs to the Fitting ideal of the corresponding (dual) Selmer group.  
In fact, they formulated their conjecture more generally, looking at Selmer groups over any abelian extension of $\Q$.  However, as the methods of this paper are rooted in Iwasawa theory, we restrict ourselves to the case of $p$-power cyclotomic fields and  aim to prove an analogous conjecture for modular forms of arbitrary weight.

To state our main theorem, we first set some notation.  Let $\rho_f$ denote the (cohomological) $p$-adic Galois representation attached to $f$ and let $V_f$ denote the underlying vector space.
Let $V_f^*$ denote the linear dual of $V_f$ and let $T_f^*$ denote a Galois stable lattice of $V_f^*$.  Set $A_f^* = V_f^*/T_f^*$ and let $H^1_f(\Q(\mu_{p^n}),A_f^*)$ denote the Block-Kato Selmer group attached to $A_f^*$ over $\Q(\mu_{p^n})$.

The following hypotheses will be needed for our main result:
\begin{equation*}
\label{irred}
\tag{Irred}
\rho_f |_{G_{\Qp}} \text{~is~irreducible;}
\end{equation*}
\begin{equation*}
\tag{No~pole~at~$s=j+1$}
\label{nopole}
\text{the~local~Euler~factor~of~} f \text{~at~} p  \text{~does~not~have~a~pole~at~} s=j+1;
\end{equation*}
\begin{align*}
\tag{Euler}
\label{Euler}
&\text{there~is~a~basis~of~} T_f^* \text{~such~that~if~we~identify~} \Aut(T_f^*) \text{~with} \GL_2(\O), \\
&\text{then~} \rho_f(\Gal(\Qbar/\Q(\mu_{p^\infty}))) \text{~contains~} 
\SL_2(\O).
\end{align*}

\subsection{Main results}
\begin{theoremA}
\label{thm:A}
If \eqref{irred}, \eqref{nopole}, and \eqref{Euler} hold, 
then there exists a non-zero constant $C \in \O$ (independent of $n$ and $j$) such for all $n \geq 0$
$$
C \cdot \MTnj \in \Fit_{\Ln}(H^1_f(\Q(\mu_{p^n}),A_f^*(-j))^\vee)
$$
where $\Ln := \O[\G_n]$ and $^\vee$ denotes Pontryagin dual.
\end{theoremA}

Theorem A should be viewed as a ``finite-level" (weak) main conjecture for $f$.  When $f$ is ordinary at $p$, one can deduce this theorem directly from the main conjecture under relatively mild hypotheses simply by using standard control theorems (see for instance \cite[Theorem 1.14]{Kim-Kurihara}).  
We note though that \eqref{irred} forces $f$ to be non-ordinary at $p$.  In particular, the standard control theorems no longer hold and there seems to be no direct connection between the main conjecture and Theorem A in this case.

Regarding the other hypotheses, \eqref{nopole} is needed to compute the inverse of $1-\varphi$ in the context of $(\varphi,\Gamma)$-modules.  We do not know what happens when this hypothesis is removed.  See Lemma \ref{lemma:nopole} for a complete classification of newforms which do not satisfy \eqref{nopole}. The hypothesis \eqref{Euler} is needed to invoke Kato's Euler system.

We will discuss the constant $C$ appearing in the statement of Theorem A later in the introduction,  but for now let us just say that it arises because of our inability to compare various normalizations of the periods attached to $f$.

We note that our theorem holds without any restriction on the power of $p$ that occurs in the level of our form:\ we have no crystalline hypotheses nor a finite slope hypothesis.  
Further,
Theorem A implies the so-called weak vanishing conjecture  (as in \cite[Proposition 3]{MT}).

\begin{corB}
Under the assumptions of Theorem A, the order of vanishing of $\MTnj$ at a character $\chi$ of $\G_n$ is greater than or equal to the dimension of the $\chi$-part of 
$H^1_f(\Q(\mu_{p^n}),A_f^*(-j))$.
\end{corB}

\subsection{Method of proof}
Our method of proof of Theorem A follows two steps.  First we define the notion of an {\it algebraic} $\theta$-element, $\algnj \in \Ln$, which is meant to be the algebraic counterpart of the corresponding Mazur-Tate element.  The construction of such elements dates back to Perrin-Riou's work \cite{PR90} (see also \cite{Pollack-JNT,PR-theta}).  As these elements are defined purely algebraically, we are able to check directly that $\algnj$ belongs to the Fitting ideal appearing in Theorem A. The second step then is to prove that $\algnj$ divides $C \cdot \MTnj$ in $\Ln$.  Unsurprisingly, our proof of this fact comes via Kato's Euler system.

Our construction of algebraic $\theta$-elements relies upon certain local cohomology classes $c_{n,j} \in H^1_f(\Qp(\mu_{p^n}),T_f(1+j)) / ({\rm torsion})$ which for $n \geq 1$ satisfy the three-term relation
$$
\cores^{n+1}_n(c_{n+1,j}) = a_p(f) c_{n,j} - \psi(p) p^{k-1} \res^n_{n-1}(c_{n-1,j}).
$$
With these classes in hand, we can directly define $\algn$.  Namely
$$
\algnj = \sum_{\sigma \in \G_n} \left\langle c_{n,j}^\sigma , w_n \right\rangle_n \sigma^{-1} \cdot \chr_{\Lambda} \H^2_P(T_f^*(-j))^\iota.
$$
Here  $(w_n)_n$ is a generator of $\H^1(T_f^*(-j))$ which is free of rank 1 over the Iwasawa algebra $\Lambda :=  \O[[\Gal(\Q(\mu_{p^\infty})/\Q)]]$ and $\langle \cdot, \cdot \rangle_n$ is the Tate local duality pairing at $p$.  (See section \ref{sec:euler} for the remaining notation.)

The relation between $\algnj$ and $\MTnj$ comes via Kato's Euler system as we prove that 
$$
C \cdot \MTnj =  \sum_{\sigma \in \G_n} \left\langle c_{n,j}^\sigma , z_{\Kato,n} \right\rangle_n \sigma^{-1} 
$$
where $(z_{\Kato,n})_n \in \H^1(T_f^*(-j))$ is Kato's Euler system and $C$ is some non-zero constant (independent of $n$ and $j$).  The divisibility of $C \cdot \MTnj$ by $\algnj$ then follows immediately from Kato's proof of the main conjecture without $p$-adic $L$-functions \cite[Theorem 12.5]{Kato}.

We note that in section \ref{sec:selmer}, when we introduce algebraic $\theta$-elements, we give a different definition from the one above.  This alternative definition is more easily connected with Fitting ideals, while the above definition more easily relates to $\MTnj$ via Kato's Euler system.  In section \ref{sec:compare}, we verify that these two definitions are the same.

\subsection{Local methods}
We are now left to describe our construction of the local classes $c_{n,j} \in H^1_f(\Qp(\mu_{p^n}),T_f(1+j))$.  (Here, for simplicity, we are assuming this cohomology group is $p$-torsion free.  For the general case, see section \ref{sec:normal}.)  We note that these local classes have appeared in literature in several places before.  For example, Kobayashi in \cite{Kobayashi} made use of such local points to define plus/minus Selmer groups of elliptic curves at supersingular primes.  His construction proceeded via the formal group attached to the elliptic curve.  As we are working with arbitrary weight modular forms, this geometric object is unavailable to us.  We instead give a purely local construction of these classes via the $p$-adic local Langlands correspondence.

Namely, let $V$ denote an irreducible de Rham representations of $G_{\Qp}$ defined over a finite extension $L/\Qp$; for instance, we could take $V = V_f |_{G_{\Qp}}$.  We will produce elements $c_{n,j} \in H^1_f(\Qp(\mu_{p^n}),T(1+j))$ for $T$ a $G_{\Q_p}$-stable lattice in $V$.  For the remainder of the introduction, we will assume that $j=0$ to ease the exposition and simply write $c_n$ for the class $c_{n,j}$.  To construct these classes, we first produce, for each $n \geq 0$, a functional on $H^1_{\Iw}(V^*) :=
\bigl( \varprojlim_n H^1(\Qp(\mu_{p^n}),T^*)\bigr) \otimes \Qp$ which:
\begin{enumerate}
\renewcommand{\theenumi}{\alph{enumi}}
\item factors through $H^1(\Qp(\mu_{p^n}),T^*)\otimes{\Q_p}$,
\item takes integral values on $H^1(\Qp(\mu_{p^n}),T^*)$, and
\item kills $H^1_f(\Qp(\mu_{p^n}),T^*)$.  
\end{enumerate}
By Tate local duality, such a functional corresponds to a class $c_{n} \in H^1_f(\Qp(\mu_{p^n}),T(1))$.

To produce this functional, we  identify $H^1_{\Iw}(V^*)$ with $\psi$-invariants of the $(\varphi,\Gamma)$-module $\bD(V^*)$.  The $p$-adic local Langlands correspondence gives rise to a Banach space representation $\pi(V)$ of $\GL_2(\Qp)$, and moreover, an embedding
$$
\bD(V^*)^{\psi=1} \hookrightarrow \pi(V)^*
$$
where the superscript $*$ on the right denotes the topological dual.  

Thus, to define  functionals on $\bD(V^*)^{\psi=1}$, it suffices to simply give natural elements of $\pi(V)$.    Such natural elements arise from the locally algebraic vectors in $\pi(V)$.  Namely, there is an embedding
$$
\pi_{\sm}(V) \otimes (\Sym^{k-2}(L^2))^* \hookrightarrow \pi(V)
$$
where $\pi_{\sm}(V)$ is the smooth $\GL_2(\Qp)$-representation associated to $V$ via the classical local Langlands correspondence.  Let $v_{\new}$ denote a newvector of $\pi_{\sm}(V)$ and let $v_{\hw}$ denote a highest weight vector of $(\Sym^{k-2}(L^2))^*$.  For $n\geq 0$, set
$$
d_n := \left(\begin{smallmatrix} 1 & 1 \\ 0 & 1 \end{smallmatrix}\right) \left(\begin{smallmatrix} p^n & 0 \\ 0 & 1 \end{smallmatrix}\right) (v_{\new} \otimes v_{\hw}) \in \pi(V)
$$
which, as described above, gives rise to a functional on $H^1_{\Iw}(V^*)$.  

Moreover, we have the following explicit reciprocity law, proven (essentially) by Colmez \cite[Prop.~VI.3.4]{Colm2}:\ if $z' \in H^1_{\Iw}(V^*) \cong \bD(V^*)^{\psi=1}$, we have
$$
\{ d_n , z' \} = \langle \alpha_n , \exp^*(z'_n) \rangle.
$$
Here the pairing $\{ \cdot, \cdot \}$ is a pairing of $(\varphi,\Gamma)$-modules, $\langle \cdot, \cdot \rangle$ is the pairing from Tate local duality in Galois cohomology, $z'_n$ is the projection of $z'$ to $H^1(\Qp(\mu_{p^n}),V^*)$, and $\alpha_n$ is some element of  $\bD_{\dR,n}(V(1))$ which is described precisely in Corollary \ref{cor:specialized explicit}.
For now, let us just mention that $\alpha_n$ arises by viewing $(1-\varphi)^{-1}(d_n)$ in a generalized Kirillov model of $\pi(V)$, and then evaluating at $p^{-n}$.  

From this explicit formula for $\{d_n,z'\}$ together with the surjectivity of $H^1_{\Iw}(V^*) \to H^1(\Qp(\mu_{p^n}),V^*)$, we can see that the pairing factors through $H^1(\Qp(\mu_{p^n}),V^*)$, as the dependence of the right hand of the formula on $z'$ is only through $z'_n$.  Moreover, the formula vanishes for $z'_n \in H^1_g(\Qp(\mu_{p^n}),V^*)$ as this subspace is the kernel of $\exp^*$. 
In particular, there are unique classes $c_n \in H^1_e(\Qp(\mu_{p^n}),V(1))$ such that
$$
\{ d_n , z' \} = \langle c_n , z'_n \rangle_n;
$$
here the pairing on the right is the pairing on Galois cohomology arising from Tate local duality.
Lastly, under \eqref{nopole}, \cite[Theorem 4.1(ii)]{BK} implies $ H^1_e(\Qp(\mu_{p^n}),V(1)) = H^1_f(\Qp(\mu_{p^n}),V(1))$.

The fact that the $c_n$ satisfy a three-term relation follows from the Hecke properties of $v_{\new}$.
Further we mention that the above reciprocity law is the key to our comparison of 
$
\sum_{\sigma \in \G_n} \left\langle c_n^\sigma, z_{\Kato,n} \right\rangle_n \sigma^{-1}
$
with $\MTnv{1}$, as one recovers $L$-values from $z_{\Kato,n}$ via the dual exponential map. 

\subsection{Normalizations}
The constant $C$ which appears in the statement of Theorem A is a result of our inability to control various normalizations which we now describe.

The first issue that arises is that the construction of the classes $c_n$ only determines them up to a scalar as the class $v_{\new} \otimes v_{\hw}$ is only well-defined up to a scalar.  Our normalization of the choice of this element is spelled out in section \ref{sec:normal} and is well-defined up to a $p$-adic unit.  For now, we just mention that our normalization guarantees that $v_{\new} \otimes v_{\hw} \in \pi(T)$ and thus the $c_n$ take values in $T(1)$ (rather than just in $V(1)$). It appears difficult however to compare this normalization to the normalization one makes to arise at the cohomological periods $\Omega_f^\pm$ which are also well-defined up to a $p$-adic unit.  This is the first reason we are obliged to state our theorem with the ambiguity of a single non-zero constant in $\O$ (independent of $n$ and $j$).  Note that we can assume the constant is in $\O$ as we can scale away any denominators appearing in Theorem A.

The second issue is that the cohomological periods $\Omega_f^\pm$ may not match the periods that arise when using Kato's Euler system (see \cite[section 4.3]{ChanHo} for more details).  Of course, one could simply replace the cohomological periods by Kato's periods, but we choose not to do this because the cohomogical periods are directly computable by modular symbols where as, to us, it is not clear how to compute Kato's periods.

Lastly, we note that Chan-Ho Kim in \cite{ChanHo} has recently proven (under some assumptions) Theorem A in the crystalline case with $j=(k-2)/2$, and moreover, he can take $C=1$ when $2 \leq k \leq p-1$.

\subsection{Organization of the paper}

We have divided the paper into two parts.  The first part is entirely local and devoted to the construction of the classes $c_{n,j}$.  The second half of the paper is more global in nature and is devoted to the construction of algebraic $\theta$-elements and the proof of Theorem A sketched above.  The second part depends on the first part only through (a) the existence of the local classes $c_{n,j}$ satisfying a three-term relation, and (b) an explicit formula for $\sum_\sigma \langle c_{n,j}^\sigma, z' \rangle_n \sigma$ (Proposition \ref{prop:recip}) which is derived through the explicit reciprocity law.  The reader only interested in the global methods from Iwasawa theory may attempt to read the second part with only a limited knowledge of the first part.

\vspace{0.2cm}

\noindent
{\it Acknowledgements}:\ We heartily thank Pierre Colmez and Chan-Ho Kim for several very helpful conversations about this paper and we thank the anonymous referees for their careful reading and their many valuable suggestions.   The second author thanks MPIM-Bonn for its hospitality and stimulating environment during his year-long visit there.  The first author was partially supported by NSF grant DMS-2201242.  The second author was partially supported by NSF grant DMS-2302285 and a Simons Travel Support award.

\part{The local theory}

\section{Rings}

We introduce various rings of Fontaine using  the notation introduced
by Berger and Colmez.

\subsection{Rings in characteristic $p$}

The basic ring $\widetilde{\mathbf E}^+$ is the perfection of the characteristic $p$ ring $\mathcal O_{\mathbb C_p}/p$:
$$\widetilde{\mathbf E}^+ := \plim{x\mapsto x^p} \mathcal O_{\mathbb C_p}/p.$$
This ring admits another description, namely
$$\widetilde{\mathbf E}^+ := \plim{x\mapsto x^p} \mathcal O_{\mathbb C_p};$$
with this latter description, the multiplication is given in the obvious way,
but the addition involves taking an appropriate limit (since raising to the $p$-th
power is a multiplicative homomorphism of $\mathcal O_{\mathbb C_p}$, but is not
an additive homomorphism).  
These two definitions give the same object as the map from the second projective limit to the first  given by reducing each component modulo $p$ is an isomorphism.

The second description has an advantage over the first, in that it puts into 
evidence the fact that $\widetilde{\mathbf E}^+$ is a complete valuation ring, with
valuation given simply by projecting onto the initial component in the
projective limit, and taking the $p$-adic valuation.

We let $\widetilde{\mathbf E}$ denote the field of fractions of the valuation
ring $\widetilde{\mathbf E}^+$.  Alternately, it can be described as 
$$\widetilde{\mathbf E} := \plim{x\mapsto x^p} {\mathbb C_p},$$
with multiplication being given componentwise in the projective limit, and
addition being given by the same formula as before.

We choose now a generator of~$\mathbb Z_p(1)$ and denote the corresponding element of $\widetilde{\mathbf E}^+$ as $\varepsilon$.  
The valuation of $\varepsilon - 1$ is equal to $p/(p-1)$, which is positive.
Thus we also have the description $\widetilde{\mathbf E} = \widetilde{\mathbf E}^+[
(\varepsilon - 1)^{-1} ].$

We define 
$$\mathbf E_{\mathbb Q_p}^+ := \mathbb F_p[[\varepsilon - 1 ]] \subset \widetilde{\mathbf E}^+,$$
and
$$\mathbf E_{\mathbb Q_p} := \mathbb F_p((\varepsilon - 1 )) \subset \widetilde{\mathbf E}.$$
The valuation on $\widetilde{\mathbf E}$ induces the usual discrete valuation (with
the slightly odd $p/(p-1)$ normalization mentioned above) on $\mathbf E_{\mathbb Q_p}$.

The field  $\widetilde{\mathbf E}$ is algebraically closed,
and is in fact the completion of the algebraic closure of $\mathbf E_{\mathbb Q_p}$,
or, equivalently, of the separable closure of $\mathbf E_{\mathbb Q_p}$.
We write $\mathbf E \subset \widetilde{\mathbf E}$
to denote the separable closure of the field $\mathbf E_{\mathbb Q_p}$ (it is a
non-complete valuation ring),
and also write $\mathbf E^+ := \mathbf E \cap \widetilde{\mathbf E}^+$ to denote
the ring of integers in $\mathbf E$.
As already noted, $\widetilde{\mathbf E}$ is the completion of $\mathbf E$.
Finally, we let $\widetilde{\mathbf E}_{\mathbb Q_p}$ denote the completion
of the radiciel ({\it i.e.}\ purely inseparable) closure of $\mathbf E_{\mathbb Q_p}$
in $\widetilde{\mathbf E}$, and write 
$\widetilde{\mathbf E}_{\mathbb Q_p}^+ := \widetilde{\mathbf E}_{\mathbb Q_p} \cap
\widetilde{\mathbf E}^+.$

An important point is that $\Aut(\widetilde{\mathbf E}/\mathbf E_{\mathbb Q_p})
= \Gal(\mathbf E/\mathbf E_{\mathbb Q_p}) = H,$
where $H \subset G_{\mathbb Q_p}$ is the kernel of the cyclotomic character,
and also that $\mathbf E_{\mathbb Q_p} = \mathbf E^H$,
while $\widetilde{\mathbf E}_{\mathbb Q_p} = \widetilde{\mathbf E}^H.$

\subsection{General notational principles}
In addition to introducing a host of fields and domains,
the preceding subsection illustrates some general notational conventions:
(i) Objects in characteristic $p$ are denoted via $\mathbf E$;
(ii) The objects with tildes are in some sense complete, while the objects without
the tildes are the purely algebraic, ``decompleted'' analogues;  
(iii) The objects with $+$ superscripts are integral with respect to the valuation;
(iv) The objects with the $\mathbb Q_p$-subscript are the ones with trivial Galois
action.

There will be some further notational conventions introduced below:\ (v) objects
in characteristic zero obtained from the $\mathbf E$s by Witt vector or Cohen ring
constructions will be denoted with an $\mathbf A$; (vi) objects obtained from the various
$\mathbf A$s by inverting $p$ will be denoted with a $\mathbf B$.

\subsection{Rings in characteristic zero}

The key ring is now $$\widetilde{\mathbf A}^+ := W(\widetilde{\mathbf E}^+)$$
where $W$ denotes Witt vectors.
This gives rise to other $\mathbf A$-type rings related to $(\varphi,\Gamma)$-modules,
and also underlies the definition of $\mathbf B_{\crys}$ and $\mathbf B_{\dR}$.
We let $\varphi$ denote the Frobenius on $\widetilde{\mathbf A}^+$.
As well as its $p$-adic topology, the ring $\widetilde{\mathbf A}^+$ has another topology,
which is more natural to consider, namely its so-called weak topology:
we think of the Witt vectors $W(\widetilde{\mathbf E}^+)$ as being isomorphic to
the product $\prod_{i=0}^{\infty}\widetilde{\mathbf E}^+,$ and give them the corresponding
product topology, where $\widetilde{\mathbf E}^+$ is given its valuation topology.

An important element in $\widetilde{\mathbf A}^+$ is $T := [\varepsilon] - 1$.
(Here, as usual, $[x]$ denotes the circular lift to the Witt vectors of an element
$x \in \widetilde{\mathbf E}^+$.)  Note that this lifts the positive valuation element
$\varepsilon - 1$ of $\widetilde{\mathbf E}^+$.
Thus the weak topology on $\widetilde{\mathbf A}^+$ is also the $(p,T)$-adic topology.

Since $T$ lifts $\varepsilon -1$,
we find that 
$$ T \equiv [\varepsilon - 1 ] \bmod p \widetilde{\mathbf A}^+,$$
and hence that 
$$
\widehat{\widetilde{\mathbf A}^+[\dfrac{1}{T}]}
= 
\widehat{\widetilde{\mathbf A}^+[\dfrac{1}{[\varepsilon -1]}]}
$$
(where $~\widehat{}~$ denotes the $p$-adic completion);
we denote these (equal) rings by $\widetilde{\mathbf A}$.
Since $\widetilde{\mathbf E} = \widetilde{\mathbf E}^+[(\varepsilon -1)^{-1}],$
we have the more canonical description
$$\widetilde{\mathbf A} := W(\widetilde{\mathbf E}).$$

We now imitate all the other $\mathbf E$ constructions in this new $\mathbf A$ context.
Firstly, we set
$$\mathbf A_{\mathbb Q_p} := \widehat{\Z_p((T))} \subset \widetilde{\mathbf A}$$
(this 
is a discretely valued Cohen ring with uniformizer $p$ and
residue field $\mathbf E_{\mathbb Q_p}$),
and then we write
$$\mathbf A^+_{\mathbb Q_p} := \Z_p[[T]] =
\mathbf A_{\mathbb Q_p} \cap \widetilde{\mathbf A}^+.$$ 
We then let $\mathbf A$ denote the $p$-adic completion of the maximal unramified extension
of $\mathbf A_{\mathbb Q_p}$ (in $\widetilde{\mathbf A}$), which is a Cohen ring with
residue field $\mathbf E$, and let $\mathbf A^+ :=
\mathbf A \cap \widetilde{\mathbf A}^+$.
Since $\mathbf E^+$ is dense in $\widetilde{\mathbf E}^+$,
we find that $\mathbf A^+$ is weakly dense in
$\widetilde{\mathbf A}^+$.

Finally, we write
$$\widetilde{\mathbf A}_{\mathbb Q_p} := W(\widetilde{\mathbf E}_{\mathbb Q_p}),$$
and $$\widetilde{\mathbf A}_{\mathbb Q_p}^+ := W(\widetilde{\mathbf E}^+_{\mathbb Q_p}) 
= \widetilde{\mathbf A}_{\mathbb Q_p} \cap \widetilde{\mathbf A}^+.$$
The $H$-action on $\widetilde{\mathbf E}$ induces an $H$-action on
$\widetilde{\mathbf A}$, and
$$\widetilde{\mathbf A}^H = \widetilde{\mathbf A}_{\mathbb Q_p}.$$
Since $H$ fixes $\varepsilon,$ we find that
the $H$-action on $\widetilde{\mathbf A}$ fixes $T$, and hence fixes $\mathbf A_{\mathbb Q_p}$
elementwise.  Thus it preserves $\mathbf A$, and we have
$$\mathbf A^H = \mathbf A_{\mathbb Q_p}.$$

\subsection{The $\mathbf B$-rings}
We can replace $\mathbf A$ by $\mathbf B$ everywhere by inverting $p$.
There is the famous surjection
$$\theta: \widetilde{\mathbf A}^+ \to \mathcal O_{\mathbb C_p},$$
which extends to a surjection
$$\theta: \widetilde{\mathbf B}^+ \to \mathbb C_p,$$
and as usual we write $\mathbf B_{\dR}^+$ to denote the completion of
$\widetilde{\mathbf B}^+$ with respect to the kernel of $\theta$.

Note that the kernel of $\theta$ in $\widetilde{\mathbf A}^+$,
and hence also in $\widetilde{
\mathbf B}^+,$ is principal,
generated by $T/\varphi^{-1}(T)$.  This latter element
is traditionally denoted $\omega$.  Note also that the famous element
$t$ (see just below) does not belong to $\widetilde{\mathbf B}^+$,
but only to the completion $\mathbf B_{\dR}$; there is thus some 
advantage to being aware of the element $\omega$, which generates
the kernel of $\theta$ {\em in} $\widetilde{\mathbf B}^+$.
Indeed, we will have use for $\omega$ later.

Recall that $\mathbf B_{\dR}^+$ is a DVR with uniformizer $t := \log (1 + T)$ and 
with residue field $\mathbb C_p$ (given by the extension of $\theta$,
which we again denote by $\theta: \mathbf B_{\dR}^+ \to \mathbb C_p$).
We set $$\mathbf B_{\dR} := \mathbf B_{\dR}^+[1/t] = \mathbf B_{\dR}^+[1/T].$$
(To see the asserted equality, note that 
$$t = \log (1 + T) = T\left(\sum_{n = 0}^{\infty} (-1)^{n}\dfrac{T^n}{n+1} \right),$$
where the second factor in the right-hand expression is a unit in
$\mathbf B_{\dR}^+$, as $T \in~\ker~\theta$.)

There are further observations regarding $T$, $t$, and $\mathbf B_{\dR}^+$ that
are useful.  For example, $\varphi(T)/T = \varphi(\omega)$ is a unit
in $\mathbf B_{\dR}^+$ (since its image under $\theta$ is $p$, which
is non-zero).  Thus in $\mathbf B_{\dR}^+$, the elements $t$ 
and $\varphi^n(T)$ differ by a unit for all $n \geq 0.$
On the other hand, if $n < 0$ then $\varphi^n(T)$ is a unit in $\mathbf B_{\dR}^+$
(since its image under $\theta$ is non-zero).

There are two more rings that we have to introduce, namely $\mathbf B_{\max}$
and $\widetilde{\mathbf B}_{\rig}^+$. 
From the point of view of $p$-adic Hodge theory,
these rings play the same role as
$$
\mathbf B_{\crys}^+  = \{ x \in \mathbf B_{\dR}^+ ~:~ x = \sum_{n=0}^\infty x_n \frac{\omega^n}{n!} \text{~with~} x_n \to 0 \text{~in~}  \widetilde{\mathbf B}^+\}
$$
in that they compute the crystalline Dieudonn\'e module, but they have the advantage of being
better behaved, and relating more directly to the rings of $(\varphi,\Gamma)$-module
theory.    The definitions are as follows:\ we first let
$\mathbf A_{\max}$ denote the $p$-adic completion of 
$\widetilde{\mathbf A}^+[\omega/p]$,
and then
set $\mathbf B_{\max} := \mathbf A_{\max}[1/p];$
these are both subrings of $\mathbf B_{\dR}$.
Note that $\varphi(\omega) \equiv \omega^p \bmod p \widetilde{\mathbf A}^+$
(by definition of the action of $\varphi$ on a ring of Witt vectors),
so that $\varphi(\omega)/p \in \widetilde{\mathbf A}^+[\omega/p]$.
Thus $\varphi$ extends by continuity to
$\mathbf A_{\max}$ and $\mathbf B_{\max}$.
However $\varphi$ is not surjective on either of these rings,
and we set
$$\widetilde{\mathbf B}_{\rig}^+ := \bigcap_{n \geq 0}\varphi^n(\mathbf B_{\max});$$
we then have, by construction, that
$\varphi$ is bijective on $\widetilde{\mathbf B}_{\rig}^+$.
The sequence of inclusions
$$\widetilde{\mathbf B}^+ \subset \widetilde{\mathbf B}_{\rig}^+
\subset \mathbf B_{\dR}^+$$
will be of fundamental importance to us.

The following proposition relating divisibilities in the rings involved
in these inclusions may seem technical, but is crucial.

\begin{prop}
\label{prop:divisibilities}
~
\begin{enumerate} 
\item
If $x \in \widetilde{\mathbf A}^+$
{\em (}resp.\ $\widetilde{\mathbf B}^+${\em )}
is such that $\varphi^n(x) \in \ker\theta$ for all $n \geq 0$
{\em (}i.e.\ if $t$ divides $\varphi^n(x)$ in $\mathbf B^+_{\dR}$
for all $n \geq 0${\em )},
then $x \in T\widetilde{\mathbf A}^+$
{\em (}resp.\ $T \widetilde{\mathbf B}^+${\em )}.
\item
If $x \in \widetilde{\mathbf B}^+_{\rig}$ is such that
$\varphi^n(x) \in \ker \theta$ for all $n \in \mathbb Z$
{\em (}i.e.\ if $t$ divides $\varphi^n(x)$ in $\mathbf B^+_{\dR}$
for all $n \in \mathbb Z${\em )},
then $x \in t \widetilde{\mathbf B}^+_{\rig}$.
\end{enumerate}
\end{prop}

\begin{proof}
The first part is \cite[Lemma III.3.7]{Colm3}.  For the second part, in \cite[Lemma III.3.4]{Colm3}, it is proven that if $z \in {\mathbf B}_{\max}^+$ with $\varphi^n(z) \in\ker \theta$ for all $n \geq 0$, then $t$ divides $z$ in ${\mathbf B}_{\max}^+$.  Thus, we have $x = tb$ with $b \in \mathbf B_{\max}^+$ and need to show that $b \in \widetilde{\mathbf B}^+_{\rig}$.  
But, taking $z=\varphi^{-m}(x)$ for $m \geq 0$, gives $\varphi^{-m}(x) = t b_m$ for some $b_m \in \mathbf B_{\max}^+$.  Then $b = p^m \varphi^m(b_m)$ for all $m \geq 0$ implying that  $b \in \widetilde{\mathbf B}^+_{\rig}$.  
\end{proof}

\begin{cor}
\label{cor:divisibilities}
For any $m \geq 0,$
the inclusion
$\varphi^m(T) \widetilde{\mathbf A}^+ \subset
\widetilde{\mathbf A}^+
\cap \,  \varphi^m(T) \widetilde{\mathbf B}_{\rig}^+$
is an equality
{\em (}and similarly with $\widetilde{\mathbf B}^+$ 
in place of $\widetilde{\mathbf A}^+$.{\em )}
\end{cor}
\begin{proof}
Let 
$x \in \widetilde{\mathbf A}^+
\cap \, \varphi^m(T) \widetilde{\mathbf B}_{\rig}^+$.
Then $\varphi^{-m}(x) \in
\widetilde{\mathbf A}^+
\cap \, T \widetilde{\mathbf B}_{\rig}^+$.
Then for any $n \geq 0,$
we have $\varphi^{n-m}(x) \in  \varphi^n(T) \widetilde{\mathbf B}_{\rig}^+
\subset \ker \theta$, and so part~(1) of 
the preceding proposition shows that $\varphi^{-m}(x) \in T \widetilde{\mathbf A}^+.$
Applying $\varphi^m$, we find that $x \in \varphi^m(T)\widetilde{\mathbf A}^+$,
as required.
\end{proof}

\section{$(\varphi,\Gamma)$-modules and $p$-adic Hodge theory}
We will recall the constructions of $(\varphi,\Gamma)$-modules
attached to representations of~$G_{\mathbb Q_p}$, as well as various 
constructions in $p$-adic Hodge theory, and the relations between them.  
This material is all standard, but the systematic use of the rings
with tildes is perhaps less familiar than it might be, 
and is very important for us, in part because it gives a very simple approach
to relating $(\varphi,\Gamma)$-module theory to $p$-adic Hodge theory.

\subsection{$(\varphi,\Gamma)$-modules}
If $T$ is a continuous $G_{\mathbb Q_p}$-representation on a finite type $\mathbb Z_p$-module,
then we define
$$\mathbf D(T) := (\mathbf A \otimes_{\Z_p} T)^H,$$
which is an \'etale $(\varphi,\Gamma)$-module over $\mathbf A_{\mathbb Q_p}$.
We can equally well define
$$\widetilde\bD(T) := (\widetilde{\mathbf A}\otimes_{\Z_p} T)^H,$$
which is an \'etale $(\varphi,\Gamma)$-module over $\widetilde{\mathbf A}_{\mathbb Q_p}$.
(Note that in this second context,
the notion of \'etale is particularly simple:\ the map $\varphi$
should simply be bijective.)
If $V$ is a continuous $G_{\mathbb Q_p}$-representation on a finite dimensional
$\mathbb Q_p$-vector space  (a $p$-{\em adic representation}, for short),
then we define $\mathbf D(V)$ and $\widetilde\bD(V)$ by the same 
formulas (or, perhaps a little more naturally, 
we could use $\mathbf B$ and $\widetilde{\mathbf B}$ instead),
and obtain \'etale $(\varphi,\Gamma)$-modules over $\mathbf B_{\mathbb Q_p}$
or $\widetilde{\mathbf B}_{\mathbb Q_p}$ instead.
Either one of these functors, $\mathbf D$ or $\widetilde\bD$,
gives an equivalence of categories between $G_{\mathbb Q_p}$-representations
and \'etale $(\varphi,\Gamma)$-modules.  (This is probably most familiar in the
$\mathbf D$ context.)   

Now one can similarly define $\mathbf D^+(T)$ and $\widetilde\bD^+(T)$
using $\mathbf A^+$ and $\widetilde{\mathbf A}^+$ instead.  
If $T$ is {\em torsion} over $\mathbf Z_p$, then $\mathbf D^+(T)$ is a
finite type $\mathbf A_{\mathbb Q_p}^+$-submodule of $\mathbf D(T)$,
which generates $\mathbf D(T)$ over $\mathbf A_{\mathbb Q_p}$.
However, if $T$ is not torsion,
then $\mathbf D^+(T)$ vanishes in general. 
Similarly, in the case of a $p$-adic representation~$V$,
it is generally the case that  $D^+(V)$  vanishes.
One says that $V$ is of {\em finite height}
if $\mathbf D^+(V)$ generates $\mathbf D(V)$ over $\mathbf A_{\mathbb Q_p}$,
and this is a very restrictive condition.\footnote{Related to this,
Colmez \cite{Colm4} has proved a conjecture of Fontaine, stating that
all crystalline representations are of finite height.   The theory of 
Wach modules \cite{Berger-limits}
is a strong witnessing of this fact, since the Wach module
of a crystalline representation $V$ is pretty close to $\mathbf D^+(V)$.}

On the other hand, the module $\widetilde\bD^+(V)$ {\em always}
generates $\widetilde\bD(V)$ over $\widetilde{\mathbf A}_{\mathbb Q_p}$.
Thus there is an asymmetry between the $\mathbf D^+$ context and the $\widetilde\bD^+$
context, which is not there when we remove the $+$.  

For later use, we note that there is a kind of trace map
$\Tr:\widetilde\bD(V)\to \mathbf D(V),$
which restricts to the identity on $\mathbf D(V)$ (or equivalently,
it is a projection onto $\mathbf D(V)$, {\it i.e.}\ it is surjective, and $\Tr\circ \Tr = \Tr$),
and with the additional properties that it is continuous (with respect to
the weak topologies on source and target), it is $\mathbf A_{\mathbb Q_p}$-linear,
and $\Tr\circ \varphi^{-1} = \psi \circ \Tr.$  (Here $\psi: \mathbf D(V) \to \mathbf D(V)$ is the usual left inverse of $\varphi$ on \'etale $(\varphi,\Gamma)$-modules.)

Using $\Tr$, we may define a map
$$\widetilde\bD(V) \to \plim{\psi} \bD(V) =: 
\bD(V) \boxtimes \mathbb Q_p,$$
via
$$\tilde{z} \mapsto  \bigl(\Tr\varphi^n(\tilde{z}) \bigr)_{n \geq 0}.$$
This map is injective (but not surjective in general), 
and the image of $\widetilde\bD^+(V)$ is equal to
the intersection of the image of $\widetilde\bD(V)$ with
$\bD^{\natural}(V)\boxtimes
\mathbb Q_p.$  Here $\bD^{\natural}(V)$ is as defined in \cite[I.3.2]{Colm2}.

{
There are a slew of other species of $(\varphi,\Gamma)$-module --- see, for instance, \cite[V.1]{Colm2}.    For instance, we can define $\widetilde\bD_{\rig}^+(V)
:= (\widetilde{\mathbf B}_{\rig}^+\otimes V)^H$ and this module will play a large role in what follows.
In section \ref{sec:rec}, we will also make use of several other kinds of $(\varphi,\Gamma)$-modules.  Namely, if $\sR$ is the Robba ring over $\Qp$ --- that is, the collection of $\Qp$-power series which converge on an annulus defined by $0 < v_p(\cdot) \leq r$ for {\it some} $r$ --- we have a $(\varphi,\Gamma)$-module $\bD_{\rig}(V)$ which is a module over $\sR$.  

We will also need $(\varphi,\Gamma)$-modules over even larger rings.  Namely, the ring $\widetilde{\mathbf B}_{\rig}^\dag$ is the largest of all the period rings we will use in this paper (see \cite[\S 2]{Berger-p-adic}).  With this ring in hand, we define $\widetilde{\bD}^\dag_{\rig}(V) := (\widetilde{\mathbf B}_{\rig}^\dag \otimes V)^H$ which is a module over $\widetilde{\sR} := (\widetilde{\mathbf B}_{\rig}^\dag)^H$.}

\subsection{Sen's theory, and Fontaine's generalization}
The relation between $\widetilde\bD$ and $\mathbf D$ (and in particular,
the fact that no information is lost when one passes from the former to the latter),
is an instance of a more general principle of ``decompletion'',
first studied by Sen.

The basic point in the $\widetilde\bD/\mathbf D$-context is that 
the ring $\widetilde{\mathbf A}$ used to compute
$\widetilde\bD$ is the completion (in the weak topology)
of the ring $\mathbf A$ used to compute $\mathbf D$.
Since the resulting coefficient ring $\widetilde{\mathbf A}_{\mathbb Q_p}$
is {\em not} the completion of $\mathbf A_{\mathbb Q_p}$, 
this situation does not compare perfectly with Sen's, but the analogy is still meaningful.

Let us now recall Sen's context.  He works with the period ring $\mathbb C_p$.
Recall that $\mathbb C_p^H = {\hatKinfp}$ where $\Kinfp := \Qp(\mu_{p^\infty})$.  Thus if $V$ is
a continuous $G_{\mathbb Q_p}$-representation over $\mathbb Q_p$, we can consider
$\widetilde\bD_{\Sen}(V)  :=(\mathbb C_p \otimes_{\mathbb Q_p} V)^H,$ which is a semi-linear $\Gamma$-representation
over $\hatKinfp$, of dimension (over $\hatKinfp$)
equal to the dimension of $V$ over $\mathbb Q_p$.
Now Sen's theorem says that we can descend this canonically to a semi-linear
$\Gamma$-representation over $\Kinfp$ itself, denoted
$\mathbf D_{\Sen}(V)$, where ``descend'' has the meaning that
the natural map
$$\hatKinfp\otimes_{\Kinfp} \mathbf D_{\Sen}(V) 
\to \widetilde\bD_{\Sen}(V)$$ is an isomorphism.

What is the benefit of this?  Well, any element of $\Kinfp$ 
is actually invariant under a sufficiently small open subgroup of $\Gamma$,
and so the Lie algebra $\Lie \Gamma$ (which is canonically identified with
$\mathbb Q_p$ via the derivative of the cyclotomic character) acts {\em linearly}
on the $\Kinfp$-vector space $\mathbf D_{\Sen}(V)$.
In particular, the basis vector $1 \in \mathbb Q_p$ gives rise to a linear
operator $\Delta_V$ on $\mathbf D_{\Sen}(V),$ and it is the eigenvalues of
this operator that are the Hodge--Sen--Tate weights of $V$.
(So the decompletion allows us to pass from a semi-linear context to a linear
one, where we can then use linear algebra to define invariants.)

Fontaine \cite{Fon-Sen-theory}
generalized Sen's theory as follows:\ instead of $\mathbb C_p$,
the period ring is now $\mathbf B_{\dR}$.  Remember that this a twisted version
of $\mathbb C_p((t))$.  Since it contains the algebraic closure $\Qbar_p$,
it does contain $\Qbar_p((t))$ as a dense subfield.  Noting that $H$ fixes $t$,
we find that $\mathbf B_{\dR}^H$ contains $\Kinfp((t))$ as a dense subfield.

Now if $V$ is as before, we define 
$$\widetilde\bD_{\dif}(V) := (\mathbf B_{\dR} \otimes_{\mathbb Q_p} V)^H,$$
which is a vector space over $\mathbf B_{\dR}^H$ of dimension equal to the
dimension of $V$ over $\mathbb Q_p$, equipped with a semi-linear $\Gamma$-action. 
What Fontaine shows is that we may canonically descend
$\widetilde\bD_{\dif}(V)$, together with its $\Gamma$-action,
to a vector space $\mathbf D_{\dif}(V)$ over $\Kinfp((t))$,
where again ``descend'' means that
$$\mathbf B_{\dR}^H\otimes_{\Kinfp((t))} \mathbf D_{\dif}(V)
\to \widetilde\bD_{\dif}(V)$$
is an isomorphism.
On $\bD_{\dif}(V)$, the basis vector $1 \in \mathbb Q_p = \Lie \Gamma$
acts by a $\Kinfp$-linear differential operator $\Delta_V$
(satisfying the Leibnitz rule $\Delta_V(t x) = tx + t \Delta_V(x)$).

We can proceed similarly with $\mathbf B_{\dR}^+$ in place of $\mathbf B_{\dR}$,
to define $\widetilde\bD^+_{\dif}(V)$ and $\mathbf D^+_{\dif}(V)$,
which are modules over $(\mathbf B_{\dR}^+)^H$ and $\Kinfp[[t]]$ 
respectively, and both of which are invariant under $\Gamma$. 
The latter is thus also preserved by the differential operator $\Delta_V$.
Note that reducing mod $t$ returns us to Sen's situation:\ $$\mathbf D_{\dif}^+(V)/t
\mathbf D_{\dif}^+(V) \iso \mathbf D_{\Sen}(V)$$
(compatibly with the operators $\Delta_V$).

\subsection{So what?}
\label{sec:sowhat}
The inclusions $\widetilde{\mathbf B}^+ \subset 
\widetilde{\mathbf B}_{\rig}^+ \subset \mathbf B_{\dR}^+$
gives rise to canonical inclusions 
\begin{equation}
\label{eqn:first}
\widetilde\bD^+(V) \subset
\widetilde\bD_{\rig}^+(V)\subset
\widetilde\bD^+_{\dif}(V)
\end{equation}
(the first being $(\varphi,\Gamma)$-equivariant, and the second
being $\Gamma$-equivariant).
These inclusions set up a relation between the world of $(\varphi,\Gamma)$-modules
and the world of $p$-adic Hodge theory which is at the basis of certain
explicit reciprocity laws, and also at the heart of Colmez's approach
to the $p$-adic local Langlands correspondence.

To begin our discussion, we first note that the inclusion
$
\widetilde\bD^+(V) \subset
\widetilde\bD^+_{\dif}(V)$
does not extend to a map with the $+$'s removed, because $\widetilde\bD(V)$
is obtained by inverting $T$ and then $p$-adically completing,
while $\widetilde\bD_{\dif}(V)$ is obtained by inverting $t$ (which
certainly allows us to invert $T$, since $T$ divides $t$), but without any subsequent
$p$-adic completion.

Nevertheless, each of the 
elements $\varphi^n(T)$ divides $t$ in $\widetilde{\mathbf B}_{\rig}^+$,
and so also in $\mathbf B_{\dR}^+$,
and so we do have inclusions
$$
\widetilde\bD^+(V)[1/\varphi^n(T)] \subset
\widetilde\bD^+_{\rig}(V)[1/t]
\subset
\widetilde\bD^+_{\dif}(V)[1/t] = \widetilde\bD_{\dif}(V),
$$
for each $n \geq 0.$
Thus, if we write
$$\widetilde\bD^+(V)[\bigr(1/\varphi^n(T)\bigr)_{n\geq 0}]
:=
\bigcup_{n \geq 0} \widetilde\bD^+(V)[1/\varphi^n(T)],
$$ 
then we obtain the inclusions
\begin{equation}
\label{eqn:second}
\widetilde\bD^+(V)[\bigl(1/\varphi^n(T)\bigr)_{n\geq 0}]
\subset
\widetilde\bD^+_{\rig}(V)[1/t]
\subset
\widetilde\bD_{\dif}(V),
\end{equation}
which in turn induces maps
\begin{equation*}
\widetilde\bD^+(V)[\bigl(1/\varphi^n(T)\bigr)_{n \geq 0}]/
\widetilde\bD^+(V)
\to
\widetilde\bD^+_{\rig}(V)[1/t]/\widetilde\bD^+_{\rig}(V)
\to
\widetilde\bD_{\dif}(V)/\widetilde\bD_{\dif}^+(V).
\end{equation*}
The first of these maps is an injection, by Corollary~\ref{cor:divisibilities},
and we regard it as an inclusion.
The second map, which is not injective
(unless $V = 0$),
will be crucial for us, and so we give
it a name, namely $\imath^-_0$. (The minus sign is to remind us that
we are quotienting out by the plus objects; the reason for 
the subscript $0$ will become apparent soon.)

A fundamental fact is that we have $\varphi$ actions on
our $(\varphi,\Gamma)$-modules, but not in $p$-adic Hodge theory.
In other words,
$\imath^-_0$ is a $\Gamma$-equivariant map from a module with
a $\varphi$-action to a module without such an action.  
General principles of algebra then suggest that we induce 
$\widetilde\bD_{\dif}(V)/\widetilde\bD_{\dif}^+(V)$
to obtain the module
$$
\prod_{i \in \mathbb Z}
\widetilde\bD_{\dif}(V)/\widetilde\bD_{\dif}^+(V),
$$
which we equip with the $\varphi$-action given by translation,
namely
$$\varphi\bigl((x_i)\bigr) := (x_{i-1}).$$
(The reason for the negative, rather than positive, shift, is to accord
with existing conventions in the literature, in particular in \cite{Colm2}.)
We make the product
$
\prod_{i \in \mathbb Z}
\widetilde\bD_{\dif}(V)/\widetilde\bD_{\dif}^+(V)
$
a module over $\widetilde{\mathbf B}^+_{\mathbb Q_p}$ by defining the action of 
an element $b$ of this ring on an element $(x_i)_{i \in \mathbb Z}$
of the product via
$$b\cdot (x_i)_{i \in \mathbb Z} := (\varphi^{-i}(b) x_i).$$
In this way the product
$
\prod_{i \in \mathbb Z}
\widetilde\bD_{\dif}(V)/\widetilde\bD_{\dif}^+(V),
$
becomes an \'etale $(\varphi,\Gamma)$-module over $\widetilde{\mathbf B}^+_{\mathbb Q_p}$.

We then define a map 
$$\imath^-: 
\widetilde\bD^+_{\rig}(V)[1/t] /
\widetilde\bD^+_{\rig}(V)
\to
\prod_{i \in \mathbb Z}
\widetilde\bD_{\dif}(V)/\widetilde\bD_{\dif}^+(V)
$$
of $(\varphi,\Gamma)$-modules over $\widetilde{\mathbf B}^+_{\mathbb Q_p}$ 
via the formula
$$\imath^-(x) = \Bigl(\imath^-_0\bigl(\varphi^{-i}(x)\bigr)\Bigr)_{i \in \mathbb Z}.$$
For any $i \in \mathbb Z,$ we define
$$\imath^-_i:
\widetilde\bD^+_{\rig}(V)[1/t]/
\widetilde\bD^+_{\rig}(V)
\to
\widetilde\bD_{\dif}(V)/\widetilde\bD_{\dif}^+(V)
$$
via
$\imath^-_i(x) := \imath^-_0\bigl(\varphi^{-i}(x)\bigr),$
so that $\imath^-$ also admits the description
$\imath^-(x) := \bigl(\imath^-_i(x)\bigr)_{i \in \mathbb Z}.$

\begin{prop}
\label{prop:support}
We have:
\begin{enumerate}
\item
The map $\imath^-$ is injective.
\item For any $n \geq 0,$ the restriction
of $\imath^-_i$ to
$
\widetilde\bD^+(V)[1/\varphi^n(T)]/
\widetilde\bD^+(V)$
vanishes if $i > n$.
\end{enumerate}
\end{prop}
\begin{proof}
Fix $x = b/\varphi^n(T)^k \in 
\widetilde\bD^+(V)[1/\varphi^n(T)].$
For any $i \in \mathbf Z$ we see that
$\varphi^{- i}(x) = \varphi^{-i}(b)/(\varphi^{n-i}(T))^k.$
Since $\varphi^{n - i}(T)$ is a unit in $\mathbf B_{\dR}^+$ when $i > n,$
we see that $\varphi^{-i}(x) \in \widetilde\bD_{\dif}^+(V)$ when $i > n,$
proving~(2).
On the other hand, if $\varphi^{-i}(x) \in \widetilde\bD_{\dif}^+(V)$
for some $x \in \widetilde\bD^+_{\rig}(V)[1/t]/\widetilde\bD^+_{\rig}(V)$
and all $i$, then we conclude from part~(2) of Proposition~\ref{prop:divisibilities}
that $x \in \widetilde\bD^+_{\rig}(V)$, proving part~(1) of the present
proposition.
\end{proof}

From part~(2) of the preceding result we see in particular that,
when restricted to $\widetilde\bD^+(V)[\bigl(1/\varphi^n(T)\bigr)_{n\geq 0}]/
\widetilde\bD^+(V),$
the map $\imath^-$ in fact takes values in
$$\ilim{n} \prod_{i \leq n}
\widetilde\bD_{\dif}(V)/\widetilde\bD_{\dif}^+(V).
$$

The next result involves a slight refinement of the argument used to prove part~(1)
of the preceding proposition.

\begin{prop}
\label{prop:compact support}
We have:
\begin{enumerate}
\item
For $n \neq i$ we have that $\imath^-_i$ vanishes on
$$
\widetilde\bD^+(V)[1/\varphi^n(\omega)]/
\widetilde\bD^+(V),$$
while $\imath^-_n$ induces an isomorphism 
$$
\widetilde\bD^+(V)[1/\varphi^n(\omega)]/
\widetilde\bD^+(V)
\iso
\widetilde\bD_{\dif}(V)/
\widetilde\bD_{\dif}^+(V).
$$
\item
The evident map
$$
\bigoplus_{n\in \mathbb Z}
\widetilde\bD^+(V)[1/\varphi^n(\omega)]/
\widetilde\bD^+(V)
\to
\widetilde\bD^+(V)[\bigl(1/\varphi^n(\omega)\bigr)_{n \in \Z}]/
\widetilde\bD^+(V)
$$ 
is an isomorphism,
and $\imath^-$ induces an isomorphism
$$
\widetilde\bD^+(V)[\bigl(1/\varphi^n(\omega)\bigr)_{n \in \Z}]/
\widetilde\bD^+(V)
\iso
\bigoplus_{n\in \mathbb Z}
\widetilde\bD_{\dif}(V)/\widetilde\bD_{\dif}^+(V).
$$
\end{enumerate}
\end{prop}
\begin{proof}
The key point is that if $n \neq 0,$ then $\varphi^n(\omega)$ is a unit
in $\mathbf B^+_{\dR},$
which in turn follows from the facts that
if $a \geq b\geq 0$, then $\varphi^a(T)/\varphi^b(T)$
is a unit in $\mathbf B^+_{\dR}$, while if $n < 0,$ then
$\varphi^n(T)$ is a unit in $\mathbf B^+_{\dR}.$
\end{proof}

\section{Representations of the mirabolic}
\label{sec:mirabolic}
Let
$P := \begin{pmatrix} \mathbb Q_p^{\times} & \mathbb Q_p \\ 0 & 1\end{pmatrix}$
denote the mirabolic subgroup of $\GL_2$.
We also write
$P^+  := \begin{pmatrix} \mathbb Z_p\setminus \{0\} & \mathbb Z_p \\ 0 & 1 \end{pmatrix};$
this is a submonoid of $P$.
Note that $P$ is generated by $P^+$ together with $\begin{pmatrix} p^{-1} & 0 \\ 0 & 1 
\end{pmatrix}$, and that giving a representation of $P$ is the same as giving a representation
of $P^+$ on which the action of $\begin{pmatrix} p & 0 \\ 0 & 1 \end{pmatrix}$ is invertible.  Also, set 
$N := \begin{pmatrix}  1 & \mathbb Q_p \\ 0 & 1\end{pmatrix}$, the unipotent  subgroup of $\GL_2$.

\subsection{$P$-representations from $(\varphi,\Gamma)$-modues.}
A $(\varphi,\Gamma)$-module $\mathbf D$ with coefficients in~$\mathbf A^+_{{\mathbb Q}_{p}}$
may be equipped with the structure of a $P^+$ representation
in the following manner:
the matrix $\begin{pmatrix}p & 0 \\ 0 & 1\end{pmatrix}$
acts via $\varphi,$ 
the matrices $\begin{pmatrix} \Gamma & 0 \\ 0 & 1 \end{pmatrix}$ act
through the given action of $\Gamma$ on the $(\varphi,\Gamma)$-module,
and
the matrix $\begin{pmatrix} 1 & 1 \\ 0 & 1 \end{pmatrix}$ acts via $1 + T$.
This construction applies in particular 
if $\widetilde\bD$ is an \'etale $(\varphi,\Gamma)$-module
over~$\widetilde{\mathbf A}^+_{{\mathbb Q}_{p}}$.
Since in this case the action of $\varphi$ on $\widetilde\bD$ is invertible
(this is the definition of \'etale in this context), we see that the $P^+$-action
on $\widetilde\bD$ extends to an action of~$P$, so that $\widetilde\bD$
is naturally a $P$-representation.

In particular, if $V$ is any $p$-adic $G_{\mathbb Q_p}$-representation,
we obtain a natural $P$-representation on the quotients
$$\widetilde\bD\bigl(V(1)\bigr)/\widetilde\bD^+\bigl(V(1)\bigr),$$
$$\widetilde\bD^+\bigl(V(1)\bigr)[\bigl(1/\varphi^n(T)\bigr)_{n\geq 0}]/
\widetilde\bD^+\bigl(V(1)\bigr),$$
and
$$\widetilde\bD_{\rig}^+\bigl(V(1)\bigr)[1/t]/
\widetilde\bD_{\rig}^+\bigl(V(1)\bigr),$$
as well as on the product
$$
\prod_{i \in \mathbb Z}
\widetilde\bD_{\dif}\bigl(V(1)\bigr)/
\widetilde\bD^+_{\dif}\bigl(V(1)\bigr),$$
and the map $\imath^-$ from either 
of the latter two quotients to the product is $P$-equivariant. 

In terms of this $P$-action,
the subspace
$\widetilde\bD^+\bigl(V(1)\bigr)[\bigl(1/\varphi^n(T)\bigr)_{n\geq 0}]/
\widetilde\bD^+\bigl(V(1)\bigr)$
of
$\widetilde\bD\bigl(V(1)\bigr)/\widetilde\bD^+\bigl(V(1)\bigr)$
is identified as its subspace of $N$-locally algebraic 
vectors.
In particular, the $N$-smooth vectors in
$\widetilde\bD\bigl(V(1)\bigr)/\widetilde\bD^+\bigl(V(1)\bigr)$
are equal to the subspace
$\bigcup_{n \geq 0} \dfrac{1}{\varphi^n(T)}\widetilde\bD^+\bigl(V(1)\bigr)/\widetilde\bD^+\bigl(V(1)\bigr).$
Their image under the map $\imath^-$ lies in the subspace of 
$N$-smooth vectors in 
$
\prod_{i \in \mathbb Z}
\widetilde\bD_{\dif}\bigl(V(1)\bigr)/
\widetilde\bD^+_{\dif}\bigl(V(1)\bigr),$
which is equal to
$$
\prod_{i \in \mathbb Z}
\dfrac{1}{t} \widetilde\bD_{\dif}^+\bigl(V(1)\bigr)/
\widetilde\bD^+_{\dif}\bigl(V(1)\bigr)
=
\prod_{i \in \mathbb Z}
\dfrac{1}{t} \widetilde\bD_{\Sen}\bigl(V(1)\bigr)
=
\prod_{i \in \mathbb Z}
\widetilde\bD_{\Sen}(V).
$$

The $P$-smooth vectors in 
$\widetilde\bD\bigl(V(1)\bigr)/\widetilde\bD^+\bigl(V(1)\bigr)$
may be identified with the $\Gamma$-smooth vectors of the $N$-smooth
vectors.  Thus, they map under $\imath^-$ to the
$\Gamma$-smooth vectors in
$\prod_{i \in \mathbb Z} \widetilde\bD_{\Sen}(V),$
which is contained in the product of the $\Gamma$-smooth vectors 
in each individual factor.
Sen theory shows that
the $\Gamma$-smooth vectors in $\widetilde\bD_{\Sen}(V)$
are non-zero if and only if $V$ admits zero as a Hodge--Sen--Tate weight, 
and, more precisely, that their dimension as $\Kinfp$-vector space
is equal to the multiplicity of zero as a Hodge--Sen--Tate weight
of $V$.  In particular,
if $V$ does not admit zero as a Hodge--Sen--Tate weight, then
the subspace of $P$-smooth vectors in 
$\widetilde\bD\bigl(V(1)\bigr)/\widetilde\bD^+\bigl(V(1)\bigr)$
vanishes.

\subsection{The relationship with $\mathbf D_{\crys}$}
\label{sec:crys}
For any $p$-adic representation $V$
of $G_{\mathbb Q_p}$, we have
$$\mathbf D_{\crys}(V) :=
\bigl(\widetilde{\mathbf B}_{\rig}^+[1/t]\otimes_{\mathbb Q_p}
V\bigr)^{G_{\Q_{p}}} =
\bigl(\widetilde\bD_{\rig}^+(V)[1/t]\bigr)^{\Gamma}.
$$
(The fact that we can do this with $\widetilde{\mathbf B}_{\rig}^+[1/t]$, rather than
with $\mathbf B_{\crys}$, is one of the basic results of the theory;
see {\it e.g.}\ \cite[\S III.2]{Colm3}.)
Thus, replacing $V$ by $V(1)$, we get a $\varphi$-equivariant map
\begin{equation}
\label{eqn:crys}
\mathbf D_{\crys}\bigl(V(1)\bigr)
\rightarrow
\Bigl(\widetilde\bD_{\rig}^+\bigl(V(1)\bigr)[1/t]/
\widetilde\bD_{\rig}^+\bigl(V(1)\bigr)\Bigr)^{\Gamma}.
\end{equation}
If we compose this with $\imath^-_0,$ we just get the canonical map
$$
\mathbf D_{\crys}\bigl(V(1)\bigr) \to \mathbf D_{\dR}\bigl(V(1)\bigr)/
\mathbf D_{\dR}^+\bigl(V(1)\bigr).
$$

\begin{remark}
The image of $\mathbf D_{\crys}\bigl(V(1)\bigr)$ under $\imath^-$
is finite-dimensional and stable under the torus (since it is
fixed by $\Gamma$, and stable under $\varphi$).  Thus, this image  has trivial intersection
with 
$\imath^-\Bigl(
\widetilde\bD^+\bigl(V(1))[\bigl(1/\varphi^n(T)\bigr)_{n\geq 0}]/
\widetilde\bD^+\bigl(V(1)\bigr)
\Bigr)$,
since the $i$-th coordinate of an element of
$\imath^-\Bigl(
\widetilde\bD^+\bigl(V(1))[\bigl(1/\varphi^n(T)\bigr)_{n\geq 0}]/
\widetilde\bD^+\bigl(V(1)\bigr)
\Bigr)$
vanishes for $i$ large enough, 
whereas clearly a non-zero $\varphi$-eigenvector doesn't have this property.
\end{remark}

\section{Iwasawa cohomology and explicit reciprocity laws}
\label{sec:rec}

Let $L$ denote a finite extension of $\Qp$ with ring of integers $\O_L$.   
For the remainder of the paper we will consider $p$-adic representations which are $L$-vector spaces.  Note that if $V$ is such a representation and $T \subseteq V$ is a Galois stable $\O_L$-lattice, then $D(T)$ is a module over $\bA_{\Qp} \otimes_{\Zp} \O_L$.

We also set some notation that will be in place for the remainder of the paper.  For a vector space $V$ over $L$, set $V^* = \Hom(V,L)$.  For an $\O_L$-lattice $T$, set $T^* = \Hom(T,\O_L)$.   For an $\O_L$-module $A$, set $A^\vee = \Hom(A,L/\O_L)$. 

\subsection{$(\varphi,\Gamma)$-modules and Iwasawa cohomology}
\label{sec:Iwcoh}
The fundamental fact (see \cite{CC}) relating $(\varphi,\Gamma)$-modules
to Iwasawa theory is the following:\ if $T$ is a  $G_{\Q_p}$-invariant
 $\mathcal O_L$-lattice in
a finite-dimensional representation $V$ of $G_{\mathbb Q_p}$ over~$L$,
then
\begin{equation}
\label{eqn:integral iwasawa}
\mathbf D(T)^{\psi = 1} \iso H^1_{\Iw}(T).
\end{equation}
Inverting $p$ then yields an isomorphism
\begin{equation}
\label{eqn:rational iwasawa}
\mathbf D(V)^{\psi = 1} \iso H^1_{\Iw}(V).
\end{equation}

We let $\mathcal R^+(\Gamma)$ denote the ring of locally analytic
distributions on $\Gamma$, or equivalently, the
ring of rigid analytic functions on weight space $\mathcal W :=
\Hom_{\cont}(\Gamma,\overline{\mathbb Q}_p^{\times}).$
(This notation is due to Colmez, and is inspired by
the analogous notation in the theory of the Robba ring.)
The isomorphism~(\ref{eqn:integral iwasawa}) induces an isomorphism
$$
\mathcal R^+(\Gamma)\otimes_{\mathcal O_L[[\Gamma]]}
\mathbf D(T)^{\psi = 1}
\iso 
\mathcal R^+(\Gamma)\otimes_{\mathcal O_L[[\Gamma]]}
H^1_{\Iw}(T).
$$
There is also a natural isomorphism
\begin{equation}
\label{eqn:formal to rigid}
\mathcal R^+(\Gamma)\otimes_{\mathcal O_L[[\Gamma]]}
\mathbf D(V)^{\psi = 1}
\iso 
\mathbf D_{\rig}(V)^{\psi = 1}
,
\end{equation}
and hence we obtain an isomorphism
\begin{equation}
\label{eqn:basic}
\mathbf D_{\rig}(V)^{\psi = 1}
\iso
\mathcal R^+(\Gamma)\otimes_{\mathcal O_L[[\Gamma]]}
H^1_{\Iw}(T).
\end{equation}

\subsection{Pairings:\ $(\varphi,\Gamma)$-modules}\label{sec:pair-phi}
Let $\mathbf D := \mathbf D(T)$ be an \'etale $(\varphi,\Gamma)$-module over $\bA_{\Qp} \otimes_{\Zp} \O_L$, 
associated to a $G_{\Q_p}$-invariant $\mathcal O_L$-lattice $T$ in a $p$-adic
representation~$V$.
As in \cite[I.2.2]{Colm1}, we define the Tate dual of $\mathbf D$ as 
$
\check\bD := \Hom_{\bA_{\Qp}}(\mathbf D,\bA_{\Qp} \frac{dT}{1+T} \otimes_{\Zp} \O_L),
$
its $(\varphi,\Gamma)$-module structure on $\check\bD$ being determined by
the formulas
$\sigma_a(\frac{dT}{1+T}) = a \frac{dT}{1+T}$ for $\sigma_a \in \Gamma$
corresponding to an element $a \in \Z_p^{\times}$,
and $\varphi(\frac{dT}{1+T}) = \frac{dT}{1+T}$.
Note that these formulas identify $\bA_{\Qp} \frac{dT}{1+T}$
with~$\mathbf D\bigl(\Z_p(1)\bigr),$
and so Fontaine's equivalence of categories yields an identification
$\check\bD = \mathbf D\bigl(T^*(1)\bigr)$.
We denote the corresponding tautological pairing by
$
\langle \cdot , \cdot \rangle : \mathbf D \times \check\bD \to \bA_{\Qp}  \frac{dT}{1+T} \otimes_{\Zp} \O_L.
$ 

Define the residue map $\res_0 : \bA_{\Qp}dT = \bA_{\Q_p}\frac{dT}{1+T}\to \Zp$ by $\res_0(f dT) = a_{-1}$ where $f = \sum_k a_k T^k$.
Extending this map $\mathcal O_L$-linearly,
and then passing to residues in $\langle \cdot,\cdot \rangle$,
yields a perfect pairing
\begin{align*}
\{ \cdot , \cdot \} : \mathbf D \times \check\bD &\to \O_L \\
(x,y) &\mapsto \res_0( \langle \sigma_{-1} x, y \rangle )
\end{align*}
which identifies $\check\bD$ with the topological $\O_L$-dual of $\mathbf D$.

The preceding discussion has analogues for other flavors of
$(\varphi,\Gamma)$-modules.
The  most evident comes from working over $L$  rather than $\mathcal O_L$,
so that $\mathbf D$ is the \'etale $(\varphi,\Gamma)$-module associated to the
$p$-adic representation $V$ itself (rather than its lattice~$T$). 
Then $\check\bD$ is associated to $V^*(1)$, and the pairing
$\{ \cdot, \cdot \}$ identifies $\check\bD$ with the topological $L$-dual of~$\mathbf D$.
Similarly, 
for \'etale $(\varphi,\Gamma)$-modules defined over the Robba ring~$\sR \otimes_{\Qp} L$,
we have analogous perfect pairings, still denoted by $\langle \cdot, \cdot \rangle$
and $\{ \cdot,\cdot \}$, again valued in $L$. 
 See \cite[pg.\ 402]{Colm2} for details.

Somewhat less obviously, the story carries over
to \'etale
$(\varphi,\Gamma)$-modules over $\widetilde{\bA}_{\Qp}$ and $\widetilde{\sR}$.  Indeed, using the trace maps $\widetilde{\bA}_{\Qp} \to {\bA}_{\Qp}$ and $\widetilde{\sR} \to \sR$ (see \cite[Chapter 8]{Colmez-Monodromy} and \cite[\S 4.3]{Colm1}),
one can define $\res_0$ for these rings, and thus deduce pairings $\{\cdot ,\cdot\}$ from the
tautological pairing~$\langle \cdot, \cdot \rangle$, just as in 
the cases already considered above.

\subsection{Pairings:\ $p$-adic Hodge theory}
\label{sec:pair-hodge}

Recall that $\wt\bD_{\dif}(V) = (V \otimes_{\Qp} {\bf B}_{\dR})^H$ and 
that we can descend $\wt\bD_{\dif}(V)$ to an $L_\infty((t))$-module $\bD_{\dif}(V)$ where $L_\infty := \Kinfp \otimes_{\Qp} L$.  Further, for each $n \gg 0$, we can define a vector space $\bD_{\dif,n}(V)$ over $L_n((t))$, where $L_n := \Knp \otimes_{\Qp} L$ and $\Knp = \Q_{p}(\mu_{p^n})$. Namely, following the notation from~\cite[VI.3.1]{Colm2}, 
on $\mathbf D^{]0,r_n]} \subseteq \bD_{\rig}(V)$, there is a map $\imath_n : \mathbf D^{]0,r_n]} \to \wt\bD^+_{\dif}\left( V \right)$.  We denote the image of this map by $\mathbf D^+_{\dif,n}(V)$ and set $\mathbf D_{\dif,n}(V) := \mathbf D^+_{\dif,n}(V)[1/t]$.
(Then $\bD_{\dif}(V)$ is recovered by extending the scalars
of $\bD_{\dif,n}(V)$ from  $L_n((t))$  to $L_{\infty}((t))$.)

Then, as in \cite[VI.3.4]{Colm2}, there is a perfect $\Gamma$-equivariant pairing
$$
\langle \cdot , \cdot \rangle_{\dif}:
\mathbf D_{\dif,n}(V(1)) \times \mathbf D_{\dif,n}(V^*) \lra L
$$
which is independent 
of $n \gg 0$ (in the sense that, if $n' \geq n \gg 0$,
then the pairing on the $\mathbf D_{\dif,n'}$ restricts to the pairing on the~$\mathbf D_{\dif,n}$), 
and under which $\mathbf D^+_{\dif,n}(V(1))$ and $\mathbf D^+_{\dif,n}(V^*)$ are orthogonal complements.  

We set $\bD_{\dR,n}(V) = \wt\bD_{\dif}(V)^{\Gamma_n}$ equal to the de Rham Dieudonn\'e modules at cyclotomic level
$n$ where $\Gamma_n$ is the subgroup of $\Gamma$ of index $p^{n-1}(p-1).$  By \cite[VIII.2]{Colm2}, we have
$$
\bD_{\dR,n}(V) \cong \bD_{\dR,\Knp}(V) \cong \bD_{\dR,\Qp}(V) \otimes_{\Qp} \Knp
$$
where $\bD_{\dR,K}(V) := (V \otimes {\mathbf B}_{\dR})^{G_K}$ for any finite extension $K/\Qp$.

There is a $L_n$-valued natural pairing
$$
\langle \cdot,\cdot \rangle'_{\dR,n} : \bD_{\dR,\Knp}(V(1)) \times \bD_{\dR,\Knp}(V^*) \to 
\bD_{\dR,\Knp}(L(1)) \cong    L_n
$$
where the final isomorphism is defined by taking $t^{-1}$ as a basis of $\bD_{\dR,L_n}(L(1))$.
We then have an induced pairing
$$
\langle \cdot,\cdot \rangle_{\dR,n} : \bD_{\dR,L_n}(V(1)) \times \bD_{\dR,L_n}(V^*)  \to L
$$
defined by $\langle \cdot,\cdot \rangle_{\dR,n} = \Tr^{L_n}_L (\langle \cdot,\cdot \rangle'_{\dR,n})$.

We note that $\mathbf D_{\dR,n}(V) \subseteq \mathbf D_{\dif,n}(V)$ if $n \gg 0$,
and under this inclusion we have the relation $\frac{1}{p^n} \langle \cdot,\cdot \rangle_{\dR,n} = \langle \cdot,\cdot \rangle_{\dif}$ as the pairing on $\bD_{\dif,n}$ objects is defined via a normalized trace while on $\bD_{\dR,n}$ objects, the pairing is defined via an absolute trace.  

We close this subsection with a relation between $\imath_n$ and the dual exponential map
$$
\exp^* : H^1(\Q(\mu_{p^n}),V) \to \bD_{\dR,n}(V).
$$

\begin{lemma}
\label{lemma:inexp}
For $z \in \bD(V)^{\psi=1} \cong H^1_{\Iw}(V)$, write $z_n$ for the projection of $z$ to $H^1(\Qp(\mu_{p^n}),V)$.
We then have $p^n \exp^*(z_n)$ is the image of $\iota_m(z)$ modulo $\gamma_n-1$
for $m \gg 0$ where $\gamma_n$ is a topological generator of $\Gamma_n$.
\end{lemma}

\begin{proof}
This lemma is essentially \cite[Lemme VIII.2.1]{Colm2}, but we need to explain the factor of $p^n$ appearing.  The reason this factor appears is that the dual exponential used in \cite{Colm2} is normalized in a way that does not match the standard normalization as in \cite{Kato-exp}.

Let's write $\widetilde{\exp}^*$ for the dual exponential map normalized as in \cite{Colm2}.  Then for 
$x \in \mathbf D_{\dR,n}(V^*(1))$ and $y \in H^1(\Knp,V)$, by \cite[Th\'eor\`eme~VIII.2.2]{Colm2}, we have
$$
\langle x , \widetilde{\exp}^*(y) \rangle_{\dif} = \langle \exp(x), y \rangle_n
$$
where $\langle \cdot , \cdot \rangle_n$ is the pairing of Tate local duality.  Meanwhile, we also have that
$$
\langle x , \exp^*(y) \rangle_{\dR,n} = \langle \exp(x), y \rangle_n
$$
by \cite[Theorem 1.4.1(4)]{Kato-exp}.  Since $\frac{1}{p^n} \langle \cdot,\cdot \rangle_{\dR,n} = \langle \cdot,\cdot \rangle_{\dif}$, we deduce that $p^n \exp^* = \widetilde{\exp^*}$ and the lemma follows from 
\cite[Lemme VIII.2.1]{Colm2}.
\end{proof}

\subsection{An explicit reciprocity law}
Colmez has proven the following explicit reciprocity law \cite[Prop.~VI.3.4]{Colm2}.
(Recall that $\mathbf D_{\rig}$ is the union of the $\mathbf D^{]0,r_n]},$
and so for any element of $\mathbf D_{\rig}$, it is in the domain  of the
map $\imath_n: \bD^{]0,r_n]}  \to \widetilde\bD^+_{\dif}$ for $n \gg 0.$)

\begin{theorem}
\label{thm:Colmez explicit}
Let $z \in \bD^+_{\rig}\bigl(V(1)\bigr)[1/t]$
and let $z' \in \bD_{\rig}(V^{*})^{\psi = 1}.$
If $(1-\varphi)z$ lies in $\mathbf D_{\rig}^+\bigl(V(1)\bigr)$,
then the sequence $\langle \imath^-_n(z),\imath_n(z')\rangle_{\dif}$
of elements of $L$ becomes constant for $n\gg 0$, and this constant
value is equal to $\{(1-\varphi)z, \sigma_{-1} \cdot z'\}.$
\end{theorem}

We need the following variant of Theorem~\ref{thm:Colmez explicit}.

\begin{theorem}
\label{thm:explicit}
Let $z \in \widetilde\bD_{\rig}^+\bigl(V(1)\bigr)[1/t]$, $z' \in \bD_{\rig}(V^{*})^{\psi = 1}$.  Assume
\begin{enumerate}
\item $(1-\varphi)z \in \widetilde\bD^+(V(1)) \left[ 1/\varphi^r(T) \right]$ for some $r \in \Z^{\geq 0}$,
\item $z$ is fixed by $\Gamma_n$ for some $n$.
\end{enumerate}
Then 
$$
 \{(1-\varphi)z, \sigma_{-1} \cdot z'\} = \langle \imath^-_s(z),\imath_m(z')\rangle_{\dif} 
$$
for $s \geq r$ and $m$ large enough.
\end{theorem}

\begin{remark}
We explain the pairings in the statement of Theorem \ref{thm:explicit}.  First note that we can view $(1-\varphi)z \in \widetilde\bD^+(V(1))\left[ 1/\varphi^r(T) \right] \subseteq \widetilde{\bD}_{\rig}^{\dag}(V(1))[1/t]$ and $ z' \in \bD_{\rig}(V^{*}) \subseteq \widetilde{\bD}_{\rig}^{\dag}(V^*)$.  The pairing $ \{(1-\varphi)z,  \sigma_{-1} \cdot z'\}$ is thus the pairing between
$\widetilde{\bD}_{\rig}^{\dag}(V(1))$ and $\widetilde{\bD}_{\rig}^{\dag}(V^*)$
noted at the end of section \ref{sec:pair-phi}, as these are $(\varphi,\Gamma)$-modules defined over $\widetilde{\sR} \otimes_{\Qp} L$, 
extended to a pairing between
$\widetilde{\bD}_{\rig}^{\dag}(V(1))[1/t]$ and $\widetilde{\bD}_{\rig}^{\dag}(V^*)[1/t]$
in the evident way.  (The pairing $\langle  \, , \, \rangle$ now takes values 
in $\widetilde{\sR}[1/t] \otimes_{\Q_p} L,$ the normalized traces map this to
$\sR[1/t]\otimes_{\Q_p} L$, and the pairing $(f,g) \mapsto \res_0(fg \frac{dT}{1+T}) = \res_0 (fg dt)$
still makes sense.)  

To make sense of $\langle \imath^-_s(z),\imath_m(z')\rangle_{\dif}$ for $z \in \widetilde\bD_{\rig}^+\bigl(V(1)\bigr)[1/t]$ and for $m$ large enough,
note that since $z$ is fixed by $\Gamma_n$, we have
$$
\imath_s(z) \in \widetilde\bD_{\dif}(V(1))^{\Gamma_n} \subseteq \mathbf D_{\text{pdR},n}(V(1)) \subseteq \mathbf D_{\dif,n}(V(1)) \subseteq \mathbf D_{\dif,m}(V(1))
$$
as long as $m \geq n$ (see \cite[VI.3.4]{Colm2}).  As $\imath_m(z') \in \mathbf D_{\dif,m}(V^*)$, the pairing is defined as in section \ref{sec:pair-hodge}.
\end{remark}

\begin{proof}[Proof of Theorem \ref{thm:explicit}]
Following the proof of \cite[Prop.~VI.3.4]{Colm2}, mutatis mutandis, we find that
$$
 \{(1-\varphi)z,  \sigma_{-1} \cdot z'\} = \langle \imath^-_m(z),\imath_m(z')\rangle_{\dif} 
$$
for $m$ large enough.  (Note that the assumption in \cite[Prop.~VI.3.4]{Colm2} on $z$ being fixed by $\psi$ is only used in the second part of that proposition; the proof of the first part of that proposition does not use this assumption.) Further, since $(1-\varphi)z$ lies in $\widetilde\bD^+\bigl(V(1)\bigr)[1/\varphi^r(T)]$, by Proposition \ref{prop:support}, we have 
$\imath^-_s\bigl((1-\varphi)z\bigr)=0$ for $s > r$.  Thus,
$$
\imath^-_s(z) = \imath^-_s(\varphi z) = \imath^-_{s-1}(z) = \dots = 
\imath^-_{r+1}(z) = \imath^-_{r+1}(\varphi z) = \imath^-_r(z)
$$ 
as desired.
\end{proof}

\section{Kirillov models and local Mellin transforms}
This section is an interlude, in which we recall some
basic facts regarding Kirillov models and local Mellin transform
formulas for local Euler factors.
We begin with the case of $\GL_2(\mathbb Q_p)$-representations 
with complex coefficients
before explaining the modifications necessary for treating
the case of representations with $p$-adic coefficients.

\subsection{Kirillov models of complex
smooth $\GL_2(\mathbb Q_p)$-representations}
If $\pi$ is an infinite-dimensional irreducible representation
 of $\GL_2(\mathbb Q_p)$ over $\mathbb C$, its Kirillov model  
  identifies it  with a space of locally constant
functions on $\mathbb Q_p^{\times}$, each of which is compactly
supported as a function on $\mathbb Q_p$.   (More precisely, the closure 
in $\Q_p$ of the support of each function in the Kirillov model is compact,
or, equivalently, the support in $\mathbb Q_p^{\times}$ of each
such function is bounded away
from infinity.)  The entire $\GL_2(\mathbb Q_p)$-action on this space of functions
is a bit tricky to describe, but the action of $B$ is quite explicit.  Firstly,
the centre acts by  the central character of~$\pi$.  
Secondly, the subgroup $\begin{pmatrix}\mathbb Q_p^{\times} & 0 \\ 0 & 1
\end{pmatrix}$ acts by scaling the variable.  
Finally, if $b \in \mathbb Q_p$,
then $\Bigl( \begin{pmatrix} 1 & b \\ 0 & 1 \end{pmatrix} \phi \Bigr)
(x) = e^{2 \pi i b x} \phi(x).$
(Note that $e^{2 \pi i y}$ makes sense for $y \in \mathbb Q_p$, via the natural isomorphism
$\mathbb Q_p/\mathbb Z_p \iso (\mathbb Q/\mathbb Z)[p^{\infty}].$)
This can be summarized by the following formula:
$$\Bigl(\begin{pmatrix} a & b \\ 0 & d \end{pmatrix} \phi\Bigr)(x)
= \delta(d) e^{2 \pi i (b/d) x} \phi\bigl((a/d)x\bigr),
$$
where $\delta$ is the central character of $\pi$.

An important aspect of the theory is that the realization 
of $\pi$ in the Kirillov model is unique (up to rescaling by a non-zero
constant).
One way to phrase and prove this is as follows:
if $\phi$ is an element of $\pi$, thought of in the Kirillov model,
the map $\phi \mapsto \phi(1)$
 gives a non-zero linear functional
$\ell:\pi \rightarrow \mathbb C$,
with the property that 
$\ell\Bigl( \begin{pmatrix} 1 & b \\ 0 & 1 \end{pmatrix} v \Bigr)
= e^{2 \pi i b } \ell(v)$
for all $v \in \pi.$  The uniqueness of Kirillov models can then
be stated as saying that (when $\pi$ is irreducible and infinite-dimensional)
the space of such $\ell$ is one-dimensional, hence $\ell$ is uniquely
determined up to a non-zero scalar.  Note that we can then recover the
Kirillov function $\phi_v$ attached to an element $v \in \pi$ via the formula
$$\phi_v(a) = 
\ell\Bigl( \begin{pmatrix} a & 0 \\ 0 & 1 \end{pmatrix} v \Bigr).$$

As an example of Kirillov models,
suppose that $\pi$ is the principal series
\begin{multline*}
\Ind_{\Bbar}^{\GL_2(\mathbb Q_p)} \chi_1\otimes\chi_2\\
= \{f : \GL_2(\mathbb Q_p) \rightarrow \mathbb C \, | \, f(\overline{n} t g)
= (\chi_1\otimes\chi_2)(t) f(g) \text{ for all } g \in \GL_2(\mathbb Q_p),
\overline{n} \in \Nbar, t \in T\}  
\end{multline*}
(where $\Nbar$ denotes the lower triangular unipotent subgroup
of $\GL_2(\mathbb Q_p),$ $T$ denotes the diagonal torus,
and $\Bbar = \overline{N} T$ denotes the lower triangular Borel),
equipped with the right regular $\GL_2(\mathbb Q_p)$-action.
We may regard $\pi$ as the space of locally constant functions
on $\mathbb Q_p$ which are proportional to $\chi_1/\chi_2$ in a
neighbourhood of infinity, via
$$f(x) := f\left( \begin{pmatrix} 1 & x \\ 0 & 1 \end{pmatrix}\right);$$
the $\GL_2(\mathbb Q_p)$-action on $\pi$ is then given by the formula
$$
\left(\begin{pmatrix} a & b \\ c & d \end{pmatrix} f\right)\bigl(x\bigr)
=
\frac{\chi_1}{\chi_2}(c x + a) \chi_2(a d - b c)
f\left( \dfrac{d x + b}{c x + a}\right).$$
Forming the Kirillov model of $\pi$ in this case amounts to taking
the Fourier transform:\ we map a function $f$ in the principal
series to (a rescaling of) its Fourier transform, namely to
$$\phi(x) := | x |_p \chi_1(x) \int_{\mathbb Q_p} f(y)e^{-2\pi i x y} dy.$$
(For definiteness, we normalize Haar measure on $\Q_p$ so that
$\Z_p$ has measure~$1$.)

\subsection{Local newvectors}
\label{sec:newvec}
If $\pi$ is an irreducible infinite-dimensional
smooth representation of $\GL_2(\mathbb Q_p)$,
then the newvector in $\pi$ (which is well-defined up to a scalar)
is the vector
fixed by $\begin{pmatrix} \mathbb Z_p^{\times} & 
\mathbb Z_p \\ 0 & 1 \end{pmatrix}$
which is of smallest conductor,
{\it i.e.}\ which is also fixed under $\begin{pmatrix} 1 & 0
\\ p^n \mathbb Z_p & 1\end{pmatrix}$ for the smallest possible choice of $n$.

We will also consider the $p$-deprived newvector of $\pi$ (which is 
again well-defined up to a scalar); this is any element of
$\pi$ which is
fixed by $\begin{pmatrix} \mathbb Z_p^{\times} & 
\mathbb Z_p \\ 0 & 1 \end{pmatrix}$
and which is annihilated by the operator
$$U_p := \sum_{i = 0}^{p-1} \begin{pmatrix} p & i \\ 0 & 1 \end{pmatrix}.$$
Note that newvector and $p$-deprived newvector of $\pi$ coincide
if and only if the local Euler factor attached to $\pi$ is trivial.

As already noted, both the Kirillov model of $\pi$ and the newvector
(or the $p$-deprived newvector) of $\pi$ are only determined up to a non-zero scalar.
However, given a choice of Kirillov model, we can fix the newvector
precisely by requiring that the corresponding function in the Kirillov model takes 
the value $1$ at the point $1 \in \mathbb Q_p^{\times}$.
With this normalization,
the $p$-deprived newvector corresponds in the Kirillov model to 
the characteristic function $\triv_{\mathbb Z_p^{\times}}.$

If $\pi$ is the unramified principal series $\Ind_{\Bbar}^{\GL_2(\mathbb Q_p)}
\chi_{\alpha}\otimes\chi_{\beta/p}$ with $\alpha, \beta \in \C$ and 
$\chi_{\lambda}$ the unramified character taking $p$ to $\lambda$
(and $\pi$ is assumed to be irreducible; equivalently,
we assume that $\alpha/\beta \neq p^{\pm 1}$),
then the newvector $v_{\new}$, as a function on~$\Q_p$,
coincides (up to scaling) with the function
$$v_{\new} = (1 - \beta/(\alpha p))^{-1} \bigl( \triv_{\mathbb Z_p} +
(\chi_{\alpha p/\beta})\triv_{\mathbb Q_p\setminus \mathbb Z_p}\bigr).$$
Converting to the Kirillov model (an exercise in taking Fourier transforms)
one finds that in the Kirillov model (so now as a function on 
$\mathbb Q_p^{\times}$), 
one has
$$
v_{\new}
=
\sum_{n = 0}^{\infty}  \dfrac{1}{p^n} \dfrac{\alpha^{n+1}
-\beta^{n+1}}{\alpha - \beta} \triv_{p^n \Z_p^{\times}}
=
\sum_{n = 0}^{\infty}  \dfrac{(\alpha/p)^{n+1}
-(\beta/p)^{n+1}}{(\alpha/p) - (\beta/p)} \triv_{p^n \Z_p^{\times}}.
$$
(This is the $\GL_2(\mathbb Q_p)$ case of a formula of 
Casselman and Shalika.)

\subsection{The local Birch Lemma}
\label{sec:localbirch}
The goal of this subsection is to explain how to describe the $p$-deprived newvector
of a ramified twist of a smooth $\GL_2(\mathbb Q_p)$-representation $\pi$
in terms of the newvector of $\pi$.

If $\chi:\mathbb Q_p^{\times} \rightarrow \mathbb C^{\times}$
is a character, then $f \mapsto f\chi$ gives a bijection between
the Kirillov model of $\pi$ and the Kirillov model of the twist
$\pi_{\chi} := (\chi\circ \det)\otimes \pi.$  In particular, if $\phi$ in Kirillov model corresponds to the newvector of $\pi$
 (normalized as above so that $\phi(1) = 1$), 
then $\phi\chi$ is a function in the Kirillov model
of $\pi_{\chi}$ (satisfying $(\phi\chi)(1) = 1$).  

Let $p^n$ be the conductor of $\chi|_{\mathbb Z_p^{\times}}.$
If $n = 0,$ {\it i.e.}\ if $\chi$ is unramified, then $\phi\chi$ is the newvector
of $\pi_{\chi}$.   However, if $n > 0$, {\it i.e.}\ if $\chi$ is ramified,
then $\phi\chi$, although it is invariant under 
$\begin{pmatrix} 1 & \mathbb Z_p \\ 0 & 1\end{pmatrix}$,
is not invariant under $\begin{pmatrix} \mathbb Z_p^{\times} & 0 \\ 0 & 1
\end{pmatrix},$ and so cannot be the newvector.
We remedy this by forming a Gauss sum.

\begin{lemma}
\label{lemma:Birch}
With the notation as above, if $\chi|_{\mathbb Z_p^{\times}}$ has conductor $p^n > 1$, we have
$$
\sum_{a \in (\mathbb Z/p^n\mathbb Z)^{\times}}
\chi(a) \begin{pmatrix} 1 & a/p^n \\ 0 & 1 \end{pmatrix} (\phi\chi)
= \tau(\chi) \triv_{\mathbb Z_p^{\times}}
$$
as elements of the Kirillov model of $\pi_\chi$.  Here $\tau(\chi)$ is the usual Gauss sum attached to $\chi$.
\end{lemma}

\begin{proof}
Evaluating the left hand side at $x$ gives
\begin{align*}
\sum_{a \in (\mathbb Z/p^n\mathbb Z)^{\times}}
&\chi(a) \begin{pmatrix} 1 & a/p^n \\ 0 & 1 \end{pmatrix} (\phi\chi)(x)\\
&= 
\left( \sum_{a \in (\mathbb Z/p^n\mathbb Z)^{\times}}
\chi(ax)  e^{2\pi i ax/p^n} \right) \phi(x) = 
\begin{cases}
\tau(x) \phi(x) & x \in \Zp^\times, \\
0 & {\rm otherwise}. 
\end{cases}
\end{align*}
Since $\phi(x)=1$ when $x$ is a unit, we have proven the lemma.
\end{proof}

Thus the expression on the right hand side of Lemma \ref{lemma:Birch}
is equal to $\tau(\chi)$ times the $p$-deprived newform
of $\pi_{\chi}$.  
(Note that, since $\chi$ is ramified, typically the $p$-deprived
newform of $\pi_\chi$ coincides with the newform itself; the only 
exceptions occur for very particular choices of $\chi$ when $\pi$
is a ramified principal series or a ramified twist of Steinberg.)
The formula in Lemma \ref{lemma:Birch}
 is a local version of the formula known as the Birch Lemma
in the theory of modular symbols.

\subsection{Local Mellin transforms}
Let $\pi$ be an infinite dimensional irreducible representation
of $\GL_2(\mathbb Q_p)$ over $\mathbb C$, and let
$\phi \in \Con^{\sm}(\mathbb Q_p^{\times},\mathbb C)$
be an element of $\pi$, thought of in the Kirillov model.
The local Mellin transform of $\phi$ is the integral
$$Z(\phi,s) := \int_{\Q_p^{\times}} |x|^{s-1} \phi(x) d^{\times}x.$$
(We normalize the Haar measure $d^{\times}x$
on $\mathbb Q_p^{\times}$ so that $\mathbb Z_p^{\times}$ has 
unit measure.)
{\em A~priori}, unless $\phi$ is compactly supported on $\mathbb Q_p^{\times}$,
this integral only converges when the real part of $s$
is sufficiently large.  However, for a fixed choice of $\phi$, the convergent values of the integral are a rational function of $p^{-s}$, and so the
integral can be defined for any value of $s$ in terms of this rational
function.

These integrals come up in Jacquet--Langlands \cite{JL} in the description of local
Euler factors and epsilon factors; in particular, taking $\phi$ to be
the newvector gives the local Euler factor.
For example, in the case of a principal series, taking $\phi$ to be the newvector,
we find that this integral is a sum of geometric series, 
which we may formally evaluate to obtain the value
$$(1-\alpha p^{-s})^{-1}(1 - \beta p^{-s})^{-1}.$$
Note also that the local Euler factor is the ``largest'' denominator that
can occur in any $Z(\phi,s)$; more precisely, it is divisible by the denominator 
of  $Z(\phi,s)$ for any function $\phi$ in the Kirillov model.

We now describe an alternative approach to evaluating the Mellin transform
integral.

\begin{prop} 
\label{prop:Mellin}
Suppose that the local Euler factor 
attached to $\pi$
does not have a pole at $s = 1$.  
Then given any function $\phi$ in the Kirillov model of $\pi$,
we may find a uniquely determined function $\xi$ in 
$\Con^{\sm}(\Q_p^{\times},\mathbb C)$
such that:
\begin{enumerate}
\item \label{part:1}
The restriction of $\xi$ to (the intersection with $\mathbb Q_p^{\times}$ with) 
any compact neighbourhood of $0$ in $\mathbb Q_p$ lies in the Kirillov
model of $\pi$.
\item \label{part:2} $\left(1 - \begin{pmatrix} p & 0 \\ 0 & 1 \end{pmatrix}\right) \xi = \phi.$
\end{enumerate}
If $\phi$ is $\Gamma$-invariant, then so is $\xi$.
Furthermore, one then has that
$$Z(\phi,1)  = \int_{p^{-i} \Z_p^{\times}} \xi(x) d^{\times} x =
\int_{\Z_p^{\times}} \xi(p^{-i}x)d^{\times} x, $$
for any $i$ chosen so large that
$\phi$ vanishes outside of $p^{-i}\Z_p$.
\end{prop}

\begin{proof}
We first consider the case where $\phi$ has support bounded away from 0.  In this case, we define $\displaystyle \xi = \sum_{n \geq  0} \psmallmat{p^n}{0}{0}{1} \phi$.  Note that for a fixed $x \in \Qp^\times$, $\xi(x)$ is a finite sum.  Since $\phi$ has support bounded away from $0$, the same is true of $\xi$, proving (\ref{part:1}).  Further (\ref{part:2}) follows immediately from the definition of $\xi$,
as does the fact that $\xi$ is $\Gamma$-invariant if $\phi$ is.  Lastly, for $i$ large enough so that $\phi$ vanishes outside of $p^{-i}\Zp$, we have
$$
\int_{\Z_p^{\times}} \xi(p^{-i} x)d^{\times} x
= \sum_{n \geq 0} \int_{\Z_p^{\times}} \phi(p^{n-i} x) d^{\times} x
= \int_{p^{-i}\Z_p \setminus \{0\}} \phi(x) d^{\times} x 
= \int_{\Q_p^{\times}} \phi(x) d^{\times} x
$$
as desired.

For the general case, we may write $\phi = \phi^- + \phi^+$ where $\phi^+$ is bounded away from $0$, and $\phi^-$ is a linear combination of functions supported on $\Zp$ of the form $x \mapsto \gamma^{\ord_p(x)}$ for some $\gamma \in \C$ (see {\it e.g}.\ \cite[Chapter 6]{GH}).  Note further that our assumption on the local Euler factor forces that $\gamma \neq 1$.  
To complete the proof, we are thus reduced to the case where $\phi(x) = \begin{cases}\gamma^{\ord_p(x)} & x \in \Zp, \\ 0 & \text{otherwise.} \end{cases}$

In this case, we define $\xi(x) = \begin{cases} \frac{\gamma^{\ord_p(x)}}{1-\gamma} & x \in \Zp, \\ \frac{1}{1-\gamma} & \text{otherwise}.
\end{cases}$  Clearly, $\xi$ satisfies (\ref{part:1}) as near zero it is of the form $x \mapsto D \gamma^{\ord_p(x)}$ for some constant $D$.  To check (\ref{part:2}), for $x \in \Zp$, we have
$$
\left( 1 - \psmallmat{p}{0}{0}{1} \right) \xi(x) = \xi(x) - \xi(px) = \frac{\gamma^{\ord_p(x)}}{1-\gamma} - \frac{\gamma^{\ord_p(x)+1}}{1-\gamma} = \gamma^{\ord_p(x)} = \phi(x)
$$
while for $x \not\in \Zp$, we have
$$
\left( 1 - \psmallmat{p}{0}{0}{1} \right) \xi(x) = \xi(x) - \xi(px) = \frac{1}{\gamma-1} - \frac{1}{\gamma-1} = 0 = \phi(x)
$$
as desired.

We note that from our explicit construction of $\xi$, it is clear that if $\phi$ is invariant under $\Gamma$, then so is $\xi$. Further, regarding the uniqueness of $\xi$, if $\psi$ satisfies $
\left(\psmallmat{p}{0}{0}{1}-1\right)\psi = 0$, then $\psi(px) = \psi(x)$ for all $x$ and thus $\psi$ is not bounded away from infinity unless $\psi=0$.

Lastly, we compute that
$$
Z(\phi,s) = \int_{\Qp^\times} |x|^{s-1} \phi(x) d^\times x  
= \sum_{n=0}^\infty p^{-n(s-1)} \gamma^n = \frac{1}{1 - \frac{\gamma}{p^{s-1}}}
$$
for $s \gg 0$.  Thus, $Z(\phi,1) = \frac{1}{1-\gamma}$.  Finally, for $i>0$, 
$$
\int_{\Zp^\times} \xi(p^{-i}x) d^\times x = 
\int_{\Zp^\times} \frac{1}{1-\gamma} d^\times x = \frac{1}{1-\gamma} = Z(\phi,1),
$$
completing the proof.
\end{proof}

\subsection{Kirillov models of smooth representations of $\GL_2(\Q_p)$
with $p$-adic coefficients.}
\label{sec:smoothp-adic}
We now consider the theory of
Kirillov models for infinite dimensional irreducible smooth representations  $\pi$
of $\GL_2(\mathbb Q_p)$ having
coefficients in $L$, a finite extension of $\mathbb Q_p$. 
The main difference  between this case and the case of complex coefficients is one of 
rationality:\ the $p$-power roots of unity are not all in $L$.   

Recall $\Linf:=\Kinfp\otimes_{\mathbb Q_p} L$. If we let $\Con^{\sm}(\mathbb Q_p^{\times},L)$
(resp.\ $\Con^{\sm}(\mathbb Q_p^{\times},
\Linf)$)
denote the space of locally constant 
functions on $\mathbb Q_p^{\times}$ with values in $L$
(resp.\ $\Linf$),
then the Kirillov model will be an embedding
$$\pi \hookrightarrow \Kinfp\otimes_{\mathbb Q_p} 
\Con^{\sm}(\mathbb Q_p^{\times},L)
\subseteq 
\Con^{\sm}(\mathbb Q_p^{\times}, \Linf),
$$
which is $B$-equivariant for a certain $B$-action on the target.

To describe this $B$-action, we have to replace our complex valued additive character
on $\mathbb Q_p$ by a $\Kinfp$-valued one.
Recall that we have chosen a generator $\varepsilon$ of $\mathbb Z_p(1)$;
write $\varepsilon = (\varepsilon^{(n)}),$ where $\varepsilon^{(n)}$ is a primitive
$p^n$th root of unity, and $(\varepsilon^{(n+1)})^p = \varepsilon^{(n)}$.
We will replace the complex valued character $e^{2 \pi i x}$ of $\mathbb Q_p$ by the
$\Qbar_p$-valued character $$\varepsilon(x) := \lim_{n \to \infty}
(\varepsilon^{(n)})^{p^n x}.$$  (Note that the sequence $(\varepsilon^{(n)})^{p^n x}$
eventually stabilizes, so that the limit is just equal to this stable value.)
We then define the $B$-action on 
$ \Con^{\sm}(\mathbb Q_p^{\times}, \Linf) $
via the formula
$$\Bigl(\begin{pmatrix} a & b \\ 0 & d \end{pmatrix} \phi\Bigr)(x)
= \delta(d) \varepsilon\bigl( (b/d) x \bigr) \phi\bigl((a/d)x\bigr),
$$
where, as before, $\delta$ denotes the central character of $\pi$.
We may also define a Galois-twisted action of $\Gamma$ on
$ \Con^{\sm}(\mathbb Q_p^{\times}, \Linf)$
via the formula
$(a\cdot \phi)(x) = \sigma_a\bigl(\phi(a^{-1} x)\bigr).$
(Here the Galois element $\sigma_a$ acts on $\Linf$
via its action on the first factor in the tensor product which we normalize so that $\sigma_a(\zeta)=\zeta^a$ for $\zeta \in \mu_{p^\infty}$.)
Using the fact that $\pi$ is defined over~$L$, together with the uniqueness of
Kirillov models,
one then shows that the image of the Kirillov model lies in
$\bigl(\Kinfp\otimes \Con^{\sm}(\mathbb Q_p^{\times},L)\bigr)^{\Gamma}$,
where the invariants are computed with respect to the Galois-twisted action.

Note that the Kirillov functional given by evaluation at $1 \in \mathbb Q_p^{\times}$
is now a non-zero linear functional
$$\ell: \pi \rightarrow \Linf,$$
such that
$$\ell\Bigl( \begin{pmatrix} 1 & b \\ 0 & 1 \end{pmatrix} v \Bigr)
= \varepsilon(b) \ell(v),$$
and such that
$$\ell\Bigl( \begin{pmatrix} a & 0 \\ 0 & 1 \end{pmatrix} v \Bigr)
= \sigma_a \bigl(\ell(v)\bigr),$$
for all $a \in \Gamma$.
The second property of $\ell$ 
shows that (unlike in the complex case), in the $p$-adic setting
the functions in the Kirillov model of $\pi$ are determined by their values
on the elements $p^i$ ($i \in \mathbb Z$) of $\mathbb Q_p^{\times}$.
Thus, in the following, we will frequently consider the functions
in the Kirillov model just as functions on the powers $p^i$
(and lose no information by so doing).

\subsection{Local Mellin transforms with $p$-adic coefficients}
Suppose that $\pi$ is defined over $L,$
and let $$\phi \in \Con^{\sm}(\mathbb Q_p^{\times},\Linf)^{\Gamma}$$
be an element of $\pi$, thought of via its Kirillov model.
(Here $\Gamma$-invariance of $\phi$ refers to invariance under the Galois-twisted
$\Gamma$-action, as discussed above.)

When $s$ is an integer, we can write down the local Mellin transform integral, and express it as a rational function
of $p^{-s}$.  However, we can further simplify the integral in this
context to be just a sum over powers of $p$
(just as the Kirillov model can be regarded simply as functions
on the set of powers of $p$, rather than on all of $\mathbb Q_p^{\times}$),
as we now explain.

Let $\Tr:\Kinfp \rightarrow \mathbb Q_p$
be the normalized trace, {\it i.e.}\ its restriction to $\Knp := \Qp(\mu_{p^n})$ is given by
$\dfrac{1}{(p-1)p^{n-1}}\Tr^{\Knp}_{\mathbb Q_p}.$
We then have the following formula, easily proved using the Galois-twisted
$\Gamma$-invariance of $\phi$:
$$\int_{\mathbb Z_p^{\times}} \phi(x) d^{\times}x = (\Tr \otimes\id_L)\bigl( \phi(1) \bigr).$$
From this, we deduce the corresponding formula for the local Mellin transform:
\begin{equation*}
\int_{\mathbb Q_p^{\times}} \phi(x) d^{\times}x =
(\Tr \otimes\id_L)\bigl(
\sum_{n=-\infty}^{\infty} \phi(p^n) \bigr).
\end{equation*}
In light of this formula, we regard
the expression
\begin{equation}
\label{eqn:mellin}
\widetilde{Z}(\phi,1)
=
\sum_{n=-\infty}^{\infty} \phi(p^n)
  \in \Linf
\end{equation}
(which takes values in $\Knp\otimes_{\mathbb Q_p}L$
if $\phi$ is invariant under the (regular, not Galois-twisted) action of $1 +p^n \Z_p$)
as being a refined Mellin transform.

We then have the following analogue of Proposition~\ref{prop:Mellin}.

\begin{prop} 
\label{prop:Mellinpadic}
Suppose that the local Euler factor 
attached to $\pi$
does not have a pole at $s = 1$.  
Then given any function $\phi$ in the Kirillov model of $\pi$,
we may find a uniquely determined function $\xi$ in 
$\Con^{\sm}(\mathbb Q_p^{\times},\Linf)^{\Gamma}$
such that:
\begin{enumerate}
\item 
The restriction of $\xi$ to (the intersection with $\mathbb Q_p^{\times}$ with) 
any compact neighbourhood of $0$ in $\mathbb Q_p$ lies in the Kirillov
model of $\pi$.
\item $\left(1 - \begin{pmatrix} p & 0 \\ 0 & 1 \end{pmatrix}\right) \xi = \phi.$
\end{enumerate}
If $\phi$ is invariant under the regular $\Gamma$-action, then so is $\xi$. 
Furthermore, one then has that
$\tilde{Z}(\phi,1)  = \xi(p^{-i})$, 
for any $i$ chosen so large that
$\phi$ vanishes outside of $p^{-i}\Z_p$.  
\end{prop}

\begin{proof}
The proof of this proposition proceeds along the same lines as that of Proposition \ref{prop:Mellin}.  
\end{proof}

\subsection{The case of locally algebraic representations}
\label{sec:localg}

In what follows, we are closely following \cite[VI.4]{Colm2}.
Namely, we now consider Kirillov models of locally algebraic representations ({\it e.g.}\ representations of the form $\pi \otimes \Sym^{k-2}(L^2) \otimes \det^{\ell}$ with $\pi$ a smooth representation).  Here we view $\Sym^{k-2}(L^2)$ as homogenous polynomials in $e_1$ and $e_2$ of degree $k-2$ where $\GL_2(\Qp)$ acts by 
$$
\begin{pmatrix} a & b \\ c & d \end{pmatrix} e_1 = ae_1 + ce_2 \text{~and~}
\begin{pmatrix} a & b \\ c & d \end{pmatrix} e_2 = be_1 + de_2.
$$

In the smooth $p$-adic case (section \ref{sec:smoothp-adic}), our Kirillov models were smooth functions taking values in $\Linf$.  We now instead consider larger spaces of functions.

Define $\LP^{[\ell,\ell+k-2]}(\Qp^\times,t^\ell L_\infty[t] / t^{k-1+\ell} L_\infty[t])$ to be the space of locally 
polynomial functions on $\Qp^\times$ valued in $t^\ell L_\infty[t] / t^{k-1+\ell} L_\infty[t]$.  We equip this space with an action of $B$ via 
$$
\left( \begin{pmatrix} a & b \\ 0 & d \end{pmatrix} \phi \right)(x)  = \delta(d) [(1+T)^{bx/d}] \phi(ax/d)
$$
where $\delta$ is the central character of our locally algebraic representation, 
and where $[(1+T)^z]$ is the character of $\Q_p$ discussed in~\cite[I.1.6]{Colm2}.
The simplest way 
to describe it in our present context  is to use the formula
$[(1+T)^{z}] = \ve(z) e^{tz}$ (see {\em loc.\ cit.}).

In \cite[Prop VI.2.11]{Colm2}, it is proven that $\pi \otimes \Sym^{k-2}(L^2) \otimes \det^{\ell}$ admits a locally algebraic Kirillov model; that is, it admits a $B$-equivariant map 
\begin{align*}
\pi \otimes \Sym^{k-2}(L^2) \otimes {\det}^{\ell} &\lra \LP^{[\ell,\ell+k-2]}(\Qp^\times,t^\ell L_\infty[t] / t^{k-1+\ell}L_\infty[t]).\\
z &\mapsto {\mathcal K}_z
\end{align*}
Explicitly, fix a Kirillov model of $\pi$ where we write $\phi_v$ for the function corresponding to $v \in \pi$.  Then for 
$$
z = \sum_{i=0}^{k-2} z_i \otimes e_1^{k-2-i} e_2^i \in \pi \otimes \Sym^{k-2}(L^2) \otimes {\det}^{\ell}
$$
we set 
\begin{equation}
\label{eqn:localg}
{\mathcal K}_z(x) = (tx)^\ell \sum_{i=0}^{k-2} i!~\!\phi_{z_i}(x) (tx)^{k-2-i}.
\end{equation}

It is easy to see that for $a \in \Zp^\times$, we have $\sigma_a({\mathcal K}_z(x)) = {\mathcal K}_z(ax)$ where $\sigma_a(t) = at$ as usual.  Thus, as before, ${\mathcal K}_z$ is invariant under the Galois-twisted action of $\Gamma$. In particular, all of the values of ${\mathcal K}_z(x)$ are determined by its values on powers of $p$.

\begin{prop} 
\label{prop:Mellinpadic_localg}
Suppose that the local Euler factor 
attached to $\pi$
does not have a pole at $s = j+3 -\ell -k$.  
Then given $$z = v \otimes e_1^{k-2-j} e_2^j \in \pi \otimes \Sym^{k-2}(L^2) \otimes {\det}^{\ell},$$
we may find a uniquely determined function $\xi$ in 
$\LP^{[\ell,\ell+k-2]}(\Qp^\times,t^\ell L_\infty[t] / t^{k+\ell-1} L_\infty[t])^\Gamma$ such that:
\begin{enumerate}
\item 
The restriction of $\xi$ to (the intersection with $\mathbb Q_p^{\times}$ with) 
any compact neighbourhood of $0$ in $\mathbb Q_p$ lies in the Kirillov
model of $\pi \otimes \Sym^{k-2}(L^2) \otimes {\det}^{\ell}$.
\item $\left(1 - \begin{pmatrix} p & 0 \\ 0 & 1 \end{pmatrix}\right) \xi = {\mathcal K}_z$.
\end{enumerate}
If $\phi_{v}$ is invariant under the regular $\Gamma$-action, then so is $\xi$. 
Furthermore, one then has that
$\widetilde{Z}({\mathcal K}_z,1)  = \xi(p^{-i})$, 
for any $i$ chosen so large that
$\phi_{v}$ vanishes outside of $p^{-i}\Z_p$.  
\end{prop}

\begin{proof}
First note that we have 
$$
{\mathcal K}_z(x) = (tx)^\ell j!~\!\phi_{v}(x) (tx)^{k-2-j} = j! (tx)^{k-2+\ell-j} \phi_{v}(x).
$$
If $\phi_v$ is bounded away from 0, then we can argue exactly as in Proposition \ref{prop:Mellin} and define $\displaystyle \xi = \sum_{n \geq 0} \left( \begin{smallmatrix} p^n & 0 \\ 0 & 1 \end{smallmatrix}\right) {\mathcal K}_z$.  Otherwise, again as in Proposition \ref{prop:Mellin}, we can reduce to the case where $\phi_v(x) = \begin{cases}\gamma^{\ord_p(x)} & x \in \Zp, \\ 0 & \text{otherwise.} \end{cases}$
In this case, we define $\displaystyle \xi(x) = \begin{cases}  \displaystyle\frac{(tx)^a \gamma^{\ord_p(x)}}{1-p^a\gamma} & x \in \Zp, \vspace{0.2cm}\\ 
\displaystyle\frac{(tx|x|)^a }{1-p^a\gamma} & x \not\in \Zp,
\end{cases}$
where $a=k-2+\ell-j$.  A direct computation with this function verifies the remaining claims of the proposition.  We also note that to define $\xi$, we need that $\gamma \neq p^{-a} = p^{2+j-k-\ell}$ which is implied by our condition on the local Euler factor of $\pi$.
\end{proof}

\section{$p$-adic local Langlands correspondence and Kirillov models}
\label{sec:Kirillov}

In this section we will turn to the consideration of Colmez's work
on the $p$-adic local Langlands correspondence.  This work uses the constructions
of the preceding sections in the particular case of two-dimensional
$p$-adic representations of $G_{\Qp}$
to associate a $\GL_2(\mathbb Q_p)$-Banach space
representation to any such two-dimensional representation. 
For us, one of the most important features of Colmez's work
will be the very powerful analogue of the classical
theory of Kirillov models for smooth representations of $\GL_2(\mathbb Q_p)$
that he develops using the maps $\imath^-_i$ and~$\imath^-$ introduced
above. 

\subsection{Kirillov models via $(\varphi,\Gamma)$-modules}

\label{subsec:LL}

If $V$ is a continuous irreducible two-dimensional
representation of $G_{\mathbb Q_p}$ over some finite extension $L$
of $\mathbb Q_p$,
then Colmez
associates to $V$ an $L$-Banach space representation $\pi(V)$ of
$\GL_2(\mathbb Q_p)$.  (We use slightly different conventions from Colmez; he would call this $\pi(V(1))$.)  We won't recall the construction of $\pi(V)$ as a
$\GL_2(\mathbb Q_p)$-representation here.  Rather, we will focus on 
$\pi(V)$ just as a $P$-representation (recall $P$ is the mirabolic subgroup),
in which case $\pi(V)$ admits the following description:\
there is a canonical isomorphism of $P$-modules
\begin{equation*}
\widetilde\bD\bigl(V(1)\bigr)/\widetilde\bD^+\bigl(V(1)\bigr)
\iso \pi(V),
\end{equation*}
where the source has the $P$-representation structure discussed in the
section \ref{sec:mirabolic} (see \cite[Corollaire II.2.9]{Colm2}).

Suppose now that $V$ admits zero as a Hodge--Sen--Tate weight with multiplicity one.
If we let 
$
(\widetilde\bD\bigl(V(1)\bigr)/\widetilde\bD^+\bigl(V(1)\bigr))_{
P-\sm}
$ 
denote the subspace of $P$-smooth vectors of the $P$-representation
$\widetilde\bD\bigl(V(1)\bigr)/\widetilde\bD^+\bigl(V(1)\bigr)$,
and if we choose an $L$-basis $\mathbf e$ for the one-dimensional space of $\Gamma$-invariant vectors in $\widetilde\bD_{\Sen}$,
then we see that
$\imath^-_0$ restricts to a map 
\begin{equation}
\label{eqn:third}
\imath^-_0:
\Bigl(\widetilde\bD\bigl(V(1)\bigr)/
\widetilde\bD^+\bigl(V(1)\bigr)\Bigr)_{P-\sm}
\to
\Linf{\mathbf e},
\end{equation}
with the following two properties:\ for any $n \geq 0$ and any $b \in \mathbb Z_p,$
$$
\imath^-_0\left(\varphi^{-n}(1+T)^b x\right) = 
(\varepsilon^{(n)})^b \imath^-_0(x),
$$
or, rewriting this in terms of the $P$-action,
\begin{equation}
\label{eqn:fourth}
\imath^-_0\left(\begin{pmatrix}1 & b/p^n \\ 0 & 1 \end{pmatrix} x\right) = 
(\varepsilon^{(n)})^b \imath^-_0(x);
\end{equation}
and for any $a \in \Gamma$,
$$\imath^-_0(a \cdot x) = \sigma_a\bigl(\imath^-_0(x)\bigr),$$
or, again rewriting this in terms of the $P$-action,
\begin{equation}
\label{eqn:fifth}
\imath^-_0\left(\begin{pmatrix} a & 0 \\ 0 & 1 \end{pmatrix} x\right)
= \sigma_a\left(\imath^-_0(x)\right).
\end{equation}

We suppose now (and for the remainder of this section) that $V$ is furthermore de Rham,
with Hodge--Tate weights equal to $0$ and $1 - k$, for some
$k\geq 2$ (so that the jumps in the Hodge filtration of
$\bD_{\dR}(V)$ occur at $0$ and $k-1$).
The theory of $p$-adic local Langlands shows that
$\pi(V)$ 
contains as a subrepresentation
$\pi_{\sm}(V)\otimes_L (\Sym^{k-2} L^2)^{*},$
where $\pi_{\sm}(V)$ is the smooth representation of $\GL_2(\mathbb Q_p)$
over $L$ attached to the Weil--Deligne representation underlying the
potentially semi-stable Dieudonn\'e module of~$V$ (see \cite[Th\'eor\`eme 0.20, 0.21]{Colm2}).

Let $v_{\hw}$ denote a highest weight vector of
$(\Sym^{k-2} L^2)^{*}$ (well-defined up to scaling).  Note that 
$v_{\hw}$ is not only fixed by $N$ (by definition) but is actually
fixed by $P$.  Thus we have an inclusion
$$\pi_{\sm}(V) \iso
\pi_{\sm}(V)\otimes_L v_{\hw} \hookrightarrow \pi(V)_{P-\sm}.$$
In particular, we may restrict the map $\imath^-_0$ of~(\ref{eqn:third})
above to a map
$$\pi_{\sm}(V) \rightarrow \Linf \mathbf e.$$
A consideration of the formulas~(\ref{eqn:fourth}) and~(\ref{eqn:fifth}) above then
shows that this map is precisely of the form
$$v \mapsto \ell(v) \mathbb e,$$
where $\ell$ is the (suitably scaled) Kirillov functional of $\pi_{\sm}(V)$.
In summary, we have the following result, due to Colmez.

\begin{theorem}[Colmez]
\label{thm:Kirillov}
The Kirillov model of $\pi_{\sm}(V)$ (suitably scaled) satisfies the following:\ if  $v_{\sm} \in \pi_{\sm}(V)$ corresponds to the function $\phi$ in 
the Kirillov model, then $\imath^-_0(v_{\sm}\otimes v_{\hw}) = \phi(1) \mathbf e.$
\end{theorem}

\begin{proof}
See \cite[Proposition VI.5.6]{Colm2}.
\end{proof}

Note that since $V$ is assumed to be de Rham with
Hodge--Tate weights $0$ and $1-k$,
the $\Gamma$-fixed points of $\widetilde\bD_{\Sen}(V)$
are naturally identified
with $\mathbf D_{\dR}(V(1))/ \mathbf D^+_{\dR}(V(1)),$
and so our choice of $\mathbf e$ is tantamount to a choice of basis of the
latter quotient.

In light of the preceding theorem, we regard the map
$\imath^-$ as providing generalized Kirillov models 
for the $P$-representations
$\widetilde\bD^+\bigl(V(1)\bigr)[\bigl(1/\varphi^n(T)\bigr)_{n \geq 0}]/
\widetilde\bD^+\bigl(V(1)\bigr)$
and 
$\widetilde\bD_{\rig}^+\bigl(V(1)\bigr)[1/t]/
\widetilde\bD_{\rig}^+\bigl(V(1)\bigr).$
More precisely, we regard $\imath^-_i$ as evaluation at $p^{-i}$, 
and regard elements of these representations as functions on the
set $\{p^{-i}\}_{i \in \mathbb Z}$, taking values in 
$
\widetilde\bD_{\dif}\bigl(V(1)\bigr)/
\widetilde\bD_{\dif}^+\bigl(V(1)\bigr).
$

\subsection{The locally algebraic case}

\label{subsec:localg}

In the previous subsection, we embedded $\pi_{\sm}(V)$ into $\pi(V)_{P-\sm}$ by utilizing the highest weight vector in $\Sym^{k-2}(L^2)^*$.  However,  using only the highest weight vector loses information and we want to be able to work with all of $\Sym^{k-2}(L^2)^*$.  This consideration was exactly the reason we introduced Kirillov models of locally algebraic representations in section \ref{sec:localg}. 
We now explain how to see this locally algebraic Kirillov model in terms
of the morphism~$\imath^-$, {\it i.e.}\ in terms of Colmez's generalized Kirillov model.

To this end, following Colmez \cite[VI.5.2]{Colm2}, we first define 
$$
\displaystyle \pi(V)_{U-\alg} := \bigcup_{n \geq 0} {\frac{1}{(\varphi^n(T))^{k-1}}{\widetilde{\bD}}^+(V(1))} /{\widetilde{\bD}^+(V(1))}.
$$
The map $\imath^-_0$ (of section \ref{sec:sowhat}) then induces a map
$$
\imath^-_0 : \displaystyle \pi(V)_{U-\alg} \lra t^{1-k} \widetilde\bD^+_{\dif}(V(1)) / \widetilde\bD^+_{\dif}(V(1)).
$$
For $z \in \pi(U)_{U-\alg}$, we define a function
\begin{align*}
\widetilde{\mathcal K}_z : \Qp^\times &\lra t^{1-k} \widetilde\bD^+_{\dif}(V(1)) / \widetilde\bD^+_{\dif}(V(1)) \\
x &\mapsto \imath^-_0\left( \begin{pmatrix} x & 0 \\ 0 & 1 \end{pmatrix} z\right).
\end{align*}

In \cite[Lemme VI.5.4]{Colm2}, it is verified that 
$\sigma_a(\widetilde{\mathcal K}_z(x)) = \widetilde{\mathcal K}_z(ax)$ for $a \in \Zp^\times$ and thus $\widetilde{\mathcal K}_z$ is uniquely determined by its values on powers of $p$. 
That is, $\widetilde{\mathcal K}$ is simply (a concrete reinterpretation  of)
the restriction of $\imath^-$ to~$\pi(V)_{U-\alg}$.

We now state an analogue of Theorem~\ref{thm:Kirillov},
which will describe the precise
relationship of this generalized Kirillov model to the Kirillov model of a locally algebraic representation of Section \ref{sec:localg}.
To this end,
recall that for each $$z \in \pi_{\sm}(V)\otimes_L (\Sym^{k-2} L^2)^{*} \cong \pi_{\sm}(V)\otimes_L (\Sym^{k-2} L^2) \otimes {\det}^{2-k},$$ we have an associated locally polynomial map ${\mathcal K}_z : \Qp^\times \to t^{2-k} L_\infty[t] / tL_\infty[t]$.
Further, we have previously fixed a basis ${\bf e}$ of $\mathbf D_{\dR}(V(1))/ \mathbf D^+_{\dR}(V(1))$ (see section \ref{sec:normal}).  Since the Hodge-Tate weights of $V(1)$ are $1$ and $2-k$, we have that ${\bf e} \in t^{-1} \widetilde\bD^+_{\dif}(V(1))$.  Thus 
\begin{align*}
t^{2-k} L_\infty[t] / tL_\infty[t] &\to {\bf \wt{D}}_{\dif}^+(V(1))[1/t]/{\bf \wt{D}}_{\dif}^+(V(1)) \\
f(t) &\mapsto f(t) {\bf e}
\end{align*}
is a well-defined $\Gamma$-equivariant map.

Write ${\mathcal K}_z {\bf e}$ to be the map with values in ${\bf \wt{D}}_{\dif}^+(V(1))[1/t]/{\bf \wt{D}}_{\dif}^+(V(1))$ defined by sending $x$ to ${\mathcal K}_z(x) {\bf e}$. Thus ${\mathcal K}_z {\bf e}$ and $\widetilde{{\mathcal K}}_z$ both take values in the same space.  The following theorem of Colmez compares these two Kirillov models.

\begin{thm}[Colmez]
\label{thm:Kirillov_localg}
As functions of $z$ on 
$$
\pi_{\sm}(V)\otimes_L (\Sym^{k-2} L^2) \otimes {\det}^{2-k} \subseteq \pi(V)_{U-\alg},
$$
${\mathcal K}_z {\bf e}$ and $\widetilde{\mathcal K}_z$ agree up to multiplication by a scalar.
\end{thm}

\begin{proof}
This theorem follows from \cite[Prop.\ VI.5.6(iii)]{Colm2}.
\end{proof}

Note that by suitably scaling the Kirillov model of $\pi_{\sm}(V)$ we can (and will) force ${\mathcal K}_z {\bf e}$ and $\widetilde{\mathcal K}_z$ to exactly agree.

We now return to the problem of inverting $1-\left(\begin{smallmatrix} p & 0 \\ 0 & 1 \end{smallmatrix}\right)$ in the setting of generalized Kirillov models.  Recall that in Proposition~{\ref{prop:Mellinpadic_localg}, we solved the equation 
$\left(1-\left(\begin{smallmatrix} p & 0 \\ 0 & 1 \end{smallmatrix}\right)\right)\xi = {\mathcal K}_z$.  However, the solution $\xi$ was not literally in the Kirillov model of $\pi$, but instead $\xi$ restricted to any compact neighborhood was in the Kirillov model of $\pi$.  The following proposition realizes $\xi {\bf e}$ in a generalized Kirillov model.

\begin{prop}
\label{prop:Mellin two}
Suppose that the local Euler factor of $\pi_{\sm}(V)$ does not 
have a pole at $s = 1+j$.  For $z = v \otimes e_1^{k-2-j} e_2^j \in \pi \otimes \Sym^{k-2}(L^2) \otimes {\det}^{2-k}$,
let 
$$
\xi \in \LP^{[2-k,0]}(\Qp^\times,t^{2-k}L_\infty[t] / t L_\infty[t])^\Gamma 
$$
be the function of Proposition~{\em \ref{prop:Mellinpadic_localg}}
such that $\left(1-\left(\begin{smallmatrix} p & 0 \\ 0 & 1 \end{smallmatrix}\right)\right)\xi = {\mathcal K}_z$.
Then there exists an element $w \in \widetilde\bD_{\rig}^+\bigl(V(1)\bigr)[1/t]/
\widetilde\bD_{\rig}^+\bigl(V(1)\bigr)$ such that
$$
\imath_i^-(w) = \xi(p^{-i}) {\bf e}
$$
in ${\bf \wt{D}}_{\dif}^+(V(1))[1/t]/{\bf \wt{D}}_{\dif}^+(V(1))$ for all $i$.
\end{prop}
\begin{proof}
From the proof of Proposition \ref{prop:Mellinpadic_localg}, we can write the function $\xi$ as $\xi_1+\xi_2$ where $\xi_1$ lies in the Kirillov model, and $\xi_2$ satisfies:\ $\xi_2(p^i) = \begin{cases} 0 & \text{if~} i  \geq 0,\\  c &\text{if~} i < 0,\end{cases}$  for $c \in t^{2-k}L_\infty[t] / t L_\infty[t]$.  In light of Theorem \ref{thm:Kirillov_localg}, we thus need to show that there is some element $w$ in 
$\widetilde\bD_{\rig}^+\bigl(V(1)\bigr)[1/t]/
\widetilde\bD_{\rig}^+\bigl(V(1)\bigr)$ such that $\imath_i^-(w) = \xi_2(p^{-i}) {\bf e}$ for $i \in \Z$.  Lastly, the proof of \cite[Lemme VI.4.11]{Colm2} describes how such an element $w$ can be constructed.
\end{proof}

\subsection{Construction of local cohomology classes}
\label{sec:cn}

We now build the local cohomology classes $c_{n,j} \in H^1(\Knp,V(1+j))$ which will play a key role in the global arguments in the second half of the paper.

For any $n\geq 0,$ we define 
$$
d_{n,j} \in \pi_{\sm}(V) \otimes \Sym^{k-2}(L^2) \otimes {\det}^{2-k} \subseteq 
\pi(V) = \wt\bD(V(1))/\wt\bD^+(V(1))
$$
by
$$
d_{n,j} := 
 \left( \begin{smallmatrix} 1 & 1 \\ 0 & 1 \end{smallmatrix} \right) \left( \begin{smallmatrix} p^n & 0 \\
0 & 1 \end{smallmatrix} \right)  v_{\new}\otimes e_1^{k-2-j} e_2^j.
$$
Recall that $e_1$ and $e_2$ are our basis of $L^2$ and the induced action on $\Sym^{k-2}(L^2)$ is described in Section~\ref{sec:localg}.

\begin{lemma}
\label{lemma:dm}
We have
\begin{enumerate}
\item $d_{n,j} \in \displaystyle \frac{1}{\varphi^n(T)^{j+1}} \wt\bD^+(V(1))/\wt\bD^+(V(1))$;
\item for $a \in \Gamma_n = 1 + p^n \Zp$, we have $  \left( \begin{smallmatrix} a & 0 \\ 0 & 1 \end{smallmatrix} \right)  d_{n,j} = a^{-j} \cdot d_{n,j}$.
\end{enumerate}
\end{lemma}

\begin{proof}
First note the simple matrix equation:
$$
\left( \begin{smallmatrix} 1 & p^{n} \\ 0 & 1 \end{smallmatrix} \right)  \left( \begin{smallmatrix} 1 & 1 \\ 0 & 1 \end{smallmatrix} \right) \left( \begin{smallmatrix} p^n & 0 \\ 0 & 1 \end{smallmatrix} \right)
=
\left( \begin{smallmatrix} 1 & 1 \\ 0 & 1 \end{smallmatrix} \right)  \left( \begin{smallmatrix} p^n & 0 \\ 0 & 1 \end{smallmatrix} \right) \left( \begin{smallmatrix} 1 & 1 \\ 0 & 1 \end{smallmatrix} \right).
$$
Thus,
\begin{align*}
 \left( \begin{smallmatrix} 1 & p^{n} \\ 0 & 1 \end{smallmatrix} \right)  d_{n,j}
&= \left( \begin{smallmatrix} 1 & 1 \\ 0 & 1 \end{smallmatrix} \right)  \left( \begin{smallmatrix} p^n & 0 \\ 0 & 1 \end{smallmatrix} \right)  v_{\new} \otimes e_1^{k-2-j} (p^n e_1 + e_2)^j
\end{align*}
as $\left( \begin{smallmatrix} 1 & 1 \\ 0 & 1 \end{smallmatrix} \right)$ fixes $v_\new$.
Computing further, we see
\begin{align*}
\varphi^n(T)(d_{n,j}) 
&= \left( \begin{smallmatrix} 1 & p^{n} \\ 0 & 1 \end{smallmatrix} \right)d_{n,j} -d_{n,j}\\
&= \left( \begin{smallmatrix} 1 & 1 \\ 0 & 1 \end{smallmatrix} \right)  \left( \begin{smallmatrix} p^n & 0 \\ 0 & 1 \end{smallmatrix} \right)  v_{\new} \otimes (e_1^{k-2-j} (p^n e_1 + e_2)^j -  e_1^{k-2-j} e_2^j);
\end{align*}
note that $e_1^{k-2-j} (p^n e_1 + e_2)^j -  e_1^{k-2-j} e_2^j$ is a homogenous polynomial in $e_1$ and $e_2$ where $e_2^{j-1}$ is the highest power of $e_2$ which occurs.
The same computation as above then shows
$$
(\varphi^n(T))^2(d_{n,j}) 
= \left( \begin{smallmatrix} 1 & 1 \\ 0 & 1 \end{smallmatrix} \right)  \left( \begin{smallmatrix} p^n & 0 \\ 0 & 1 \end{smallmatrix} \right)  v_{\new} \otimes g(e_1,e_2) 
$$
where $g(e_1,e_2)$ is again a homogenous polynomial in $e_1$ and $e_2$, but all powers of $e_2$ present are less than or equal to $j-2$.  Iterating this argument yields that $(\varphi^n(T))^{j+1}(d_{n,j}) = 0$, proving the first part of the proposition.

For the second part, for $a = 1+p^n x$ with $x \in \Zp$, we have
\begin{align*}
\left( \begin{smallmatrix} a & 0 \\ 0 & 1  \end{smallmatrix} \right) d_{n,j} 
&=
\left( \begin{smallmatrix} a & 0 \\ 0 & 1  \end{smallmatrix} \right) 
\left( \left( \begin{smallmatrix} 1 & 1 \\ 0 & 1  \end{smallmatrix} \right) 
 \left( \begin{smallmatrix} p^n & 0 \\ 0 & 1  \end{smallmatrix} \right) 
 v_{\new}\otimes e_1^{k-2-j} e_2^j \right)\\
&=
a^{-j} \left( \begin{smallmatrix} p^n a & a \\ 0 & 1  \end{smallmatrix} \right) 
 v_{\new}\otimes e_1^{k-2-j} e_2^j \\
&=
a^{-j} \left( \begin{smallmatrix} p^n & 1 \\ 0 & 1  \end{smallmatrix} \right) 
 \left( \begin{smallmatrix} 1 & x \\ 0 & 1  \end{smallmatrix} \right) 
 \left( \begin{smallmatrix} a & 0 \\ 0 & 1  \end{smallmatrix} \right) 
 v_{\new}\otimes e_1^{k-2-j} e_2^j \\
&=
a^{-j} \left( \begin{smallmatrix} p^n & 1 \\ 0 & 1  \end{smallmatrix} \right) 
 v_{\new}\otimes e_1^{k-2-j} e_2^j
\end{align*}
proving the claim.
We note that the factor of $a^{-j}$ arises by combining the action of $\left( \begin{smallmatrix} a & 0 \\ 0 & 1  \end{smallmatrix} \right)$ on $e_1^{k-2-j}e_2^j$ and the twist by ${\det}^{2-k}$.
\end{proof}

We now aim to build local cohomology classes from the elements $d_{n,j} \in \pi(V)$.
To this end, note that as vector spaces we have $\pi(V) = \pi(V(j))$ though the underlying $\Gamma$-action differs.  For $z \in \pi(V)$, we write $z^{(j)}$ for the element $z$, but viewed in $\pi(V(j))$.  Let $z' \in {\bf D}(V^{*}(-j))^{\psi = 1}$ which we identify with $H^1_{\Iw}(V^{*}(-j))$ and write $z'_n$ for the image of $z'$ to level $n$ in $H^1(\Knp,V^{*}(-j))$ where $\Knp := \Qp(\mu_{p^n})$.
Lastly, recall the dual exponential map
$$
\exp^* : H^1\bigl(\Knp,V^{*}(-j)\bigr) \to \bD^+_{\dR}\bigl(V^{*}(-j)\bigr) \otimes \Knp.
$$
We then have the following twisted version of Theorem~\ref{thm:explicit}.

\begin{cor}
\label{cor:specialized explicit}
Let $z \in \widetilde\bD_{\rig}^+\bigl(V(1)\bigr)[1/t]$
and let $z' \in \mathbf D(V^{*}(-j))^{\psi = 1}.$  Suppose that 
\begin{enumerate}
\item 
$(1-\varphi)z$ lies in $\widetilde\bD^+\bigl(V(1)\bigr)[1/\varphi^r(T)]$ for some $r \geq 0$,
\item 
\label{item:2}
for some $n~\geq~0$, we have $\gamma z = \gamma^{-j} \cdot z$ for all $\gamma \in 1+p^n\Zp = \Gamma_n \subseteq \Gamma$.
\end{enumerate}
Then
$$
\{(1-\varphi)z^{(j)}, z'\} = p^n \langle \imath^-_r(z^{(j)}), \sigma_{-1} \exp^*(z'_n)\rangle_{\dif}.
$$
\end{cor}

\begin{proof}
By assumption (\ref{item:2}), $z^{(j)}$ is fixed by $\Gamma_n$ and Theorem \ref{thm:explicit} implies
$$
\{(1-\varphi)z^{(j)},  z'\} = \langle \imath^-_r(z^{(j)}), \sigma_{-1} \cdot  \imath_m(z') \rangle_{\dif}
$$
for $m$ large enough.  This corollary then follows immediately from Lemma \ref{lemma:inexp}.
\end{proof}

As we intend to be pairing with $d_{n,j}$, to apply the above corollary, we need to show that $d_{n,j}$ is in the image of $1 - \varphi$.  To this end, suppose that the local Euler factor of $\pi_{\sm}(V)$ has no pole
at $s = 1+j$.  Thus, by Propositions~\ref{prop:Mellinpadic_localg}, we can solve the equation 
$\left(1-\left(\begin{smallmatrix} p & 0 \\ 0 & 1 \end{smallmatrix}\right)\right)\xi_{n,j} = {\mathcal K}_{d_{n,j}}$
in the Kirillov model of $\pi_{\sm}(V) \otimes \Sym^{k-2}(L^2) \otimes {\det}^{2-k}$.  
By Proposition~\ref{prop:Mellin two}, we can then realize $\xi_{n,j}$ in a generalized Kirillov model for $\pi(V)$; that is, we can  find $w_{n,j} \in \widetilde\bD_{\rig}^+(V(1))[1/t] $ such that $\imath^-_i(w_{n,j}) = \xi_{n,j}(p^{-i}) {\bf e}$ for all $i$.  
Since $\imath^-$ is injective (Proposition \ref{prop:support}), we deduce that 
$(1-\varphi)w_{n,j} \equiv d_{n,j} \bmod \widetilde\bD_{\rig}^+\bigl(V(1)\bigr).$

\begin{remark}
One can explicitly write down the classes $d_{n,j}$ and $w_{n,j}$ in the (locally algebraic) Kirillov model of   $\pi_{\sm}(V) \otimes \Sym^{k-2}(L^2) \otimes {\det}^{2-k}$.  For instance, when $\pi_{\sm}(V)$ is supercuspidal, then $v_{\new}$ corresponds to ${\bf 1}_{\Zp^\times}(x)$ in the (smooth) Kirillov model of $\pi_{\sm}(V)$ (when appropriately normalized) and $\left( \begin{smallmatrix} 1 & 1 \\ 0 & 1 \end{smallmatrix} \right)\left( \begin{smallmatrix} p^n & 0 \\ 0 & 1 \end{smallmatrix} \right) v_{\new}$ corresponds to $\ve(x) \cdot {\bf 1}_{\Zp^\times}(p^n x)$. Thus, by (\ref{eqn:localg}),
$$
{\mathcal K}_{d_{n,j}}(x) = j! \cdot (tx)^{-j} \cdot \ve(x) \cdot {\bf 1}_{\Zp^\times}(p^nx);
$$
that is, 
$$
{\mathcal K}_{d_{n,j}}(p^r) = \begin{cases} j! \cdot p^{nj} \cdot \ve(p^{-n}) \cdot t^{-j} & r=-n, \\ 0 & r \neq -n.\end{cases}
$$
Further, tracing through the proof of Proposition \ref{prop:Mellinpadic_localg}, we see that 
$$
{\mathcal K}_{w_{n,j}} = \sum_{n \geq 0} \left( \begin{smallmatrix} p^n & 0 \\ 0 & 1 \end{smallmatrix} \right) {\mathcal K}_{d_{n,j}};
$$
that is,
$$
{\mathcal K}_{w_{n,j}}(p^r) = \begin{cases} 0 & r>-n,\\ j! \cdot p^{nj} \cdot \ve(p^{-n}) \cdot t^{-j} & r \leq -n.\end{cases}
$$
\end{remark}

Recall that when we view the element $d_{n,j}$ in $\pi(V(j))$ rather than $\pi(V)$, we write it as $d_{n,j}^{(j)}$. To ease notation a little, let's simply write $d_n^{(j)}$ for $d_{n,j}^{(j)}$ and likewise for $w_n^{(j)}$.
Then given any element $z' \in \mathbf D(V^{*}(-j))^{\psi = 1}$,
Lemma \ref{lemma:dm} and Corollary~\ref{cor:specialized explicit} 
show that
\begin{equation}
\label{eqn:pairing}
\{d_n^{(j)}, z'\} = \{(1-\varphi)w_n^{(j)}, z'\} = p^n \left\langle \imath^-_n(w_n^{(j)}),\sigma_{-1} \exp^*z'_n \right\rangle_{\dif}.
\end{equation}

Note that pairing with $d_n^{(j)}$ gives a functional on $\mathbf D(V^{*}(-j))^{\psi = 1} \cong H^1_{\Iw}(V^{*}(-j))$. By (\ref{eqn:pairing}), we see that this pairing only depends on the projection of $z' \in H^1_{\Iw}(V^{*}(-j))$ to $H^1(\Knp,V^{*}(-j))$.  Further, since $H^1_{\Iw}(V^{*}(-j)) \to H^1(\Knp,V^{*}(-j))$ is surjective (see \cite[1.6]{PR99}), this pairing factors through $H^1(\Knp,V^{*}(-j))$.
 Thus, by Tate local duality, there is a unique $c_{n,j} \in H^1(\Knp,V(1+j))$ such that 
$$
\langle c_{n,j},   z'_n \rangle_n = p^n \left\langle \imath^-_n(w_{n}^{(j)}),\sigma_{-1} \exp^*z'_n \right\rangle_{\dif}
$$
where $\langle \cdot, \cdot\rangle_n$ is the perfect pairing on Galois cohomology induced by Tate local duality.
Moreover, 
(\ref{eqn:pairing}) implies that pairing with $c_{n,j}$ kills the kernel of $\exp^*$ which is $H^1_g(\Knp,V^{*}(-j))$, and thus $c_{n,j} \in H^1_e(\Knp,V(1+j))$.  Since we are assuming (\ref{nopole}), \cite[Theorem 4.1(ii)]{BK} implies that $H^1_e(\Knp,V(1+j)) = H^1_f(\Knp,V(1+j))$. We summarize these observations in the following proposition.

\begin{prop}
\label{prop:cn_summary}
Suppose that the local Euler factor of $\pi_{\sm}(V)$ does not 
have a pole at $s = 1+j$.  Then there exist classes $c_{n,j} \in H^1_f(\Knp,V(1+j))$ such that for $z' \in \mathbf D(V^{*}(-j))^{\psi = 1} \cong H^1_{\Iw}(V^{*}(-j))$, we have
$$
\langle c_{n,j},   z'_n \rangle_n =p^n \left\langle \imath^-_n(w_{n}^{(j)}),\sigma_{-1} \exp^*z'_n \right\rangle_{\dif}
$$
where $(1-\varphi) w_n^{(j)} = d_n^{(j)}$ in $\pi(V)$.
\end{prop}

The local classes $c_{n,j}$ will play a key role in the global computations in the second half of the paper.

\subsection{Twisted pairing sums}

The formula of the following proposition will be used in the global half of the paper to relate our algebraic $\theta$-elements to the Mazur-Tate elements.  

In what follows, recall from section \ref{sec:pair-hodge} that $\langle \cdot, \cdot \rangle'_{\dR,n}$ denotes the pairing
$$
\langle \cdot, \cdot \rangle'_{\dR,n} : (\bD_{\dR}(V(1+j)) \otimes \Knp) \times (\bD_{\dR}(V^{*}(-j)) \otimes \Knp) \to L \otimes \Knp
$$
and $\langle \cdot, \cdot \rangle_{\dR,n} = \Tr^{L_n}_L \langle \cdot, \cdot \rangle'_{\dR,n}$.

\begin{prop}
\label{prop:recip}
Suppose that the local Euler factor of $\pi_{\sm}(V)$ does not 
have a pole at $s = 1 + j$.
For $z'_n \in H^1(\Knp,V^{*}(-j))$ and $\chi$ a Dirichlet character of conductor $p^n$, we have
\begin{multline*}
\sum_{a \in (\mathbb Z/p^n\mathbb Z)^{\times}}
\chi^{-1}(a) \left\langle  c_{n,j}^{\sigma_a},z'_n\right\rangle_n 
= \\
\begin{cases}
\displaystyle j! \cdot p^{nj} \cdot \tau(\chi^{-1})
\left\langle t^{-j} \mathbf e, 
\displaystyle \sum_{a \in (\mathbb Z/p^n \mathbb Z)^{\times}}
\chi(-a)  \exp^* (z_n')^{\sigma_a}\right\rangle'_{\dR,n} & n \geq 1, \\
&\\
j! \cdot \widetilde{Z}(x^{-j} \phi_{v_{\new}},1) \left\langle t^{-j} \mathbf e, \exp^* z_0'\right\rangle_{\dR,0} & n=0.
\end{cases}
\end{multline*}
\end{prop}

\begin{proof}
Using Proposition \ref{prop:cn_summary}, we compute for $n\geq 1$
\begin{align*}
\langle \sigma _a c_{n,j},   z'_n \rangle_n 
&= p^n \left\langle \imath^-_n( \left(\begin{smallmatrix} a & 0 \\ 0 & 1 \end{smallmatrix}\right) w_{n}^{(j)}),\sigma_{-1}  \exp^*z'_n \right\rangle_{\dif}\\
&=
\left\langle 
\left(\begin{smallmatrix} a & 0 \\ 0 & 1 \end{smallmatrix}\right) \imath^-_n(  w_{n}^{(j)}),
\sigma_{-1}  \exp^*z'_n
\right\rangle_{\dR,n}\\
&=
\Tr^{L_n}_L 
\left\langle 
\left(\begin{smallmatrix} a & 0 \\ 0 & 1 \end{smallmatrix}\right) \imath^-_n(  w_{n}^{(j)}),
\sigma_{-1}  \exp^*z'_n
\right\rangle'_{\dR,n}\\
&=
\sum_b
\sigma_b \left\langle 
\left(\begin{smallmatrix} a & 0 \\ 0 & 1 \end{smallmatrix}\right) \imath^-_n(  w_{n}^{(j)}),
\sigma_{-1}  \exp^*z'_n
\right\rangle'_{\dR,n}\\
&=
\sum_b
\left\langle 
\left(\begin{smallmatrix} ab & 0 \\ 0 & 1 \end{smallmatrix}\right) \imath^-_n(  w_{n}^{(j)}),
\sigma_{-b} \exp^*z'_n
\right\rangle'_{\dR,n}
\end{align*}
where the sums above and below are over any systems of representatives for $(\Z/p^n\Z)^\times$.

Thus
\begin{align*}
\sum_{a}
\chi^{-1}(a) &\left\langle  {\sigma_a} c_{n,j},z'_n\right\rangle_n \\
&=
\sum_{a}
\chi^{-1}(a)
\sum_b
\left\langle 
\left(\begin{smallmatrix} ab & 0 \\ 0 & 1 \end{smallmatrix}\right) \imath^-_n(  w_{n}^{(j)}),
\sigma_{-b} \exp^*z'_n
\right\rangle'_{\dR,n}\\
&=
\sum_{b}
\left\langle 
\sum_a \chi^{-1}(a)
\left(\begin{smallmatrix} ab & 0 \\ 0 & 1 \end{smallmatrix}\right) \imath^-_n(  w_{n}^{(j)}),
\sigma_{-b} \exp^*z'_n
\right\rangle'_{\dR,n}\\
&=
\sum_{b}
\left\langle 
\sum_a \chi^{-1}(a)
\left(\begin{smallmatrix} a & 0 \\ 0 & 1 \end{smallmatrix}\right) \imath^-_n(  w_{n}^{(j)}),
\chi(b) \sigma_{-b} \exp^*z'_n
\right\rangle'_{\dR,n}\\
&=
\left\langle 
\sum_a \chi^{-1}(a)
\left(\begin{smallmatrix} a & 0 \\ 0 & 1 \end{smallmatrix}\right) \imath^-_n(  w_{n}^{(j)}),
\sum_b \chi(b)\sigma_{-b} \exp^*z'_n
\right\rangle'_{\dR,n}.
\end{align*}
Focusing on the first term in the above pairing and applying Propositions~\ref{prop:Mellinpadic_localg} and~\ref{prop:Mellin two}, we have
\begin{align*}
\sum_a \chi^{-1}(a)
\left(\begin{smallmatrix} a & 0 \\ 0 & 1 \end{smallmatrix}\right) \imath^-_n(  w_{n}^{(j)})
&=
\sum_a \chi^{-1}(a)
\left(\begin{smallmatrix} a & 0 \\ 0 & 1 \end{smallmatrix}\right) \imath^-_n(  w_{n,j}) a^j\\
&=
\sum_a \chi^{-1}(a)
\left(\begin{smallmatrix} a & 0 \\ 0 & 1 \end{smallmatrix}\right) \widetilde{Z}\left({\mathcal K}_{d_{n,j}} {\bf e},1\right) a^j
=
\widetilde{Z}\left({\mathcal K}_\alpha {\bf e},1\right)
\end{align*}
where
\begin{align*}
\alpha &= \sum_a \chi^{-1}(a) a^j \left(\begin{smallmatrix} a & 0 \\ 0 & 1 \end{smallmatrix}\right) d_{n,j} \\
&= \sum_a \chi^{-1}(a) a^j \left(\begin{smallmatrix} a & 0 \\ 0 & 1 \end{smallmatrix}\right)
\left(\left(\begin{smallmatrix} 1 & 1 \\ 0 & 1 \end{smallmatrix}\right)\left(\begin{smallmatrix} p^n & 0 \\ 0 & 1 \end{smallmatrix}\right) v_{\new} \otimes e_1^{k-2-j} e_2^j  
\right)\\
&= \sum_a \chi^{-1}(a) \left(\begin{smallmatrix} p^na & a \\ 0 & 1 \end{smallmatrix}\right) v_{\new} \otimes e_1^{k-2-j} e_2^j  \\
&= \sum_a \chi^{-1}(a) \left(\begin{smallmatrix} p^n & 0 \\ 0 & 1 \end{smallmatrix}\right)\left(\begin{smallmatrix} 1 & a/p^n \\ 0 & 1 \end{smallmatrix}\right) 
\left(\begin{smallmatrix} a & 0 \\ 0 & 1 \end{smallmatrix}\right)v_{\new} \otimes e_1^{k-2-j} e_2^j  \\
&= \left(\begin{smallmatrix} p^n & 0 \\ 0 & 1 \end{smallmatrix}\right) \left(\sum_a \chi^{-1}(a) \left(\begin{smallmatrix} 1 & a/p^n \\ 0 & 1 \end{smallmatrix}\right) 
v_{\new}\right) \otimes e_1^{k-2-j} e_2^j.
\end{align*}
By the local Birch lemma (section \ref{sec:localbirch}), we have $\sum_a \chi^{-1}(a) \left(\begin{smallmatrix} 1 & a/p^n \\ 0 & 1 \end{smallmatrix}\right) v_{\new}$ corresponds simply to $\tau(\chi^{-1}) {\bf 1}_{\Zp^\times}(x)$ in the smooth Kirillov model attached to $\pi_{\sm}(V)$.
Thus, by (\ref{eqn:localg}), we have
\begin{align*}
{\mathcal K}_\alpha(x) = 
j! \cdot (tx)^{-j} \cdot \tau(\chi^{-1}) \cdot {\bf 1}_{\Zp^\times}(p^nx)
\end{align*}
and
$$
\widetilde{Z}({\mathcal K}_\alpha {\bf e},1) = j! \cdot p^{nj} \cdot \tau(\chi^{-1}) \cdot t^{-j}{\bf e}.
$$
Putting it all together gives
$$
\sum_{a}
\chi^{-1}(a) \left\langle  {\sigma_a} c_{n,j},z'_n\right\rangle_n 
=
j! \cdot p^{nj} \cdot \tau(\chi^{-1})
 \left\langle 
t^{-j} {\bf e},
\sum_b \chi(-b)\sigma_{b} \exp^*z'_n
\right\rangle'_{\dR,n}
$$
as desired.

For $n=0$, we have
\begin{align*}
\left\langle c_{0,j}, z_0' \right \rangle_0 = \left\langle \imath^-_0(w_0^{(j)}),\exp^* z'_0\right\rangle_{\dif}
\end{align*}
and
\begin{align*}
\imath^-_0(w_0^{(j)}) &= \imath^-_0(w_{0,j}) = \widetilde{Z}\left(\mathcal K_{d_{0,j}},1\right) {\bf e}
= \widetilde{Z}\left(j! (tx)^{-j} \phi_{v_{\new}},1\right){\bf e} = 
j! \widetilde{Z}\left( x^{-j} \phi_{v_{\new}},1\right) t^{-j} {\bf e}.
\end{align*}
Thus 
\begin{align*}
\left\langle c_{0,j}, z_0' \right \rangle_0 = j! \cdot \widetilde{Z}\left( x^{-j} \phi_{v_{\new}},1\right) \cdot \left\langle   t^{-j} {\bf e},\exp^* z'_0\right\rangle_{\dif}
\end{align*}
as desired.
\end{proof}

\subsection{Three-term relation}
\label{sec:three-term}

In this subsection, we describe some basic properties of the local cohomology classes $\{c_{n,j}\}$ including their key three-term relation.

We continue to assume that $V$ is de Rham with Hodge–Tate weights equal to $0$ and $1 - k$, for some $k \geq 2$, and let $\pi_{\sm}(V)$ be the associated smooth representation of $\GL_2(\Qp)$ with central character $\chi_V$.  If $T$ is the Hecke operator associated to the matrix $\psmallmat{p}{0}{0}{1}$, write $T v_{\new} = a v_{\new}$ with $a \in \O_L$.  Also, define $\delta = \chi_V(p)$ if $\pi_{\sm}(V)$ is unramified ({\it i.e.}\ if $V$ is crystalline)
and $0$ otherwise.  We then have the following lemma.

\begin{lemma}
\label{lemma:newvector}
We have
$$
\sum_{i=0}^{p-1} \begin{pmatrix} p & i \\ 0 & 1\end{pmatrix} v_{\new}  + \delta \begin{pmatrix} p^{-1} & 0 \\ 0 & 1 \end{pmatrix} v_{\new} = a \cdot v_{\new}.
$$
\end{lemma}

\begin{proof}
We have
$$
T = 
\begin{cases}
\displaystyle \sum_{i=0}^{p-1} \left( \begin{smallmatrix} p & i \\ 0 & 1\end{smallmatrix} \right) + \left( \begin{smallmatrix} 1 & 0 \\ 0 & p\end{smallmatrix} \right) & \text{if~}\pi_{\sm}(V) \text{~is~unramified,}\\
~\\
\displaystyle  \sum_{i=0}^{p-1} \left( \begin{smallmatrix} p & i \\ 0 & 1\end{smallmatrix} \right)   &\text{otherwise,}
\end{cases}
$$
which implies the lemma.
\end{proof}

For $n \geq m$, let $\cores^n_m : H^1(\Knp,V(1+j)) \to H^1(\Kmp,V(1+j))$ and $\res^n_m : H^1(\Kmp,V(1+j)) \to H^1(\Knp,V(1+j))$ denote the natural corestriction and restriction maps.  We then have the following three-term relation for the local cohomology classes $\{c_{n,j}\}$.

\begin{prop}
\label{prop:three_term}
For $n \geq 1$, we have
$$
\cores^{n+1}_n(c_{n+1,j}) = a c_n - \delta \res^n_{n-1}(c_{n-1,j}),
$$
and
$$
\cores^{1}_0(c_{1,j}) = (a-\delta-1) c_{0,j}.
$$
\end{prop}

\begin{proof}
For $n\geq1$, we compute
\begin{align*}
\cores^{n+1}_n(d_{n+1}^{(j)}) 
&= 
\sum_{i=0}^{p-1} \psmallmat{1+ip^n}{0}{0}{1} d_{n+1}^{(j)}
=
\sum_{i=0}^{p-1} \psmallmat{1+ip^n}{0}{0}{1} \psmallmat{1}{1}{0}{1} \psmallmat{p^{n+1}}{0}{0}{1} v_{\new} \otimes e_1^{k-2-j} e_2^j \\
&=
\sum_{i=0}^{p-1} \psmallmat{p^{n+1}(1+p^ni)}{1+ip^n}{0}{1}  v_{\new} \otimes e_1^{k-2-j} e_2^j\\
&=
\sum_{i=0}^{p-1} \psmallmat{1}{1}{0}{1} \psmallmat{p^{n}}{0}{0}{1}  \psmallmat{p}{i}{0}{1} \psmallmat{1+ip^n}{0}{0}{1} v_{\new} \otimes e_1^{k-2-j} e_2^j\\
&=
\psmallmat{1}{1}{0}{1}\psmallmat{p^{n}}{0}{0}{1} \left(\sum_{i=0}^{p-1}   \psmallmat{p}{i}{0}{1}  v_{\new}\right) \otimes e_1^{k-2-j} e_2^j\\
&=
\psmallmat{1}{1}{0}{1}\psmallmat{p^{n}}{0}{0}{1} \left( a \cdot v_{\new} - \delta \psmallmat{p^{-1}}{0}{0}{1} v_{\new} \right) \otimes e_1^{k-2-j} e_2^j\\
&=
a \cdot \psmallmat{1}{1}{0}{1}\psmallmat{p^{n}}{0}{0}{1} v_{\new} \otimes e_1^{k-2-j} e_2^j 
- \delta \psmallmat{1}{1}{0}{1}\psmallmat{p^{n-1}}{0}{0}{1} v_{\new}  \otimes e_1^{k-2-j} e_2^j\\
&= a \cdot d_{n}^{(j)} - \delta \cdot d_{n-1}^{(j)}.
\end{align*}
Since $\langle c_{n,j},   z'_n \rangle_n = \{d_n^{(j)}, z'\}$, the above relations for the $d_{n}^{(j)}$ imply the corresponding relations for the $c_{n,j}$ for $n\geq 1$.

For $n=0$, we have
\begin{align*}
\cores^1_0(d_1^{(j)})
&=
\sum_{i=1}^{p-1} \psmallmat{i}{0}{0}{1} d_1^{(j)}
=
\sum_{i=1}^{p-1} \psmallmat{i}{0}{0}{1} \psmallmat{1}{1}{0}{1} \psmallmat{p}{0}{0}{1} v_{\new} \otimes e_1^{k-2-j} e_2^j \\
&=
\sum_{i=1}^{p-1} \psmallmat{p}{i}{0}{1} \psmallmat{i}{0}{0}{1}  v_{\new} \otimes e_1^{k-2-j} e_2^j \\
&=
\left( a \cdot  v_{\new} - \delta \psmallmat{p^{-1}}{0}{0}{1} v_{\new} - \psmallmat{p}{0}{0}{1}  v_{\new} \right)\otimes e_1^{k-2-j} e_2^j \\
&= \left( a - \delta \psmallmat{p^{-1}}{0}{0}{1} - \psmallmat{p}{0}{0}{1}\right) d_0^{(j)}.
\end{align*}
The relation between $c_{1,j}$ and $c_{0,j}$ then follows as above combined with the fact that $\{z,z'\} = \left\{\psmallmat{p^{\pm1}}{0}{0}{1}z,z'\right\}$ for any $z \in \pi(V)$ and $z' \in D(V^*)^{\psi=1}$.  
\end{proof}

\subsection{The crystalline case}
Suppose now that $V(1)$ is crystalline of dimension two,
and that $D_{\crys}\bigl(V(1)\bigr)$  has 
distinct Frobenius eigenvalues $\alpha/p$ and $\beta/p$ arising
from eigenvectors 
$$e_{\alpha/p}, e_{\beta/p} \in D_{\crys}\bigl(V(1)\bigr) \subseteq \widetilde\bD_{\rig}^+\bigl(V(1)\bigr)[1/t].
$$
Suppose that neither $\alpha$ nor $\beta$ equals $p$ so that the local Euler factor of $\pi_{\sm}(V)$ does not have a pole at $s=1$.  Further suppose that $V$ is irreducible, so that $e_{\alpha/p}$ and $e_{\beta/p}$ both
have non-zero image in $\mathbf D_{\dR}\bigl(V(1)\bigr)/\mathbf D_{\dR}^+
\bigl(V(1)\bigr)$, and scale them so that their images coincide.
We take this common image to be our basis $\mathbf e$ in the preceding
discussion, and set $\omega = e_{\alpha/p} - e_{\beta/p} \in \mathbf D_{\dR}^+
\bigl(V(1)\bigr)$.

In this case, we can give another (more direct) description of the classes $c_n := c_{n,0}$ as follows.  (So we are fixing $j=0$ for this subsection.). First, for $n\geq 0$, set
$$
k_n = 
\ve^{(n)} \otimes \varphi \omega + \dots + \ve^{(1)} \otimes \varphi^n\omega +
(1-\varphi)^{-1} \varphi^{n+1}\omega \in \Knp \otimes_{\Qp} D_{\crys}(V(1))
$$
where for $n=0$ this formula simply means the final term.   Note that $1-\varphi$ is invertible as we have assumed that $\varphi$ does not have 1 as an eigenvalue.

Recall the exponential map
$$
\exp: \Knp \otimes D_{\crys}(V(1)) \to H^1_f(\Knp,V(1))
$$
and consider the elements $\exp(k_n) \in H^1(\Knp,V(1))$ for $n\geq 0$.  A simple computation shows that these elements satisfy the three term relation given in Proposition \ref{prop:three_term}.  
In fact, we will show the following:

\begin{prop}
For $n\geq 0$, we have $c_n = \left(\frac{\alpha}{p}-\frac{\beta}{p}\right)^{-1} \exp(k_n)$.
\end{prop}

\begin{proof}
Recall that $c_n$ is defined by the formula
$$
\langle c_n,  z_n'\rangle_n = p^n \langle \imath^-_n(w_n), \sigma_{-1}  \exp^*z_n' \rangle_{\dif}
$$
where $w_n := w_{n,0}$.  Further, by \cite[Th\'eor\`eme VIII.2.2]{Colm2}, we have
$$
\langle \exp(k_n),  \sigma_{-1} \cdot z_n'\rangle_n = p^n \langle k_n, \sigma_{-1}  \exp^*z_n' \rangle_{\dif}.
$$
Thus it suffices to see that the images of $\left(\frac{\alpha}{p}-\frac{\beta}{p}\right)^{-1} k_n$ and $\imath^-_n(w_n)$ agree in $D_{\dR}(V(1)) / D^+_{\dR}(V(1))$.

To this end, we first explicitly describe the realization of $d_n := d_{n,0}$ and $w_n$ in (generalized) Kirillov models.  
If $\phi$ denotes the realization of $v_{\new}$ in the Kirillov model of $\pi_{\sm}(V)$, then 
$$
\phi(p^n) = \begin{cases}
\dfrac{(\alpha/p)^{n+1}
-(\beta/p)^{n+1}}{(\alpha/p) - (\beta/p)} & n \geq 0, \\
0 & \text{otherwise,}
\end{cases}
$$
as in section \ref{sec:newvec}.  Thus,
$$
\mathcal K_{d_n}(p^r){\bf e} = \begin{cases} \phi(p^{n+r}){\bf e} &  r \geq 0, \\
\ve^{(-r)} \phi(p^{n+r}){\bf e} & -n \leq r \leq 0, \\ 0 & r < -n, \end{cases}
$$
where $\ve^{(m)} = \ve(p^{-m})$ is a $p^m$-th root of unity.

Tracing through the proof of Proposition \ref{prop:Mellinpadic}, we then see that
\begin{multline*}
\xi_n(p^{-n}) = \\
\left(\frac{\alpha}{p}-\frac{\beta}{p}\right)^{-1} 
\left(
 \sum_{j=0}^{n-1} \ve^{(n-j)} 
\left( \left(\frac{\alpha}{p} \right)^{j+1} - \left(\frac{\beta}{p} \right)^{j+1} \right)
+
\frac{(\alpha/p)^{n+1}}{1-\alpha/p}-
\frac{(\beta/p)^{n+1}}{1-\beta/p}
\right) {\mathbf e}.
\end{multline*}
Since $\omega = e_{\alpha/p} - e_{\beta/p}$, we have
\begin{align*}
k_n
&= 
\sum_{j=0}^{n-1}
\ve^{(n-j)} \otimes \varphi^{j+1} \omega +
(1-\varphi)^{-1} \varphi^{n+1}\omega
\\
&= 
\sum_{j=0}^{n-1}
\ve^{(n-j)} \otimes \left(\left(\frac{\alpha}{p}\right)^{j+1} e_{\alpha/p}
-
\left(\frac{\beta}{p}\right)^{j+1} e_{\beta/p}\right)
+
\frac{(\alpha/p)^{n+1}}{1-\alpha/p} e_{\alpha/p}
- \frac{(\beta/p)^{n+1}}{1-\beta/p} e_{\beta/p}
\end{align*}
and thus
$$
k_n \equiv \left(\frac{\alpha}{p}-\frac{\beta}{p}\right) \imath^-_n(w_n) \pmod{D^+_{\dR}(V(1))}
$$
as desired.
\end{proof}

\begin{remark}
One can even go further and directly check that the classes $\exp(k_n)$ satisfy Proposition \ref{prop:recip}.  Nonetheless, we note that the explicit construction in this section for the crystalline case do not supplant the main arguments given so far in this paper.  Indeed, these explicit constructions with the exponential map do not give us any integral control of these cohomology classes.  In the next section, we describe how the $c_n$ can be normalized so that they are in $H^1(\Knp,T(1))$ for all $n\geq 0$ for $T$ some Galois-stable lattice in $V$.  We cannot directly check this property for the classes $\exp(k_n)$.
\end{remark}

\subsection{Normalizations}
\label{sec:normal}
If $T$ is a $G_{\mathbb Q_p}$-invariant $\mathcal O_L$-lattice
in the two-dimensional
irreducible continuous $G_{\mathbb Q_p}$-representation $V$ over $L$,
then we can form 
$\widetilde\bD^+\bigl(T(1)\bigr) \subset 
\widetilde\bD^+\bigl(V(1)\bigr)$
and
$\widetilde\bD\bigl(T(1)\bigr) \subset 
\widetilde\bD\bigl(V(1)\bigr)$.
Note that 
$\widetilde\bD^+\bigl(T(1)\bigr) =
\widetilde\bD\bigl(T(1)\bigr) \cap
\widetilde\bD^+\bigl(V(1)\bigr)$
(the intersection taking place in
$\widetilde\bD\bigl(V(1)\bigr)$),
so that we also obtain an embedding
$$
\pi(T) := \widetilde\bD\bigl(T(1)\bigr)/
\widetilde\bD^+\bigl(T(1)\bigr)
\hookrightarrow
\widetilde\bD\bigl(V(1)\bigr)/
\widetilde\bD^+\bigl(V(1)\bigr) = \pi(V)
,$$
which gives an integral structure on the target.

Suppose now, as before, that $V$ is de Rham, with Hodge--Tate weights
$0$ and $1 - k < 0.$  Then we have the inclusion
$\pi_{\sm}(V)\otimes_L (\Sym^{k-2} L^2)^{*} \hookrightarrow \pi(V)$,
and we may normalize the vector $v_{\new} \in \pi_{\sm}(V)$ (up to an element
of $\mathcal O_L^{\times}$) by asking that
$v_{\new} \otimes_{\mathcal O_L} (\Sym^{k-2} \mathcal O_L^2)^{*}$ be contained
in 
$\pi(T)$, but not in $\unif\pi(T)$. (Here $\unif$ is a uniformizer of $\mathcal O_L$.)
Further, if we apply $\imath^-_0$ to $v_{\new}\otimes v_{\hw}$ 
(where $v_{\hw}$ is a basis element
for the highest weight space of $(\Sym^{k-2} \mathcal O_L^2)^{*}$),
we obtain an element of
$\mathbf D_{\dR}(V)/ \mathbf D^+_{\dR}(V),$
well-determined up to multiplication by an element of $\mathcal O_L^{\times}$;
in other words, we obtain an $\mathcal O_L$-structure on the one-dimensional
$L$-vector space
$\mathbf D_{\dR}(V)/ \mathbf D^+_{\dR}(V)$. Setting ${\bf e}$ equal $\imath^-_0(v_{\new}\otimes v_{\hw})$ then pins down the elements $d_{n,j}$ (and thus the $c_{n,j}$) up to scaling by a single element in $\mathcal O_L^\times$.

Ideally, under this normalization,  we would like to say that the $c_{n,j}$, which live in $H^1(\Knp,V(1+j))$, are actually in $H^1(\Knp,T(1+j))$.  Unfortunately, $H^1(\Knp,T(1+j))$ need not even be a subspace of $H^1(\Knp,V(1+j))$ as it may contain $\O_L$-torsion elements.  However, the next best integrality property of the $c_{n,j}$ hold; namely, the $c_{n,j}$ are in the image of $H^1(\Knp,T(1+j))$.  Before proving this, we give some notation and a lemma.

Write $\overline{T}(1+j) = T(1+j)/\varpi T(1+j)$ and $\overline{T}^*(-j) = T^*(-j)/\varpi T^*(-j)$. 
Set $$
W_1 := \im \left(  H^1(\Knp,T(1+j))_{\O_L-\tor} \lra H^1(\Knp,\overline{T}(1+j)) \right)
$$
and
$$
W_2 := \im \left(  H^1(\Knp,T^*(-j)) \lra H^1(\Knp,\overline{T}^*(-j)) \right).
$$

\begin{lemma}
\label{lemma:ann}
We have $W_1$ and $W_2$ are exact annihilators under Tate local duality.
\end{lemma}

\begin{proof}
Let $\overline{t}$ be an element of $W_1$ and let $t$ be some lift to $H^1(\Knp,T(1+j))_{\O_L-\tor}$.  Likewise, let $\overline{z}$ be an element of $W_2$ and let $z$ be some lift to $H^1(\Knp,T^*(-j))$.  Since $\langle t,z \rangle_n$ can be computed after inverting $p$ and $t$ is killed by inverting $p$, we deduce that $\langle t,z \rangle_n=0$.
 Thus 
$\langle \overline{t}, \overline{z} \rangle_n = 0$ and $W_1 \subseteq W_2^{\perp}$.

It suffices then to compare the dimensions of $W_1$ and $W_2^\perp$.  
To this end, we have the exact sequence:
$$
0 \to \frac{H^1(\Knp,T^*(-j))}{ \varpi H^1(\Knp,T^*(-j))} \to H^1(\Knp,\overline{T}^*(-j)) \to H^2(\Knp,T^*(-j))[\varpi] \to 0.
$$
The first term here is isomorphic to $W_2$.  For the last term, Tate local duality gives
$$
H^2(\Knp,T^*(-j))[\varpi] \cong 
H^0(\Knp,T(1+j))^\vee[\varpi].
$$
Further, since 
$$
H^1(\Knp,T(1+j))_{\O_L-\tor} \cong H^0(\Knp,T(1+j)) \otimes L/\O_L
$$
we have
$$
W_1 \cong \frac{H^0(\Knp,T(1+j))}{\varpi H^0(\Knp,T(1+j))}.
$$
In particular, $\dim W_1 + \dim W_2 = \dim H^1(\Knp,\overline{T}^*(-j))$ and thus $W_1 = W_2^\perp$.
\end{proof}

\begin{prop}
The classes $c_{n,j}$ are in the image of the natural map $$\alpha : H^1(\Knp,T(1+j)) \to H^1(\Knp,V(1+j)).$$
\end{prop}

\begin{proof}
The image of $H^1(\Knp,T(1+j))$ in $H^1(\Knp,V(1+j))$ is a lattice and thus $\varpi^r c_{n,j}$ is in the image of $H^1(\Knp,T(1+j))$ for some $r$.  Take the smallest such $r$ making this true and assume that $r \geq 1$.   Choose $h \in H^1(\Knp,T(1+j))$ with $\alpha(h) = \varpi^r c_{n,j}$.

Our normalization of the $c_{n,j}$ implies the following key property:\ if $z$ is in the image of $H^1(\Knp,T^*(-j)) \to H^1(\Knp,V^*(-j))$, then $\langle c_{n,j}, z \rangle_n \in \O_L$; that is, the Tate local duality pairing between these elements is integral.

In particular, we have that $\langle h, z \rangle_n  = \langle p^r c_{n,j}, z \rangle_n \equiv 0 \pmod{p}$ for any $z \in H^1(\Knp,T^*(-j))$ as $r \geq 1$.  Hence, $\overline{h} \in H^1(\Knp,\overline{T}(1+j))$ annihilates $W_2$ and thus, by Lemma \ref{lemma:ann}, $\overline{h} \in W_1$.  We can then write $h = t + \varpi h'$ with $h', t \in H^1(\Knp,T(1+j))$ and $t$ an $\O_L$-torsion element.

Finally, $\varpi^ r c_{n,j} = \alpha(h) = \alpha(t+\varpi h') = \varpi \alpha(h')$ and $\varpi^{r-1} c_{n,j}$ is in the image of $\alpha$, contradicting the minimality of $r$.
\end{proof}

\part{The global theory}

\section{Algebraic $\theta$-elements}
\label{sec:algtheta}

\subsection{Notations and assumptions}
Let $f$ be a newform in $S_k(\Gamma_1(N),\psi,\Qpbar)$ for $k \geq 2$, and let $L/\Qp$ denote the field generated by the Fourier coefficients of $f$.  Let $\rho_f : G_\Q \to \Aut(V_f)$ denote the associated (cohomological) Galois representation; that is, the determinant of $\rho_f$ is $\psi^{-1} \varepsilon^{1-k}$ where $\varepsilon$ is the cyclotomic character.  Let $\Sigma$ denote the finite set of primes where $\rho_f$ ramifies.

As $f$ will be fixed for the remainder of the paper, we simply write $V := V_f$ and $\overline{V} := V_{\overline{f}}$ where $\overline{f}$ is the complex conjugate of $f$.  Fix $T$ a Galois stable lattice in $V$.  Set $A = V/T$ and $A^* = V^*/ T^*$.   We similarly write $\overline{T}$, $\overline{A}$, and $\overline{A}^*$ for the analogous constructions using $\overline{V}$.  Note that $A^* \cong \overline{A}(k-1)$. 

Further, set $\Kn = \Qmpn$, $\Kinf = \Q(\mu_{p^\infty})$,
$\Gamma_n = \Gal(\Kinf/\Kn)$, $\G_n = \Gal(\Kn/\Q)$, $\Lambda = \O_L[[\Gal(\Kinf/\Q)]]$, and  $\Ln = \Lambda_{\Gamma_n} \cong \O_L[\Gal(\Kn/\Q)]$.   We fix an isomorphism of $(\Z/p^n\Z)^\times$ with $\G_n$ by mapping $a$ to $\sigma_a$ where $\sigma_a(\zeta) = \zeta^a$ for $\zeta \in \mu_{p^n}$.

 For $n \geq m$, let $\pi^{n}_m : \Ln \to \Lambda_m$ be the natural projection, and let $\nu^{n}_m :\Lambda_m \to \Ln$ be the natural injection which sends a group like element $\sigma \in \G_m$ to $\sum \tau$ where $\tau$ runs over all elements in $\G_{n}$ whose image in $\G_m$ equals $\sigma$.

\subsection{Selmer groups}
\label{sec:selmer}

In this section, we give a construction of algebraic $\theta$-elements attached to $f$ along the $\Zp^\times$-extension $\Kinf/\Q$.  In what follows $j$ will always be an integer such that $0 \leq j \leq k-2$.  We begin by recalling the definition of the relevant Selmer groups.  We have
$$
H^1_f(\Kn,\Adualj) := \ker\left( H^1(\Kn,\Adualj) \to \prod_{v \nmid p} H^1(I_v,\Adualj) \times \frac{H^1(\Knp,\Adualj) }{ H^1_f(\Knp,\Adualj)}
\right)
$$
where $v$ runs over places of $\Kn$ and $I_v$ denotes an inertia group at $v$, and
$$
H^1_f(\Kinf,\Adualj) := \ker\left( H^1(\Kinf,\Adualj) \to \prod_{w \nmid p} H^1(I_w,\Adualj) 	
\right)
$$
where $w$ runs over places of $\Kinf$.
We note that no condition at $p$ is imposed in this second Selmer group since $H^1_f(\Kinfp,\Adualj) = H^1(\Kinfp,\Adualj)$ by \cite[Theorem A]{Berger-univnorm}.

The key global input in what follows is the sequence in the following proposition which precisely describes the failure of the control theorem in this non-ordinary situation.

\begin{prop}
\label{prop:control}
Suppose \eqref{irred} and that $\rhobar_f$ is irreducible.  Then there is a natural exact sequence
\begin{equation}
\label{eqn:control}
0 \to H^1_f(\Kn,\Adualj) \to H^1_f(\Kinf,\Adualj)^{\Gamma_n} \to 
\frac{H^1(\Knp,\Adualj)}{H^1_f(\Knp,\Adualj)} \oplus B_n
\end{equation}
where $B_n$ is a finite group with size bounded independent of $n$.
\end{prop}

\begin{proof}
This sequence is derived via the Snake Lemma as in \cite[Chapter 3]{Greenberg-CIME} or as in \cite[Theorem 3.1]{IP}.  We note that the local ingredients needed to prove this are:
\begin{enumerate}
\item $H^1(\Kinfw/\Knv,H^0(\Kinfw,\Adualj))$ is finite with size bounded independent of $n$ for $v \nmid p$.
\item $H^1_f(\Kinfp,\Adualj) = H^1(\Kinfp,\Adualj)$.
\end{enumerate}
The first of these facts can be proven as in \cite[Lemma 3.3]{Greenberg-CIME}.  The second fact follows from \cite[Theorem A]{Berger-univnorm} by Tate local duality and \eqref{irred}.
\end{proof}

Dualizing (\ref{eqn:control}) and applying Tate local duality yields
\begin{equation}
\label{eqn:controldual}
H^1_f(\Knp,\Tj)  \oplus \D{B_n} \to (\D{H^1_f(\Kinf,\Adualj)})_{\Gamma_n} \to \D{H^1_f(\Kn,\Adualj)} \to 0.
\end{equation}

Let $\Q_\Sigma$ denote the maximal extension of $\Q$ unramified outside of $\Sigma$, and let
$$
\H^i(\Tdualj) := \varprojlim_n H^i(\Q_\Sigma/\Kn,\Tdualj).
$$
The following is a deep theorem of Kato.

\begin{thm}[Kato]
\label{thm:Katorank} 
We have:
\begin{enumerate}
\item If $\rhobar_f$ is irreducible, then $\H^1(\Tdualj)$ is a free $\Lambda$-module of rank one.
\item If \eqref{irred} holds, then $\D{H^1_f(\Kinf,\Adualj)}$ is a rank one $\Lambda$-module.
\end{enumerate}
\end{thm}

\begin{proof}
The first part is \cite[Theorem 12.4.3]{Kato}.  
For the second part, we have 
that the $\Lambda$-corank of $H^1(\Q_\Sigma/\Kinf,\Adualj)$ equals the $\Lambda$-rank of $\H^1(\Tdualj)$  by \cite[Proposition 1.3.2]{PR-book}. Further, consider the exact sequence
\begin{equation}
\label{eqn:selmerdef}
0 \to H^1_f(\Kinf,\Adualj) \to H^1(\Q_\Sigma/\Kinf,\Adualj) \to \bigoplus_{w \in \Sigma-\{p\}} H^1(\Kinfw,\Adualj).
\end{equation}
(Here we are again using $H^1_f(\Kinfp,\Adualj) = H^1(\Kinfp,\Adualj)$ as in the proof of Proposition \ref{prop:control}.) 
Since $H^1(\Kinfw,\Adualj)$ is $\Lambda$-cotorsion if $w \nmid p$ (see \cite[Proposition 2]{Greenberg-ordinary}), we have   
$$
\cork_{\Lambda} H^1_f(\Kinf,\Adualj) = \cork_{\Lambda} H^1(\Q_\Sigma/\Kinf,\Adualj) = 1
$$
as desired.
\end{proof}

To ease notation, let $\Xj = \D{H^1_f(\Kinf,\Adualj)}$.  Set $(\Xj)_{\Ltor}$ equal to the $\Lambda$-torsion submodule of $\Xj$, and let $\Zj = \Xj/(\Xj)_{\Ltor}$, a torsion-free $\Lambda$-module.   Let $R_\Lambda(\Zj) := \Hom_{\Lambda}\bigl(\Hom_{\Lambda}(\Zj,\Lambda),\Lambda\bigr)$
denote the reflexive hull of $\Zj$.
Choosing an isomorphism of $\alpha: R_\Lambda(\Zj) \cong \Lambda$ (via Theorem \ref{thm:Katorank}) yields maps for each $n\geq 0$
\begin{equation}
\label{eqn:jn}
(\Xj)_{\Gamma_n} \lra (\Zj)_{\Gamma_n} \lra \Ln.
\end{equation}

Combining (\ref{eqn:controldual}) and (\ref{eqn:jn}), we get a map
$$
\psi_n = \psi_{n,\alpha}: H^1_f(\Knp,\Tj) \to H^1_f(\Knp,\Tj)  \oplus \D{B_n} \lra (\Xj)_{\Gamma_n} \to (\Zj)_{\Gamma_n} \to \Ln.
$$

\begin{lemma}
The diagram 
$$
\begin{CD}
H^1_f(\Knp, \Tj) @>\psi_n>> \Ln\\
@A\res^n_mAA @AA \nu^n_mA\\
H^1_f(\Kmp, \Tj) @>\psi_m>> \Lambda_m
\end{CD}
$$
commutes for $m \leq n$.
\end{lemma}
\begin{proof}
This commutativity can easily be checked.  See \cite[Proposition 6.3]{IP} for a similar computation.
\end{proof}

Recall our local cohomology classes $c_{n,j} \in H^1_f(\Knp,\Tj) / ({\rm torsion})$ defined in sections \ref{sec:cn} and \ref{sec:normal}. For $0 \leq j \leq k-2$, we set
$$
\psi^{\alg}_{n,j}(f) = \psi_n(c_{n,j}) 
$$
and
$$
\algnj = \psi_{n,j}^{\alg}(f) \cdot \chr_\Lambda((\Xj)_{\Ltor})
$$
which is our $n$-th layer algebraic $\theta$-element.  

Further, let $I^{\alg}_{n,j}(f)$ (resp.\ $J^{\alg}_{n,j}(f)$) denote the ideal of $\Ln$ generated by the
elements
$\nu^n_m(\algvj{m})$ (resp.\ $\nu^n_m(\psi^{\alg}_{m,j}(f))$) for $0 \leq m \leq n$. 
Note that $I^{\alg}_{n,j}(f) = J^{\alg}_{n,j}(f) \cdot \chr_\Lambda (\Xj)_{\Ltor}$.  Also, changing our fixed isomorphism $\alpha$ has the effect of scaling all of the $\psi_{n,j}^{\alg}(f)$ by a unit in $\Lambda$.  In particular, the ideals $I^{\alg}_{n,j}(f)$ and $J^{\alg}_{n,j}(f)$ are independent of this choice.

\begin{remark}
This construction of algebraic $\theta$-elements dates back to Perrin-Riou
in her 1990 paper \cite{PR90}.
\end{remark}

\section{Algebraic results on Fitting ideals}

\subsection{Main algebraic results}

The remainder of this section will be occupied with proving the following purely algebraic result.

\begin{thm}
\label{thm:algfit}
\label{thm:mainthm}
Suppose \eqref{irred} holds and $\rhobar_f$ is irreducible.  Then we have
$$\psi_n(c) \cdot \chr_\Lambda (\Xj)_{\Ltor} \in \Fit_{\Lambda_n} \D{H^1_f(\Kn,\Adualj)}$$
for any $c$ in $H^1_f(\Knp,\Tj)$.
\end{thm}

We note that as an immediate corollary we get.

\begin{cor}
\label{cor:algfit}
Suppose \eqref{irred} holds and $\rhobar_f$ is irreducible.  Then we have
$$
I_{n,j}^{\alg}(f) \subseteq \Fit_{\Ln} \D{H^1_f(\Kn,{\Adualj})}. 
$$
\end{cor}

\begin{proof}
This corollary follows immediately from Theorem \ref{thm:algfit} by taking $c = \res^n_m(c_{m,j})$ for $0 \leq m \leq n$.
\end{proof}

\subsection{Fitting ideal lemmas}

\begin{lemma}
\label{lemma:F}
If $A_n, B_n, C_n$ are $\Ln$-modules, and $Y$ is a $\Lambda$-module then
\begin{enumerate}
\item 
\label{item:F1}
If $A_n \to B_n$ is surjective, then $\Fit_{\Ln}(A_n) \subseteq \Fit_{\Ln}(B_n)$.
\item
\label{item:F2}
If $0 \to A_n \to B_n \to C_n \to 0$ is exact, then 
$$\Fit_{\Ln}(B_n) \supseteq \Fit_{\Ln}(A_n) \cdot \Fit_{\Ln}(C_n).$$
\item 
\label{item:F3}
$\Fit_{\Ln}(Y_{\Gamma_n}) = \pi_n(\Fit_{\Lambda}(Y))$ where $\pi_n : \Lambda \to \Ln$ is the natural map.
\end{enumerate}
\end{lemma}

\begin{proof}
See \cite[Appendix:\ 1,9,4]{MW} 
\end{proof}

\begin{lemma}
\label{lemma:F4}
If $Y$ is a finitely generated torsion $\Lambda$-module with no non-zero finite submodules, then $\Fit_{\Lambda}(Y) = \chr_\Lambda(Y)$.  In particular, $\Fit_{\Lambda}(Y)$ is a principal ideal.
\end{lemma}

\begin{proof}
See \cite[Proposition 1.3.4]{Taleb}.
\end{proof}

We thank Cornelius Greither for the following argument.  

\begin{lemma}
\label{lemma:fitlam}
Let $I \subseteq A$ be ideals of $\Lambda$ with $A$ finite-index in $\Lambda$.  Then
$$
I \subseteq \Fit_{\Lambda}(A/I).
$$
\end{lemma}

\begin{proof}
We first note that it suffices to assume that $I$ is a principal ideal.  Indeed, 
assume that we have proven this lemma for all principal ideals, and let $f$ be any element of $I$.  Then we have 
$$
f \Lambda \subseteq \Fit_\Lambda (A/f\Lambda) \subseteq
\Fit_\Lambda (A/I)
$$
where the second inclusion follows from Lemma \ref{lemma:F}.  Since this is true for all $f \in I$, we have $I \subseteq \Fit_\Lambda (A/I)$.

Thus, we will now assume that $I = f \Lambda$.  We then have an exact sequence
$$
0 \to A/f\Lambda \to \Lambda/f \Lambda \to \Lambda / A \to 0.
$$
Since $A$ has finite index in $\Lambda$, for any height one prime $\p$ in $\Lambda$, we have
$$
(A/f\Lambda)_\p \cong (\Lambda/f \Lambda)_\p.
$$
Thus,
$$
\left(\Fit_{\Lambda}  (A/f \Lambda) \right)_\p
=
\Fit_{\Lambda_\p} \left( (A/f \Lambda)_\p \right)
=
\Fit_{\Lambda_\p} \left( (\Lambda/f \Lambda)_\p \right)
=
(f \Lambda)_{\p}.
$$
By Lemma \ref{lemma:F4}, we have $\Fit_{\Lambda} (A/f \Lambda) = g \Lambda$ for some $g$ in $\Lambda$.  Thus,
$$
(g \Lambda)_\p = 
(f \Lambda)_\p 
$$
for all height one primes $\p$.  Hence $f\Lambda=g\Lambda$ and $f \Lambda= \Fit_\Lambda(A/f\Lambda)$ as desired.
\end{proof}

\begin{lemma}
\label{lemma:fitLn}
Let $A$ be some finite-index ideal of $\Lambda$, and let $I_n$ be a $\Ln$-submodule of $A_{\Gamma_n}$.  Then
$$
i_n(I_n) \subseteq \Fit_{\Ln} A_{\Gamma_n} / I_n 
$$
where $i_n$ is the natural map $A_{\Gamma_n} \maps \Ln$. 
\end{lemma}

\begin{proof}
Let $\pi_n$ denote the natural map from $\Lambda \to \Ln$, and let $p_n$ denote the natural map from $A$ to $A_{\Gamma_n}$ so that $i_n \circ p_n = \pi_n|_A$.
The map $p_n$ induces an isomorphism
\begin{equation*}
\frac{A}{p_n^{-1}(I_n)}
\cong
\frac{A_{\Gamma_n}}{I_n},
\end{equation*}
and thus
\begin{align*}
\Fit_{\Ln}\left( \frac{A_{\Gamma_n}}{I_n}\right) 
&= \pi_n \left(\Fit_\Lambda \left( \frac{A}{p_n^{-1}(I_n)}\right) \right)
  &\mbox{by Lemma \ref{lemma:F}.\ref{item:F3}} \\
&\supseteq \pi_n \left( p_n^{-1}(I_n) \right)
  &\mbox{by Lemma \ref{lemma:fitlam}} \\
&= i_n(I_n)
\end{align*}
as desired.
\end{proof}

We can now prove the main algebraic theorem of this section.

\begin{proof}[Proof of Theorem \ref{thm:mainthm}]
Consider the sequence
$$
0 \to (\Xj)_{\Ltor} \to \Xj \to \Zj \to 0.
$$
Taking $\Gamma_n$-coinvariants yields
$$
0  \to ((\Xj)_{\Ltor})_{\Gamma_n} \to (\Xj)_{\Gamma_n} \to (\Zj)_{\Gamma_n} \to 0
$$
since the kernel of the first map equals $\Zj^{\Gamma_n}=0$.  Let $E_n$ denote the image of $H^1_f(\Knp,\Tj) \oplus \{0\}$ in $(\Xj)_{\Gamma_n}$ under (\ref{eqn:controldual}), and let $F_n$ denote its image in $(\Zj)_{\Gamma_n}$.   If $i_n$ denotes the natural map $(\Zj)_{\Gamma_n} \to \Ln$, we recall that by definition $i_n(F_n)$ equals the image of $\psi_n$.

We  have
$$
0  \to \frac{((\Xj)_{\Ltor})_{\Gamma_n}}{((\Xj)_{\Ltor})_{\Gamma_n} \cap E_n} \to \frac{(\Xj)_{\Gamma_n}}{E_n} \to \frac{(\Zj)_{\Gamma_n}}{F_n} \to 0
$$
Note that (\ref{eqn:controldual}) equates a quotient of the middle term of this sequence with $\D{H^1_f(\Kn,\Adualj)}$.  Thus,
\begin{align*}
\Fit_{\Ln}&(\D{H^1_f(\Kn,\Adualj)})\\
&\supseteq \Fit_{\Ln} \frac{(\Xj)_{\Gamma_n}}{E_n} 
  &\mbox{by Lemma \ref{lemma:F}.\ref{item:F1}} \\
&\supseteq \Fit_{\Ln}\left(\frac{((\Xj)_{\Ltor})_{\Gamma_n}}{((\Xj)_{\Ltor})_{\Gamma_n} \cap E_n}\right) 
\cdot \Fit_{\Ln}\left(\frac{(\Zj)_{\Gamma_n}}{F_n}\right) 
  &\mbox{by Lemma \ref{lemma:F}.\ref{item:F2}} \\
&\supseteq \Fit_{\Ln}\left(((\Xj)_{\Ltor})_{\Gamma_n}\right) 
\cdot \Fit_{\Ln}\left(\frac{(\Zj)_{\Gamma_n}}{F_n}\right) 
  &\mbox{by Lemma \ref{lemma:F}.\ref{item:F1}} \\
&= \pi_n(\Fit_{\Lambda}\left((\Xj)_{\Ltor}\right) )
\cdot \Fit_{\Ln}\left(\frac{(\Zj)_{\Gamma_n}}{F_n}\right) 
  &\mbox{by Lemma \ref{lemma:F}.\ref{item:F3}} \\
&= \pi_n(\chr_{\Lambda}\left((\Xj)_{\Ltor}\right) )
\cdot \Fit_{\Ln}\left(\frac{(\Zj)_{\Gamma_n}}{F_n}\right) 
  &\mbox{by Lemma \ref{lemma:F4}} \\
&\supseteq \pi_n(\chr_{\Lambda}\left((\Xj)_{\Ltor}\right) )
\cdot i_n(F_n) 
  &\mbox{by Lemma \ref{lemma:fitLn}} \\
&\supseteq \pi_n(\chr_{\Lambda}\left((\Xj)_{\Ltor}\right) )
\cdot \psi_n(H^1_f(\Knp,\Tj) )
\end{align*}
as desired.

We note that in applying Lemma  \ref{lemma:F4}, the fact that $(\Xj)_{\Ltor}$ has no non-zero finite submodules is a theorem of Greenberg (see \cite[Proposition 4.1.1]{Greenberg-structure}).
\end{proof}

\section{Mazur-Tate elements and Kato's Euler system}

\subsection{Mazur-Tate elements and statement of main results}

Fix an isomorphism of $i: \C \to \Qpbar$ so that we can view $f$ as defined over either field.  Following \cite{MTT}, we define
$$
\lambda(f,z^j,a,m) := 2 \pi i \int_\infty^{-a/m} f(z) (mz+a)^j ~\!dz
$$
and 
$$
\lambda^\pm(f,z^j,a,m) := \lambda(f,z^j,a,m) \pm \lambda(f,z^j,-a,m).
$$
By a theorem of Shimura \cite{Shimura}, there exist periods $\Omega_f^\pm \in \C$ such that $\lambda^\pm(f,z^j,a,m)/\Omega_f^\pm$ takes values in $K_f$, the field generated by the Fourier coefficients of $f$.  

The periods $\Omega_f^\pm$ are well-defined up to a scalar in $K_f^\times$.  To further specify these periods and make them into the cohomological periods of \cite[Definition 2.1]{PW}, we insist $\lambda^\pm(f,z^j,a,m)/\Omega_f^\pm$ takes values in $\O_f$, the ring of integers of $K_f$, and for some $a$, $m$, $j$, this value is a $p$-adic unit (with respect to our fixed isomorphism $i$).  

We then define
$$
\varphi(f,z^j,a,m) := \frac{\lambda^+(f,z^j,a,m)}{\Omega_f^+} + \frac{\lambda^-(f,z^j,a,m)}{\Omega_f^-}.
$$
With this notation in hand, we define our Mazur-Tate elements of $f$ over $\Kn$:
$$
\MTnj := \sum_{\sigma_a \in \G_n} \varphi(f,z^j,a,p^n) \sigma_a^{-1}
$$
where $a$ runs over representatives of $\left( \Z / p^n \Z\right)^\times$.
For $\chi$ a primitive character of $\G_n$ we have
\begin{equation}
\label{eqn:interp}
\chi(\MTnj) = \chi(-1) \cdot j! \cdot p^{nj} \cdot \tau(\chi^{-1}) \cdot \frac{L(f,\chi,j+1)}{(-2\pi i)^{j}\Omega_f^\pm}.
\end{equation}
where the sign $\pm$ equals the sign of $(-1)^{j} \chi(-1)$.

\begin{remark}
\label{rmk:nonstandard}
Our definition here is a bit non-standard as Mazur and Tate's original definition would replace $\sigma_a^{-1}$ above with $\sigma_a$.  That is, if $\iota$ is the involution on $\Lambda_n$ sending $\sigma$ to $\sigma^{-1}$, then in our notation above $\MTnj^\iota$ is the more standard definition of Mazur-Tate elements.  Nonetheless, the methods of this paper naturally show that our $\MTnj$ belong to a Fitting ideal of a Selmer group and so we chose to state our results in this more natural form.
\end{remark}

The Mazur-Tate elements satisfy the relations:
\begin{equation}
\label{eqn:recur}
\pi^{n+1}_n (\MTvj{n+1}) =
\begin{cases}
  a_p(f) \cdot \MTnj - \psi(p) p^{k-2} \cdot \nu^{n}_{n-1} (\MTvj{n-1}) &  n \geq 1, \\
~\\
 (a_p(f) - 1 - \psi(p) p^{k-2}) \cdot \MTvj{0} & n=0,
\end{cases}
\end{equation}
where $\pi^{n+1}_n$ is the natural projection from $\Lambda_{n+1}$ to $\Ln$.  These formulas follow from \cite[(4.2)]{MTT}.  See also \cite[Proposition 2.5]{PW} for the case $j=0$.

We have the following conjecture relating these analytically defined Mazur-Tate elements to Selmer groups.

\begin{conj}(Mazur-Tate \cite{MT})
\label{conj:MT}
For each $n \geq 0$, we have
$$
\MTnj \in \Fit_{\Ln}(\D{H^1_f(\Kn,\Adualj)}).
$$
\end{conj}

We note that Mazur and Tate only formulated this conjecture for elliptic curves, but in that formulation, the conjecture was much more general, covering Selmer groups over all abelian extensions of $\Q$.

\begin{remark}
\label{rmk:duality}
The functional equation for $L$-values gives an equality 
$$
\MTnj = B \cdot \MTnvv{k-2-j}{\overline{f}}^\iota
$$
where $B$ is some constant.  Thus, it may be equally reasonable to conjecture that $\MTnvv{k-2-j}{\overline{f}}^\iota$ belongs to the Fitting ideal in Conjecture \ref{conj:MT}.  Namely, that
$$
\MTnvv{k-2-j}{\overline{f}}^\iota \in \Fit_{\Ln}(\D{H^1_f(\Kn,\Adualj)})
$$  
or equivalently
$$
\MTnvv{j}{f}^\iota \in \Fit_{\Ln}(\D{H^1_f(\Kn,A(1+j))})
$$  
since $\Adualj \cong \overline{A}(k-j-1)$. In this formulation of the conjecture, the Mazur-Tate element which interpolates $L$-values at $s=r$ is related to the Selmer group of $A(r)$ which might be a more familiar formulation of the conjecture to some (with the caveat that the presence of $\iota$ is caused by our non-standard definition of the Mazur-Tate element as in Remark \ref{rmk:nonstandard}).  

However, as the above constant $B$ need not be integral, neither conjecture is stronger nor weaker than the other.  We do note that for $n$ large enough, the constant $B$ is a $p$-adic unit and thus the two conjectures (for $n$ large) are equivalent.
Also, if one could show that these Fitting ideals were stable under the $\iota$-involution (which seems very plausible), this would give a neat explanation of the equivalence of the two conjectures.
 \end{remark}

We set  $I^{\an}_{n,j}(f)$ to be the ideal of $\Ln$ generated by the elements
$\nu^n_m (\MTvj{m})$ for $0 \leq m \leq n$.  The remainder of the paper will be devoted to proving the following divisibility.

\begin{thm}
\label{thm:divisibility}
If \eqref{irred}, \eqref{nopole}, and \eqref{Euler} hold, then there exists a non-zero constant $C \in \O_L$ (independent of $n$ and $j$) such that
$$
C \cdot I^{\an}_{n,j}(f) \subseteq I^{\alg}_{n,j}(f)
$$
for all $n \geq 0$.
\end{thm}

From Corollary \ref{cor:algfit} and Theorem \ref{thm:divisibility}, we get our main  result.

\begin{cor}
If \eqref{irred}, \eqref{nopole}, and \eqref{Euler} hold, then there exists a non-zero constant $C \in \O_L$ (independent of $n$ and $j$) such that
$$
C \cdot I^{\an}_{n,j}(f) \subseteq \Fit_{\Ln}(\D{H^1_f(\Kn,\Adualj)}).
$$
for all $n \geq 0$.  In particular, $C \cdot \MTnj \in \Fit_{\Ln}(\D{H^1_f(\Kn,\Adualj)})$ for all $n \geq 0$.
\end{cor}

One can describe all forms that fail to satisfy \eqref{nopole} and we do so in the following lemma.

\begin{lemma}
\label{lemma:nopole}
If $f$ is a newform in $S_k(\Gamma_1(N),\psi)$ and \eqref{nopole} does not hold, then either:
\begin{enumerate}
\item $p \nmid N$, $j = \frac{k-3}{2}$, and $a_p(f) = p^{\frac{k-1}{2}}(1 + \psi(p))$;
\item $p || N$, $\psi(p)=1$, $j = \frac{k-4}{2}$, and $a_p(f) = p^{\frac{k-2}{2}}$;
\item $p |N$, $\ord_p(\cond(\psi)) = \ord_p(N)$, $j=\frac{k-3}{2}$, and $a_p = p^{\frac{k-1}{2}}$ (where $\cond(\psi)$ is the conductor of the character $\psi$).
\end{enumerate}
\end{lemma}

\begin{proof}
If $p \nmid N$ and \eqref{nopole} fails, then $p^{j+1}$ is a root of $X^2 - a_p(f)X + \psi(p) p^{k-1}$.  Thus $a_p(f) = p^{j+1} + \psi(p) p^{k-j-2}$.  
The generalized Ramanujan bounds imply that $j = \frac{k-3}{2}$ and hence $a_p(f) = p^{\frac{k-1}{2}}(1 + \psi(p))$ as claimed.

If $p | N$, then by \cite[Theorem 3]{Li}, $a_p(f)$ is non-zero in only two situations.  First, if $p||N$ and $\psi(p) \neq 0$ in which case $a_p^2 = \psi(p) p^{k-2}$.  Second, if $\ord_p(\cond(\psi)) = \ord_p(N)$ in which case $|a_p(f)| = p^{\frac{k-1}{2}}$.  In either case, \eqref{nopole} fails if $a_p = p^{j+1}$.  In the first case, this happens if and only if $j=\frac{k-4}{2}$, $\psi(p) = 1$, and $a_p = p^{\frac{k-2}{2}}$.  In the second case, this happens if only if $j = \frac{k-3}{2}$ and $a_p = p^{\frac{k-1}{2}}$.
\end{proof}

It is easy to write down examples of the first two cases.  For the first case, take a form with $a_p(f)=0$ and $\psi(p)=-1$.  For the second case, any form of level $\Gamma_0(N)$ with $p||N$, has $a_p(f) = \pm p^{\frac{k-2}{2}}$; the forms where $a_p$ is positive contradict \eqref{nopole}.  We know of no examples where the third case occurs.

\subsection{Kato's Euler system}

In this section, we describe how to recover the Mazur-Tate elements from Kato's Euler system and the local classes $c_{n,j}$ constructed in section \ref{sec:cn}.  

Let $z_{\Kato} = (z_{\Kato,n})_n \in \H^1(\Tdualj) = \H^1(\overline{T}(k-j-1))$ denote Kato's Euler system where we recall that $\overline{T}$ is a lattice in the Galois representation attached to $\overline{f}$, the complex conjugate of $f$.
We now state Kato's theorem relating $z_{\Kato}$ to $L$-values.  To this end, 
let $\e$ continue to denote our fixed basis from section \ref{subsec:LL} of $\mathbf D_{\dR}(V(1))/\mathbf D^+_{\dR}(V(1))$, and let $\omega_f$ denote the element of $\mathbf D_{\dR}^+(V^{*})$ canonically identified with $\overline{f}$, via Eichler--Shimura theory and the comparison theorems.

\begin{thm}[Kato]
\label{thm:katorecip}
Let $\chi$ denote a primitive character of $\G_n$.  Then 
$$
\sum_{\sigma \in \G_n}
\chi(\sigma)  \exp^*(\res_p(z_{\Kato,n})^\sigma)) 
= \dfrac{L_{\{p\}}(f,\chi,j+1)}{(-2\pi i)^{j} \Omega_{f,\Kato}^{\pm}} \cdot t^{j} \omega_f
$$
where the sign $\pm$ equals the sign of $(-1)^j \chi(-1)$, $L_{\{p\}}$ denotes the $L$-value with the Euler factor at $p$ removed, and $\Omega_{f,\Kato}^{\pm} \in \C$.
\end{thm}

\begin{proof}
See \cite[Theorem 12.5]{Kato}.
\end{proof}

\begin{remark}
The periods $\Omega_{f,\Kato}^{\pm}$ appearing in the above theorem differ from our cohomological periods $\Omega_f^\pm$ by an algebraic number.  See \cite[section 4.3]{ChanHo} for more details.
\end{remark}

The following proposition shows how one can recover Mazur-Tate elements from Kato's Euler system and the local classes $c_{n,j}$.

\begin{prop}
\label{prop:katoMT}
If \eqref{nopole} holds, then there exists a non-zero constant $D \in L$ (independent of $n$ and $j$) such that
$$
D \cdot \MTnj = \sum_{\sigma \in \G_n} \left\langle  c_{n,j}^\sigma, \res_p(z_{\Kato,n}) 
\right\rangle_{n} \sigma^{-1}.
$$
\end{prop}

\begin{proof}
Set 
$$
\tnj = \sum_{\sigma \in \G_n} \left\langle  c_{n,j}^\sigma , \res_p(z_{\Kato,n}) 
\right\rangle_{n} \sigma^{-1}.
$$
To prove this proposition it suffices to check that the $\tnj$
satisfy the same interpolation property (up to a constant) as the Mazur-Tate elements ({\it i.e.}\ (\ref{eqn:interp})) for primitive characters, and that they satisfy the same three-term relation ({\it i.e.}\  (\ref{eqn:recur})).

To this end, for $n>0$, we have
\begin{align*}
\chi(\tnj) &= \sum_{\sigma \in \G_n}
  \chi(\sigma)^{-1} \left\langle  c_{n,j}^\sigma, \res_p(z_{\Kato,n})
\right\rangle_{n}  \\
&=
\chi(-1) \cdot j! \cdot p^{mj} \cdot \tau(\chi^{-1})  \cdot 
\big\langle t^{-j} \e,
\sum_{\sigma \in \G_n} \chi(\sigma)  \exp^*(\res_p(z_{\Kato,n})^\sigma) 
\big\rangle_{\dR,n} &\quad[\text{Proposition~} \ref{prop:recip}] \\
&=
\chi(-1) \cdot j! \cdot p^{mj} \cdot \tau(\chi^{-1})   \cdot \dfrac{L_{\{p\}}(f,\chi,j+1)}{(-2 \pi i)^{j} \Omega_{f,\Kato}^{\pm}}\langle \mathbf e, \omega_f\rangle_{\dR} &\quad[\text{Proposition~} \ref{thm:katorecip}] \\
&= \chi(\MTnj) \cdot \langle \mathbf e, \omega_f\rangle_{\dR} \cdot \frac{\Omega_{f}^{\pm}}{\Omega_{f,\Kato}^{\pm}}
\end{align*}
since the local Euler factor at $p$ is trivial for $L(f,\chi,s)$.
Note that the application of Proposition \ref{prop:recip} requires the hypothesis \eqref{nopole}.

For $n=0$, we have
\begin{align*}
{\bf 1}(\theta_{0,j}) &=\left\langle c_{0,j}, \res_p(z_{\Kato,0}) 
\right\rangle_{0}  \\
&=
j! \cdot \widetilde{Z}(x^{-j} \phi_{v_{\new}},1) \left\langle t^{-j} \mathbf e, \exp^* z_0'\right\rangle_{\dR,0} 
 &\quad[\text{Proposition~} \ref{prop:recip}] \\
&=
j! \cdot \widetilde{Z}(x^{-j} \phi_{v_{\new}},1) 
 \cdot \dfrac{L_{\{p\}}(f,j+1)}{(-2\pi i)^j \Omega_{f,\Kato}^{\pm}}\langle \mathbf e, \omega_f\rangle_{\dR} &\quad[\text{Proposition~} \ref{thm:katorecip}] \\
&= 
j! \cdot \dfrac{L(f,j+1)}{(-2\pi i)^j \Omega_{f,\Kato}^{\pm}}\langle \mathbf e, \omega_f\rangle_{\dR} &\\
&= {\bf 1}(\MTvj{0}) \cdot\left\langle \mathbf e, \omega_f\right\rangle_{\dR} \cdot \frac{\Omega_{f}^{\pm}}{\Omega_{f,\Kato}^{\pm}}.
\end{align*}
Here $\phi_{v_{\new}}$ is the newvector associated to $f$ and $\widetilde{Z}(x^{-j} \phi_{v_{\new}},1) $ is simply the local Euler factor at $p$ for $L(f,j+1)$.  Thus, $\tnj$ and $\MTnj$ agree at primitive characters (up to a constant).

To finish the proof, we compute
\begin{align*}
&\pi^{n+1}_n(\theta_{n+1,j})
= \pi^{n+1}_n\left(\sum_{\tau \in \G_{n+1}} \left\langle  c_{n+1,j}^\tau , \res_p(z_{\Kato,n+1})  \right\rangle_{n+1} \tau^{-1} \right)\\
&= \sum_{\sigma \in \G_{n}} \left( \sum_{\stackrel{\tau \in \G_{n+1}}{\tau \to \sigma}}\left\langle c_{n+1,j}^\tau  ,\res_p(z_{\Kato,n+1})  \right\rangle_{n+1} \right) \sigma^{-1} \\
&= \sum_{\sigma \in \G_{n}}   \left\langle  \cores^{n+1}_n(c_{n+1,j})^{\sigma} , \cores^{n+1}_n(\res_p(z_{\Kato,n+1}))  \right\rangle_n \sigma^{-1} \\
&= \sum_{\sigma \in \G_{n}}   \left\langle  a_p(f)  c_{n,j}^\sigma - \psi(p) p^{k-2} \res^n_{n-1} ( c_{n-1,j}^\sigma) ,\res_p(z_{\Kato,n}) \right\rangle_n \sigma^{-1} \\
&= a_p(f) \cdot \sum_{\sigma \in \G_{n}}  \left\langle  c_{n,j}^\sigma, \res_p(z_{\Kato,n})  \right\rangle_n \sigma^{-1} - \psi(p) p^{k-2} \sum_{\sigma \in \G_{n}}  \left\langle  \res^n_{n-1} ( c_{n-1,j}^\sigma), \res_p(z_{\Kato,n})   \right\rangle_n \sigma^{-1} \\
&= a_p(f) \cdot \sum_{\sigma \in \G_{n}}  \left\langle  c_{n,j}^\sigma, \res_p(z_{\Kato,n})   \right\rangle_n \sigma^{-1} - \psi(p) p^{k-2} \sum_{\sigma \in \G_{n}}  \left\langle  c_{n-1,j}^\sigma, \res_p(z_{\Kato,n-1})    \right\rangle_{n-1} \sigma^{-1} \\
&= a_p(f) \cdot \tnj - \psi(p) p^{k-2} \cdot \nu_{n-1}^n( \theta_{n-1,j})
\end{align*}
as desired.  Here we used Proposition \ref{prop:three_term}.  A similar computation handles the case $n=0$ as well.
\end{proof}

\section{Algebraic $\theta$-elements, take II}

Throughout this section we suppose that \eqref{irred} holds and that $\rhobar_f$ is irreducible.

\subsection{An alternative construction of $\theta$-elements}

In section \ref{sec:algtheta}, we gave a construction of algebraic $\theta$-elements working with the Galois cohomology of the discrete module $\Adualj$ which allowed us to relate these elements to the Fitting ideals of its Selmer group.  We now give a simpler and more direct construction of these elements using the cohomology of $\Tdualj$.  The advantage of this second construction is that it will be easier to relate to  Mazur-Tate elements (via Kato's results) and thus to prove Theorem \ref{thm:divisibility}.  However, showing that these two constructions match then takes additional work, and will be carried out in this subsection.

By Theorem \ref{thm:Katorank}, let $w = (w_n) \in \H^1(\Tdualj)$ denote some generator of 
this free $\Lambda$-module of rank 1.  For $c \in H^1_f(\Knp,\Tj)$, set
$$
\widetilde{\psi}_n(c) := \widetilde{\psi}_{n,w}(c):= \sum_{\sigma \in \G_n} \left\langle  c^\sigma, \res_p(w_n)  \right\rangle_{n} \sigma   \in \Ln
$$
where $\res_p$ denotes restriction to the unique place over $p$.  The pairing $\langle \cdot, \cdot \rangle_{n}$ is given by Tate local duality
$$
H^1(\Knp,\Tj) \times H^1(\Knp,\Tdualj) \to \O_L.
$$

\subsection{Comparing the two constructions}
\label{sec:compare}

Let $\iota$ be the involution of $\Lambda_n$ which sends a group-like element $\sigma$ to $\sigma^{-1}$. We will see that the above simple and direct definition of $\widetilde{\psi}_n$ yields a function which, up to $\iota$, equals $\psi_n$  (which was the key input in defining algebraic $\theta$-elements).  However, to make this claim precise, recall that the definition of $\psi_n = \psi_{n,\alpha}$ depended on an isomorphism $\alpha :R_\Lambda(\Xj) \cong \Lambda$ while the definition of $\widetilde{\psi}_n = \widetilde{\psi}_{n,w}$ depended on a choice of a generator $w$ of $\H^1(\Tdualj)$.  Thus, to relate $\psi_n$ to $\widetilde{\psi}_n$, we must first relate these choices.

We now show how to use our fixed $w = (w_n) \in \H^1(\Tdualj)$ to give  an explicit identification of $R_\Lambda(\Xj)$ with $\Lambda$.
Set 
$\Hj = H^1(\Q_\Sigma/\Kinf,\Adualj)$ and $\Xj^\Sigma := \D{\Hj}$.
Dualizing (\ref{eqn:selmerdef}) yields
$$
W \to \Xj^\Sigma \to \Xj \to 0
$$
where $W$ is some finitely generated torsion $\Lambda$-module.  Thus, applying
$\Hom_\Lambda(\cdot,\Lambda)$ gives a canonical isomorphism
$$
\Hom_\Lambda(\Xj,\Lambda)
\stackrel{\sim}{\lra} \Hom_\Lambda(\Xj^\Sigma,\Lambda)
$$
and hence a canonical isomorphism $R_\Lambda(\Xj) \cong R_\Lambda(\Xj^\Sigma)$.  
It therefore suffices to make an identification of $R_\Lambda(\Xj^\Sigma)$ with $\Lambda$ which we now do.

From \eqref{irred}, we deduce that $H^0(\Q(\mu_p),\rhobar_f)=0$, and thus the control theorem holds for $\Hj$; that is, the restriction map
$$r_n : H^1(\Q_\Sigma/\Kn,\Adualj) \to \Hj^{\Gamma_n}
$$
is an isomorphism for any $n$.
Let
$$\pi_n^m : H^1(\Q_\Sigma/\Kn,\Tdualj) \to H^1(\Q_\Sigma/\Kn,\Tdualj/p^m\Tdualj) \cong
H^1(\Q_\Sigma/\Kn,\Adualj[p^m]) \overset{r_{n}}{\to} \Hj^{\Gamma_n}[p^m]$$
denote the canonical map.
Define a $\Lambda$-homomorphism
$$\xi_w : \Xj^\Sigma \to \Lambda = \underset{n}{\varprojlim} \, \Ln$$
given, for $\eta : \Hj \to \Qp/\Zp$, by
$\xi_w(\eta) = \left( \sum_{\sigma \in \G_n}  a_\sigma \sigma \right)_n$
where
$$a_\sigma = 
\varprojlim_{m \geq n}  p^m \eta\bigr(\pi_n^m(w_n^{\sigma})\bigl) \in \Zp.$$
In the above equation, $\eta(\pi_n^m(w_n^{\sigma}))$ lies in
$\frac{1}{p^m}\Z_p/\Z_p$ and we view $p^m \eta\bigr(\pi_n^m(w_n^{\sigma}))$ in $\Z/p^m\Z$ via the isomorphism $\frac{1}{p^m}\Z_p/\Z_p \cong \Z/p^m\Z$ given by multiplication by $p^m$.

\begin{lemma}
$\Hom_{\Lambda}(\Xj^\Sigma,\Lambda)$ is free of rank one over $\Lambda$ with
generator $\xi_w$.
\end{lemma}
\begin{proof}
We have the following $\Lambda$-module homomorphisms:
\begin{align*}
\Hom_{\Lambda}(\Xj^\Sigma,\Lambda)
&\cong \Hom_{\Lambda}(\D{\Lambda},\D{(\Xj^\Sigma)}) 
\cong \Hom_{\Lambda}(\D{\Lambda},\Hj) 
\cong \Hom_{\Lambda}\bigl(\underset{n}{\varinjlim} \,
(\D{\Lambda})^{\Gamma_n}[p^n],\Hj \bigr) \\
&\cong \underset{n}{\varprojlim} \, \Hom_{\Lambda} \bigl(
(\D{\Lambda})^{\Gamma_n}[p^n],\Hj \bigr) 
\cong \underset{n}{\varprojlim} \, \Hom_{\Lambda} \bigl(
\D{(\Lambda_n/p^n)},\Hj \bigr) \\
&\cong \underset{n}{\varprojlim} \, \Hom_{\Lambda} (
\Lambda_n/p^n,\Hj) 
\cong \underset{n}{\varprojlim} \,\, \Hj[p^n]^{\Gamma_n} 
\cong \H^1(\Tdualj).
\end{align*}
Above we used the self-duality
\begin{align*}
\Lambda_n/p^n \times \Lambda_n/p^n &\to \Z/p^n \\
\left< \sum a_\sigma \sigma, \sum b_\sigma \sigma \right> &\mapsto
\sum a_\sigma b_{\sigma^{-1}}
\end{align*}
as this is the duality which induces a $\Lambda$-module isomorphism:\ $\D{(\Lambda_n/p^n)} \cong \Lambda_n/p^n$.
 To prove this lemma, one then chases through the above maps to see that $\xi_w$ in $\Hom_{\Lambda}(\Xj^\Sigma,\Lambda)$ maps to the generator $w$ in $\H^1(\Tdualj)$.
\end{proof}

\begin{cor}
The map
$$\Xi_w : R_\Lambda(\Xj) \stackrel{\sim}{\to} R_\Lambda(\Xj^\Sigma) \to \Lambda$$
given by evaluation at $\xi_w$ 
is an isomorphism.
\end{cor}

We have thus built an explicit identification of $R_\Lambda(\Xj)$ with $\Lambda$ given a generator of $\H^1(\Tdualj)$.  

\begin{thm}
\label{thm:compare}
For $c \in H^1_f(\Knp,\Tj)$,  we have
$$
\psi_{n,\Xi_w}(c) = \widetilde{\psi}_{n,w}(c)^\iota
$$
In particular,
$$
\algnj = \left( \sum_{\sigma \in \G_n} \left\langle  c_{n,j}^\sigma, \res_p(w_n) \right\rangle_{n} \sigma^{-1} \right) \cdot \chr_{\Lambda}((\Xj)_{\Ltor})
$$
\end{thm}

\begin{proof}
Let
$\Xi_{w,n} : X_{\Gamma_{n}} \to \Lambda_{n}$
denote the $\Gamma_n$-coinvariants of the composition of the
canonical map
$\ev : \Xj \to R_{\Lambda}(\Xj)$ and $\Xi_w$.
Recall the isomorphisms:
\begin{equation} \label{soi}
(\Xj^\Sigma)_{\Gamma_n} = (\D{\Hj})_{\Gamma_n} \overset{\sim}{\longrightarrow}
\D{(\Hj^{\Gamma_n})} \overset{\sim}{\longrightarrow}
 \D{H^1(\Q_\Sigma/\Kn,\Adualj)}.
\end{equation}
Thus, we may define an element in $(\Xj^\Sigma)_{\Gamma_n}$ by defining it as the functional on $  H^1(\Q_\Sigma/\Kn,\Adualj)$.  

For $c \in H^1_f(\Knp,\Tj)$,
consider the element $\varphi^\Sigma_c \in (\Xj^\Sigma)_{\Gamma_n}$ defined by first restricting to $p$ and then pairing (via Tate local duality) against $c$.  That is, for $h \in  H^1(\Q_\Sigma/\Kn,\Adualj)$, 
set
$$\varphi^\Sigma_c(h) =  \left< c,\res_p(h) \right>.$$
Set $\varphi_c \in (\Xj)_{\Gamma_n}$ equal to the image of $\varphi^\Sigma_c$ under the natural surjection $\Xj^\Sigma \to \Xj$.  Then, by definition, $\psi_{n,\Xi_w}(c) = \Xi_{w,n}(\varphi_c)$.

However, we can also compute $\Xi_{w,n}(\varphi_c)$ directly as the image of $\xi_w(\varphi^\Sigma_c)$ in $\Ln$.
By the definition of $\xi_w$, this image equals
$\sum_{\sigma \in G_n} b_\sigma \sigma$
where
$$b_{\sigma} = \varprojlim_{m \geq n} p^{m}
\left< c,\res_p(\pi_n^m( w_n^{\sigma})) \right>_n.$$
But, by definition, this limit equals
$$\left< c, \res_p( w_n^{\sigma}) \right>_{n} =
\left<  c^{\sigma^{-1}}, \res_p(w_n) \right>_{n}$$
as desired.
\end{proof}

\section{The main results}
In light of Proposition \ref{prop:katoMT} and Theorem \ref{thm:compare}, to compare $\MTnj$ and $\algnj$ we need to understand the index of Kato's Euler system in $\H^1(\Tdualj)$ compared to $\chr_\Lambda X_{\Ltor}$.  This relation will follow from Kato's proof of half of the main conjecture (without $p$-adic $L$-functions).

\subsection{Euler system result}
\label{sec:euler}
Let $\H^{i}(T) = \varprojlim_n H^i(\Q_\Sigma/\Kn,T)$ and
$\H^i_{\Iw,w}(T) = \varprojlim_n H^i(\Knvn,T)$ where $w$ is some place of $\Qinf$ and $v_n$ is the place of $\Kn$ below $w$.   Further, set
$$
\H^2_P(T) = \ker \left( \H^{2}(T) \to \prod_{w \in \Sigma_\infty} \H^2_{\Iw,w}(T)
\right)
$$
where $\Sigma_\infty$ is the set of places of $\Kinf$ sitting over $\Sigma - \{p\}$.

\begin{thm}[Kato]
\label{thm:katodiv}
We have
\begin{enumerate}
\item 
$\H^2_P(\Tdualj)$ is a torsion $\Lambda$-module;
\item
if \eqref{Euler} holds,
$$
\chr \H^2_P(\Tdualj) \text{~divides~} \chr_\Lambda( \H^1(\Tdualj) / \langle z_{\Kato} \rangle ).
$$
\end{enumerate}
\end{thm}

\begin{proof}
This is \cite[Theorems 12.4 and 12.5]{Kato}.  See also \cite[pg. 217]{Kurihara} where the \'etale cohomology groups in \cite{Kato} are written in terms of Galois cohomology with local conditions ({\it e.g.}\ $\H^2_P$).
\end{proof}

We must now relate $\H^2_P(\Tdualj)$ to $(\Xj)_{\Ltor}$.  If $M$ is a $\Lambda$-module, we write $M^{\iota}$ for the $\Lambda$-module whose underlying set is $M$, but the group-like element $\sigma$ acts by $\sigma^{-1}$.

\begin{thm}
\label{thm:nekduality}
Assume \eqref{irred} holds.  If $X_{\Ltor}$ denotes the $\Lambda$-torsion submodule of $X$, we have
$$
\chr_\Lambda (\Xj)_{\Ltor} = \chr_\Lambda(\H^2_P(\Tdualj)^\iota).
$$
\end{thm}

\begin{proof}
By the two exact sequences in \cite{Nekovar} appearing right after (0.13.2), we have
a pseudo-isomorphism
$$
(\D{\widetilde{H}^1_{f,\Iw}(\Adualj)})_{\Ltor}
\cong
\Ext^1_{\Lambda}(\widetilde{H}^2_{f,\Iw}(\Qinf/\Q,\Tdualj),\Lambda)
$$
Here $\widetilde{H}^i_{f,\Iw}$ are Selmer complexes where, for $\Adualj$, there is no local condition at $p$ (and thus for $\Tdualj$ we impose the harshest local condition).  Relating these back to classical Selmer groups (as in \cite[(0.10)]{Nekovar}), we have a pseudo-isomorphism
$$
(X_r)_{\Ltor} \cong 
\Ext^1_{\Lambda}(\H^2_P(\Tdualj),\Lambda).
$$
By Theorem \ref{thm:katodiv} and \cite[top of pg.\ 474]{Wingberg}, we then have a pseudo-isomorphism
$$
(X_r)_{\Ltor} \cong \H^2_P(\Tdualj)^\iota$$
as desired.
\end{proof}

\subsection{Divisibility}

We begin with a simple lemma whose proof we leave to the reader.

\begin{lemma}
\label{lemma:linear}
For $\lambda \in \Lambda$, we have
$$
\lambda \cdot \sum_{\sigma \in \G_n} \left\langle  c_{n,j}^\sigma, \res_p(w_n) 
\right\rangle_n \sigma 
=
\sum_{\sigma \in \G_n} \left\langle  c_{n,j}^\sigma, \res_p(\lambda w_n) 
\right\rangle_n \sigma .
$$
\end{lemma}

We conclude with our main theorem which immediately implies Theorem \ref{thm:divisibility}.

\begin{thm}
\label{thm:div_elt}
If \eqref{irred}, \eqref{nopole}, and \eqref{Euler} hold, then there exists a non-zero constant $C \in \O_L$ (independent of $n$ and $j$) such that
$$
\algnj ~|~ C \cdot \MTnj
$$
in $\Ln$ for all $n \geq 0$.
\end{thm}

\begin{proof}
Recall, $w = (w_n)$ is our fixed generator of $\H^1(\Tdualj)$.   By Theorem \ref{thm:compare}, we have
$$
\algnj = \chr_\Lambda \left((\Xj)_{\Ltor}\right) \cdot \sum_{\sigma \in \G_n} \left\langle  c_{n,j}^\sigma, \res_p(w_n) 
\right\rangle_n \sigma^{-1}
$$
Applying the involution $\iota: \Ln \to \Ln$ which sends group-like elements to their inverse, yields 
\begin{align}
\label{eqn:div}
\iota(\algnj) &= \chr_\Lambda \H_P^2(\Tdualj) \cdot \sum_{\sigma \in \G_n} \left\langle  c_{n,j}^\sigma, \res_p(w_n)  
\right\rangle_n \sigma
\end{align}
by Theorem \ref{thm:nekduality}.  
Now write $z_{\Kato} = \alpha \cdot w$ for some $\alpha \in \Lambda$.  Thus, $\alpha = \chr_\Lambda ( \H^1(\Tdualj) / \langle z_K \rangle)$, and $\chr_\Lambda \H^2_P(\Tdualj)$ divides $\alpha$ by Theorem \ref{thm:katodiv}.  

Scaling (\ref{eqn:div}) by $\alpha$, applying Lemma \ref{lemma:linear} and then Proposition \ref{prop:katoMT} gives
\begin{align*}
\alpha \cdot \iota(\algnj) 
&= \chr_\Lambda \H_P^2(\Tdualj) \cdot \sum_{\sigma \in \G_n} \left\langle  c_{n,j}^\sigma,\res_p(z_{\Kato,n})  
\right\rangle_n \sigma \\
&= \chr_\Lambda \H_P^2(\Tdualj) \cdot D \cdot \iota(\MTnj)
\end{align*}
for $D \in L$ independent of $n$ and $j$. Thus
$$
\iota(\algnj) \cdot \frac{\alpha}{\chr_\Lambda \H^2_P(\Tdualj)} = D \cdot \iota(\MTnj)
$$
and our claim is proven with $C=D$ if $D \in \O_L$ and $C=1$ otherwise.  
\end{proof}

\end{document}